\documentclass[10pt,reqno]{article}

\usepackage{booktabs} 
\usepackage{array} 
\usepackage{paralist} 
\usepackage{verbatim} 
\usepackage{subfig} 
\usepackage{mathrsfs}
\usepackage{amssymb}
\usepackage{amsthm}
\usepackage{amsmath,amsfonts,amssymb,esint}
\usepackage{graphics}
\usepackage{enumerate}
\usepackage{mathtools}
\usepackage{xfrac}
\usepackage[margin=1in]{geometry}
\usepackage{amsmath,amsthm,amssymb}
\usepackage{xfrac}
\usepackage{subfiles}
\usepackage{paralist} 
\usepackage{enumerate}

\usepackage[colorlinks=true, pdfstartview=FitV, linkcolor=blue,citecolor=blue, urlcolor=blue]{hyperref}

\usepackage[abbrev,lite,nobysame]{amsrefs}
\usepackage{times}
\usepackage[usenames,dvipsnames]{color}

\usepackage{mathtools,enumitem}

\usepackage[compact]{titlesec}

\usepackage[title]{appendix}

\usepackage{comment}

\newcommand{\grad}{\nabla}

\newcommand{\calc}{\mathcal{C}}
\newcommand{\les}{\lesssim}
\newcommand{\MM}[1]{\ensuremath{\mathcal{M}}\left(#1\right)}

\newcommand{\norm}[1]{\left\| #1 \right\|}

\newcommand{\R}{\mathbb{R}}

\newcommand{\T}{\mathbb{T}}

\DeclareMathOperator{\supp}{\mathrm{supp}}

\DeclareMathOperator{\Div}{\mathrm{div}}
\DeclareMathOperator{\Id}{\mathrm{Id}}

\DeclareMathOperator{\curl}{\mathrm{curl}}

\usepackage{bbm}

\newtheorem{theorem}{Theorem}[section]
\newtheorem{proposition}[theorem]{Proposition}

\newtheorem{lemma}[theorem]{Lemma}
\newtheorem*{lemma*}{Lemma}

\theoremstyle{definition}
\newtheorem{definition}[theorem]{Definition}
\newtheorem{remark}[theorem]{Remark}

\newtheorem{conjecture}{Conjecture}

\newcommand{\dpot}{{\mathsf d}}

\newcommand{\divR}{\mathcal{R}^*}

\newcommand{\UU}{\mathbb{U}}

\newcommand{\optwo}{\left(u-\alpha^2\Delta u\right)}

\newcommand{\RR}{{\mathring{R}}}

\newcommand{\WW}{\mathbb{W}}
\newcommand{\Ndec}{{\mathbf{N}_{\textnormal{dec}}}}
\newcommand{\dist}{\textnormal{dist}}
\newcommand{\qi}{_{q,i}}
\newcommand{\qiprime}{_{q,i'}}
\newcommand{\rhohalf}{\rho^{\sfrac{1}{2}}}
\newcommand{\Pneq}{\mathbb{P}_{\neq 0}}
\newcommand{\Peq}{\mathbb{P}_{=0}}

\numberwithin{equation}{section}

\usepackage{etoolbox}
\usepackage{hyperref}

\makeatletter
\newcounter{author}
\renewcommand*\author[1]{%
  \stepcounter{author}%
  \ifnum\c@author=1
    \gdef\@author{#1}%
  \else
    \xdef\@author{\unexpanded\expandafter{\@author\and#1}}%
  \fi
  \csgdef{author@\the\c@author}{#1}}
\newcommand*\email[1]{%
  \csgdef{email@\the\c@author}{#1}}
\newcommand*\address[1]{%
  \csgdef{address@\the\c@author}{#1}}
\AtEndDocument{%
  \xdef\author@count{\the\c@author}%
  \c@author=1
  \print@authors}
\newcommand*\print@authors{%
  \ifnum\c@author>\author@count
  \else
    \print@author{\the\c@author}%
    \advance\c@author by 1
    \expandafter\print@authors
  \fi}
\newcommand*\print@author[1]{%
  \par\medskip
  \begin{tabular}{@{}l@{}}%
    \textsc{\csuse{address@#1}}\\
    \textit{E-mail address}:
    \href{mailto:\csuse{email@#1}}{\csuse{email@#1}}
  \end{tabular}}
\makeatother

\author{Rajendra Beekie}
\address{\small Courant Institute of Mathematical Sciences, New York University, 251 Mercer St., New York, NY 10012}
\email{beekie@cims.nyu.edu}

\author{Matthew Novack}
\address{\small {School of Mathematics, Institute for Advanced Study, 1 Einstein Dr., Princeton, NJ 08540, USA}}
\email{mdn@ias.edu}

\title{Non-conservative solutions of the Euler-$\alpha$ equations}

\begin{document}

\maketitle

\begin{abstract}
The Euler-$\alpha$ equations model the averaged motion of an ideal incompressible fluid when filtering over spatial scales smaller than $\alpha$.  We show that there exists $\beta>1$ such that weak solutions to the two and three dimensional Euler-$\alpha$ equations in the class $C^0_t H^\beta_x$ are not unique and may not conserve the Hamiltonian of the system, thus demonstrating flexibility in this regularity class. The construction utilizes a Nash-style intermittent convex integration scheme.  We also formulate an appropriate version of the Onsager conjecture for Euler-$\alpha$, postulating that the threshold between rigidity and flexibility is the regularity class $L^3_t B^{\sfrac{1}{3}}_{3,\infty,x}$.
\end{abstract}



\setcounter{tocdepth}{3}
\tableofcontents

\section{Introduction}
We consider the Lagrangian averaged Euler equations (LAE-$\alpha$ or Euler-$\alpha$) for the unknown velocity field $u$ and pressure $p$:
\begin{subequations}
    \label{eq:euler:alpha1}
    \begin{align}
    \partial_t \optwo + u\cdot \nabla \optwo + \optwo^j \nabla u^j + \nabla p &= 0   \\
    \Div u &= 0 \, .
    \end{align}
\end{subequations}
The parameter $\alpha > 0$ corresponds to a \emph{length scale}, and we pose the system for $(x,t) \in \mathbb{T}^n \times [0,T],  n = 2 \text{ or }3\, , T > 0$. Euler-$\alpha$ was originally derived within an Euler-Poincar\'e variational framework to  model the averaged motion of an ideal incompressible fluid when filtering over spatial scales smaller than $\alpha$ \cite{HolmMarsdenJerroldRatiu98a, HolmMarsdenJerroldRatiu98b}. It was later proven  that, like the Euler equations, \eqref{eq:euler:alpha1} have a deep geometric significance as the geodesic equations with respect to an $H^1$ metric on the diffeomorphism group of volume preserving flows \cite{Arnold66, Shkoller98}. We also note that the Euler-$\alpha$ equations are identical to the inviscid second-grade fluid equations, where $\alpha$ now represents the elastic response of the fluid (see \cite{DunnFosdick74} and references therein for more on second-grade fluids). The viscous variant of \eqref{eq:euler:alpha1} is known as the Lagrangian averaged Navier-Stokes equation (LANS-$\alpha$) or the viscous Camassa-Holm equation and has been successfully used as a turbulence closure model \cite{ChenFoiasHolmOlsonTiti98, FoiasHolmTiti01, MohseniKosovicShkollerMarsden03, ChenHolmMargolinZhang99 }.  

When posed on a compact Riemannian manifold of dimension $n$ without boundary, local-wellposedness in {$H^{s}$ for $s > \frac{n}{2} + 1$} was proven in \cite{Shkoller98} (see also \cite{Shkoller00, MarsdenRatiuShkoller00} for the case of domains with boundary). {The state of the global wellposedness question for Euler-$\alpha$ is analogous to that of the classical Euler equations;} in two dimensions, global wellposedness of smooth solutions is known \cite{busuioc99}, but in three dimensions, finite time singularity formation has not been ruled out. 
It was shown in \cite{HouLi06} that global existence may be phrased in terms of a Beale-Kato-Majda type continuation criterion which requires that
\begin{equation}
    \label{eq:BKM:blowup}
    \int_0^T \| \nabla \times (-\Delta)^{-1} (1 - \alpha^2 \Delta)u(t) \|_{\text{BMO}} dt < \infty \, .
\end{equation}
Note that by setting $\alpha =0$ we formally recover the Euler equations after redefining the pressure. In the case of Euclidean space, the convergence of smooth Euler-$\alpha$ solutions to Euler solutions was demonstrated in \cite{LinshizTiti10} (see also \cite{LFNLTitiZang15, BusuiocIftimieLFNL12, BusuiocIftimie17, BusuiocIftimieLFNL16 } for progress in the case of domains with boundary). 

\subsection{Euler-$\alpha$ as a {Hamiltonian System}}
Sufficiently smooth solutions to Euler-$\alpha$ satisfy the conservation law
\begin{equation}
    \label{eq:conservation}
    \frac{d}{dt} \mathcal{H}_{\alpha} := \frac{d}{dt} \left(\| u \|_{L^2}^2  + \alpha^2 \| \nabla u \|_{L^2}^2 \right) = 0 \, .
\end{equation}
Beyond merely being a conserved quantity, $\mathcal{H}_{\alpha}$ is actually the {\emph{Hamiltonian} associated to the action function from which Euler-$\alpha$ can be derived }. We follow \cite{BSV16} to expand upon this point. In the Euler-Poincar\'e framework, Euler-$\alpha$ arises as the geodesic equations associated with the metric 
\begin{equation}
    \label{eq:metric}
    (u, w)_{\alpha} := \int_{\mathbb{T}^n} (1 -\alpha^2 \Delta) u \cdot w \,  dx\, .
\end{equation}
Moreover, the geodesic equations can be computed  by identifying extrema of the action function
\begin{equation}
    \label{eq:action}
    s(u) = \int_{\T^n}(1 - \alpha^2 \Delta)u \cdot u \,  dx = \int_{\T^n} |u|^2 + \alpha^2 |\nabla u|^2 \, dx
\end{equation}
when considering variations that are Lie-advected and satisfy suitable boundary conditions. Since solutions are extrema of the action, they formally conserve the Hamiltonian 
\begin{equation}
    \label{eq:Hamiltonian}
    \mathcal{H}_{\alpha}(t) = \int_{\T^n}(1 - \alpha^2 \Delta)u \cdot u \,  dx = \int_{\T^n} |u|^2 + \alpha^2 |\nabla u|^2 \,  dx \, .
\end{equation}
This defines a functional framework in which one may study wellposedness for Euler-$\alpha$. A natural question then follows: \emph{Do solutions of \eqref{eq:euler:alpha1} which live in the function space determined by \eqref{eq:Hamiltonian} necessarily conserve the Hamiltonian?} We answer this question in the negative: 
\begin{theorem}
\label{thm:main:rough}
{Let $n\geq 2$}.  Then there exists $\beta>1$ such that there exist weak solutions of \eqref{eq:euler:alpha1} in $C^0_t H_x^{\beta}(\T^n)$ which do not conserve $\mathcal{H}_{\alpha}$. 
\end{theorem}


\subsection{Definition of Weak solution}
We first start by specifying the meaning of a weak solution of \eqref{eq:euler:alpha1}. Using incompressibility, the identity
$$  \partial_{kk} u^j \partial_l u^j = \partial_k \left( \partial_k u^j \partial_l u^j \right) - \partial_k u^j \partial_{kl} u^j = \partial_k \left( \partial_k u^j \partial_l u^j \right) - \frac{1}{2} \partial_l (\partial_k u^j \partial_k u^j) \, , $$
and redefining the pressure, we can formally rewrite \eqref{eq:euler:alpha1} as 
\begin{subequations}
    \label{eq:alpha:euler:transport1}
    \begin{align}
        \partial_t \optwo + \Div(u \otimes \optwo - \alpha^2 \nabla^Tu \nabla u ) + \nabla p &= 0\\
        \Div u &= 0 \, ,
    \end{align}
\end{subequations}
or
\begin{subequations}
\label{eq:alpha:euler:transport}
\begin{align}
     \partial_t \optwo^l  + \partial_k \left( u^k \left(u^l - \alpha^2 \Delta u^l\right) - \alpha^2 \partial_k u^j \partial_l u^j \right) + \partial_l p &= 0 \\
\partial_l u^l &= 0 \, 
\end{align}
\end{subequations}
when written in coordinates. Multiplying \eqref{eq:alpha:euler:transport1} by a smooth, divergence-free vector field $\varphi$, integrating in space and time, and integrating by parts leads to
\begin{equation}
\label{eq:weak:soln:euler:alpha}
    \int_{\mathbb{R}}\int_{\mathbb{T}^n} \partial_t \varphi \cdot u + \alpha^2 \nabla \partial_t \varphi : \nabla u  + \nabla \varphi : \left( u \otimes \optwo- \alpha^2\nabla u^T \nabla u \right) dx dt = 0 \, .
\end{equation}
This motivates the following definition:
\begin{definition}
\label{def:weak:soln}
A vector field $u \in C_t^0H_x^1$ is called a weak solution of \eqref{eq:euler:alpha1} if the vector field $u(\cdot,t)$ is weakly divergence free and has zero mean for any $t$ and satisfies \eqref{eq:weak:soln:euler:alpha} for any divergence free test function $\varphi$. 
\end{definition}
\begin{remark}
\label{rem:weak}
In order to interpret the term $\int_{\mathbb{R}}\int_{\mathbb{T}^n} \nabla  \varphi : u \otimes \Delta u \, dx \, dt$, we use duality. Specifically, since $u \in H_x^1$, $u \cdot \nabla \varphi\in H^1_x$ as well, and $\Delta u$ is well defined in $H_x^{-1}$. Using the observation that 
$$
( \Delta u_l,  u \cdot \nabla \varphi_l  )_{H^{-1},H^{1}} = - \langle \partial_i u_l, \partial_i (u \cdot \nabla \varphi_l) \rangle_{L^2,L^2}  = -\langle \partial_i u_l, \partial_i u_k \partial_k \varphi_l + u_k \partial_k \partial_i \varphi_l \rangle_{L^2,L^2} \, ,
$$
we have the following \emph{equivalent} weak formulation of \eqref{eq:weak:soln:euler:alpha}:
\begin{equation}
\label{eq:weak:soln:euler:alpha:div}
    \int_{\mathbb{R}}\int_{\mathbb{T}^n} \partial_t \varphi \cdot u + \alpha^2 \nabla \partial_t \varphi : \nabla u  + \nabla \varphi : \left( u \otimes u- \alpha^2\nabla   u^T \nabla u + \alpha^2 \nabla u \nabla u^T \right) + \alpha^2 u \cdot \nabla \nabla \varphi : \nabla u^T dx dt = 0 \, .
\end{equation}

\end{remark}
\begin{remark}
For later use, it will also be convenient to utilize the following formulation of the Euler-$\alpha$ equations, which in the smooth context a notion of solution equivalent to the previous formulations:
\begin{subequations}
    \label{eq:euler:alpha}
    \begin{align}
    \partial_t \optwo + \curl \optwo \times u + \nabla p &= 0   \\
    \Div u &= 0 \, .
    \end{align}
\end{subequations}

\end{remark}

\subsection{Rigidity for weak solutions of Euler-$\alpha$}
Inspired by Onsager's conjecture for the classical Euler equations \cite{Onsager49}, one naturally may ask at what level of regularity weak solutions as defined in Definition~\ref{def:weak:soln} conserve $\mathcal{H}_{\alpha}$. To this end, we have the following lemma (see Appendix \ref{app:energy:conservation} for the proof):
\begin{lemma}[Energy Conservation]
\label{lem:energy:conservation}
Let $u$ be a mean-zero in space weak solution of Euler-$\alpha$ as defined by definition \ref{def:weak:soln} on $\T^2$ or $\T^3$, and let $s>1$.  If ${u \in L^3\left([0,T); B_{3, \infty}^s\right) \cap C^0\left([0,T); H^1\right)}$, then the quantity 
$$
\mathcal{H}_{\alpha}(t) = \| u(t) \|_{L^2}^2 + \alpha^2 \| \nabla u \|_{L^2}^2 
$$
is constant in time for $t\in[0,T)$. 
\end{lemma}

In light of Lemma \ref{lem:energy:conservation} we make the following conjecture:
\begin{conjecture}
   Weak solutions of \eqref{eq:euler:alpha1} obey the following dichotomy.
   \begin{enumerate}
       \item For $s > 1$, weak solutions in $C_t^0H_x^1 \cap {L_t^{3}}B_{3, \infty}^s$ conserve $\mathcal{H}_{\alpha}$.
       \item For $s < 1$, there exist weak solutions in $C_t^0H_x^1 \cap {L_t^{3}}B_{3, \infty}^s$ which do not conserve $\mathcal{H}_{\alpha}$.
   \end{enumerate}
\end{conjecture}
   \begin{remark}
   Stated in terms of $L^2$-based Sobolev spaces, $H^{\sfrac{3}{2}}$ embeds into the critical regularity class $B_{3, \infty}^{1}$ in three dimensions.
   \end{remark}

\subsection{Results and Main Ideas}
We will conclude Theorem \ref{thm:main:rough} from the following:
\begin{theorem}
\label{thm:middle:third}
Let $u^{(1)}$ and $u^{(2)}$ be smooth solutions of \eqref{eq:euler:alpha1} defined on $[0,T]$ with zero mean. Moreover, suppose that $\mathcal{H}_{\alpha}(u^1) \neq \mathcal{H}_{\alpha} (u^2)$. There exists a $\beta' > 1$ and a weak solution $u \in {C^0}\left([0,T]; H_x^{\beta' }\right)$ of \eqref{eq:euler:alpha1} such that 
\begin{equation}
\label{eq:middle:third}
u \equiv u^{(1)} \, \text { on }  \left[0, \frac{T}{3} \right] \quad \text{and} \quad u \equiv u^{(2)} \, \text{ on } \left[\frac{2T}{3}, T\right] \, .
\end{equation}
\end{theorem}

The method of proof for Theorem \ref{thm:middle:third} is based on the iterative technique known as convex integration. In the context of fluid equations, convex integration schemes generally use a specially constructed stationary solution to the equations as the basic building block of the iteration.  In our iteration, we use \emph{intermittent Mikado flows}. Mikado flows were introduced in \cite{DaneriSzekelyhidi17} and featured crucially in the resolution of the Onsager conjecture for nonconservative solutions \cite{Isett2018}; see also \cite{BDLSV17} for an extension to dissipative solutions. The introduction of \emph{intermittency} to convex integration schemes dates back to the work \cite{BV19}, which gave the first example of non-uniqueness of weak solutions for the Navier-Stokes equations. Intermittent Mikado flows were later introduced in \cite{ModenaSZ17}. One may check that ``generic" Mikado tubes are not necessarily stationary solutions of \eqref{eq:euler:alpha1}, but Mikado flows with a radial flow profile about the axis (as defined in \eqref{def:radial:Mikado}) are in fact stationary solutions; see Proposition~\ref{pipeconstruction}.  For further applications of and background on convex integration schemes, we refer to the papers \cite{cheskidov2020nonuniqueness, cheskidov2020sharp, IssetVicol15, novack2020, BCV18, brue2020positive, Dai18nonunique, BV_EMS19, DLSZ19}. 

We emphasize that the need for intermittency arises from the combination of the rigidity threshhold for $\mathcal{H}_{\alpha}$, and the notion of weak solution we employ; we simultaneously need $\nabla u \in L^2$ to define the weak formulation, but also $\nabla u \notin L^3$, so that $\mathcal{H}_{\alpha}$ may vary in time. Thus our iteration should \emph{not} produce spatially homogeneous objects.  We instead need solutions for which the regularity threshold measured in $L^p$ depends fundamentally on $p$.  This dependence on $p$ of the regularity of both the building blocks and the final solution is the main attribute of intermittent flows \cite{BV19}. 

A key aspect of our construction is the composition of the high-frequency Mikado flows with Lagrangian flow maps. This composition respects the transport structure of \eqref{eq:euler:alpha1} and has been employed in earlier convex integration schemes \cite{BDLISZ15, Isett12}. The \emph{combination} of intermittency with Lagrangian flow maps presents particular challenges and was exploited for the first time in the recent work \cite{bmnv21}.  It may in fact be possible to close a convex integration scheme for \eqref{eq:euler:alpha1} \emph{without} composing with Lagrangian flow maps, perhaps using different building blocks such as intermittent jets \cite{BCV18}. However, this would require the building blocks to be very intermittent.  Utilization of such intermittent building blocks appears to impede the scheme from reaching the $L^3$ rigidity threshold. {We refer to Remark~\ref{rem:nash:explanation}, in which we make a simple computation with the \emph{Nash error} term, to justify this assertion.} Since we seek to build techniques which are as robust as possible in the hopes of eventually proving an Onsager-type conjecture for Euler-$\alpha$, we therefore use Lagrangian flow maps.

On the other hand, one difficulty created by the composition with Lagrangian flow maps is the non-trivial interaction of building blocks arising from different Lagrangian coordinate systems. We circumvent this difficulty by expanding on an insight from \cite{BBV19}, in which the supports of separate building blocks are concentrated around planes which intersect \emph{orthogonally}, in the sense that the intersection is a periodized line, and thus a lower-dimensional set.  Since the intersection of the building blocks is concentrated around a lower dimensional set, one can use intermittency to make the error terms manageable. The new difficulty not present in \cite{BBV19} is that the supports of our building blocks deform in time due to the composition with Lagrangian flow maps. On a timescale which is commensurate with that determined by $\left\| \nabla_x u \right\|_{L^\infty_{t,x}}^{-1}$, we show in Proposition~\ref{prop:weak:decoupling} that the intersection of deformed intermittent Mikado tubes oriented around non-parallel axes is concentrated around a collection of nearly-periodic \emph{points}.  Then since the intersection is concentrated around a lower-dimensional set, intermittency may be used to make these errors small.  We remark that this technique allows us to treat the intersection of non-parallel Mikado tubes in both two and three dimensions. In fact this is the only part of the scheme which requires an adjustment in two dimensions; see Remark~\ref{rem:2d:one} and Remark~\ref{rem:2d:two} for details. 

There are other methodologies one could employ for preventing the intersection of Mikado tubes arising from different Lagrangian coordinate systems.  One could try to appeal to the gluing technique \cite{Isett2018}, a placement strategy predicated on a restriction of the timescale \cite{dlk20}, or a placement strategy predicated on the degree of freedom offered by the sparseness of intermittent Mikado tubes \cite{bmnv21}.  However, these methods either do not appear to work for the Euler-$\alpha$ equations, or may work but would introduce error terms which again appear to impede the scheme from reaching the sharp $L^3$ threshold. Furthermore, none of the aforementioned methods could treat the 2D Euler-$\alpha$ equations, in which the usage of Mikado flows conflicts with the fact that non-parallel lines must intersect in two dimensions. Therefore we use Proposition~\ref{prop:weak:decoupling}, which functions perfectly well in our two- and three-dimensional intermittent context.  For additional commentary on this issue, we refer to Remark~\ref{rem:why}.






\section{Convex Integration Scheme}
Our aim is to construct a sequence of solutions $(u_q, \RR_q )$ to the relaxed system of equations\footnote{The pressure $p_q$ is determined via the incompressibility of $u_q$.}
\begin{subequations}
    \label{eq:inductive:equation}
    \begin{align}
    \partial_t (u_q - \alpha^2 \Delta u_q)^l + \partial_k\left(u_q^k(u_q^l -\alpha^2\Delta u_q^l) - \alpha^2 \partial_k u_q^j \partial_l u_q^j \right) + \partial_l p &= \partial_k \RR_q^{kl}  \\
    \partial_l u_q^l &= 0 \, ,
    \end{align}
\end{subequations}
where $\RR_q$ is a symmetric traceless matrix, and show that $\RR_q \to 0$. In order to quantify the convergence of $\RR_q$ and $u_q$, we use a frequency parameter $\lambda_q$ and an amplitude parameter $\delta_q$ defined as follows:
\begin{equation}
    \label{def:freq:amp}
    \lambda_q = a^{b^q} \qquad \delta_q = \lambda_q^{-2\beta} \, ,
\end{equation}
where $a, b \in \mathbb{N}$ are large and $\beta > 1$. We will assume the following inductive bounds on the sequence 
\begin{subequations}
    \label{eq:inductive}
    \begin{align}
    \label{eq:inductive:stress}
    \left\| \RR_q \right\|_{L^1} & \leq {{\alpha} ^2} \cdot {\calc_\RR} \cdot \delta_{q+1}\lambda_{q+1}^2  \\
    \left\| \nabla u_q \right\|_{L^2} &\leq 1 \\
  \left\|   \nabla^{3}_{t,x} u_q \right\|_{0} &\leq {\lambda_q^{4}} \\
    r_q &= {\left( \frac{\lambda_q}{\lambda_{q+1}} \right)^{\Gamma}}\, .
    \end{align}
\end{subequations}
where $\calc_\RR$ is a dimensional constant and $\Gamma \in (0,1)$.  All norms in the above assumptions are Lebesgue norms in space, measured uniformly in time; furthermore, we use $\left\| \cdot \right\|_{k}$ to denote the spatial $C^k$ norm of functions, measured uniformly in time.  All parameters used above, as well as all other parameters used throughout the paper, are defined in subsection~\ref{ss.parameters}.

Theorem \ref{thm:middle:third} is a consequence of the following proposition:
\begin{proposition}
\label{prop:main:iterative:proposition}
There exists $\beta > 1$, $b\in \mathbb{N}$, and $a_0\in\mathbb{N}$ such that for any natural number $a \geq a_0$, the following holds.  Suppose that $(u_q, \RR_q)$ are given and satisfy \eqref{eq:inductive:equation} and \eqref{eq:inductive}.  Then there exists $(u_{q+1}, \RR_{q+1} )$ which satisfy \eqref{eq:inductive:equation} and \eqref{eq:inductive} with $q$ replaced by $q+1$. Furthermore, $u_{q+1}-u_q$ satisfies the bounds
\begin{equation}
    \label{eq:iterate:size}
     \left\| u_{q+1} - u_q \right\|_{L^2} + \lambda_{q+1}^{-1} \left\| \nabla\left( u_{q+1} - u_q \right) \right\|_{L^2} + \lambda_{q+1}^{-2} \left\| \nabla^2 \left( u_{q+1} - u_q \right)\right\|_{L^2} \leq \delta_{q+1}^{\sfrac{1}{2}} \, .
\end{equation}
Finally, if the temporal support of $\RR_q$ satisfies
\begin{equation}\label{eq:supp:Rq}
    \supp_t \RR_q := \{ t: |\RR_q(x,t)| \not\equiv 0 \} \subseteq [t_1,t_2] \, ,
\end{equation}
then 
\begin{equation}\label{eq:supp:uqplus}
   \supp_t \RR_{q+1} \cup \supp_t( u_{q+1} - u_q)  \subset \left(t_1-\frac{1}{30\lambda_q},t_2+\frac{1}{30\lambda_q}\right) \, . 
\end{equation}

\end{proposition}
Assuming Proposition \ref{prop:main:iterative:proposition}, we now prove Theorem~\ref{thm:middle:third}.
\begin{proof}[Proof of Theorem~\ref{thm:middle:third}]
We proceed as in \cite{BCV18}. Let $\eta$ be a smooth cutoff function with the property that 
$\eta = 1$ on $\left[0, \frac{2T}{5} \right]$ and $ \eta = 0$ on $\left[\frac{3T}{5}, T \right]$. Define
\begin{equation}
    \label{eq:first:iterate}
    u_0 = \eta u^{(1)} + (1 - \eta) u^{(2)} \, ,
\end{equation}
where $u^{(1)}, u^{(2)}$ are as in Theorem \ref{thm:middle:third},
and 
\begin{align}
    &\RR_0 =    \mathcal{R}  \partial_t\eta \left(u^{(1)} - u^{(2)}\right)   \notag   \\
    &+ \eta (1 - \eta) \mathcal{R}\Div \left[ u^{(1)}\otimes (u^{(2)} - \alpha^2 \Delta u^{(2)}) + u^{(2)} \otimes (u^{(1)}  - \alpha^2 \Delta u^{(1)} )  - \alpha^2 \left( (\nabla u^{(1)})^T \nabla u^{(2)} +   (\nabla u^{(2)})^T \nabla u^{(1)}    \right)  \, .        \right] \notag  \end{align}
Above we have used the operator $\mathcal{R}$, which is defined in Proposition \ref{prop:fourier:inverse:div}; note that the usage of $\RR$ is justified since the inputs have zero mean. Then the pair $(u_0, \RR_0)$ solves the system \eqref{eq:inductive:equation} for an appropriate choice of pressure $p_0$, and $\supp_t \RR_0$ is contained in $[\sfrac{2T}{5}, \sfrac{3T}{5} ]$.
{We \emph{assume} for the moment that $(u_0, \RR_0)$ satisfy \eqref{eq:inductive}; at the end of the proof, we shall see that by a rescaling, this assumption may be imposed without loss of generality.}
We now inductively apply Proposition \ref{prop:main:iterative:proposition} to construct a sequence $\{(u_q, \RR_q)\}_{q=0}^{\infty}$. For $\beta' \in (1, \beta)$, we have  
\begin{align*}
   \sum_{q \geq 0} \|u_{q+1} - u_q \|_{H^{\beta'}} &\lesssim \sum_{q\geq 0} \| \nabla(u_{q+1} - u_q) \|_{L^2}^{2 - \beta'}\| \nabla^2 (u_{q+1} - u_q) \|_{L^2}^{\beta' - 1}\\
   &\lesssim (\lambda_{q+1} \delta_{q+1}^{\frac{1}{2}})^{2 - \beta'} (\lambda_{q+1}^2 \delta_{q+1}^{\frac{1}{2}})^{\beta' -1 }\\
   &\lesssim 1 \, ,
\end{align*}
where we have used \eqref{eq:iterate:size} and that the $u_q$ have zero mean. This implies the existence of a strong limit $u \in C^0_t H^{\beta'}$. By \eqref{eq:inductive}, $\RR_q$ converges to zero strongly in $L^1$.  Combining this with the regularity on $u_q$ implies that the limit $u$ satisfies the weak formulation \eqref{eq:weak:soln:euler:alpha}. Note also that by the inductive assumption on the support of the stress and the increment $ u_{q+1} - u_q$, we can conclude that the temporal support of $u- u_0$  is contained in the interval 
$$
\left( \frac{2T}{5} - \frac{1}{30}\sum_{q \geq 0} \lambda_q^{-1}\, , \quad \frac{3T}{5} + \frac{1}{30}\sum_{q \geq 0} \lambda_q^{-1} \right) \subset \left( \frac{2T}{5} - \frac{1}{30(1 -a^{-b})} \, , \quad \frac{3T}{5} + \frac{1}{30(1 -a^{-b})   }\right) \subset \left(\frac{T}{3}, \frac{2T}{3} \right)
$$
where we have used that $b\geq 2$  to deduce that $a^{bq} \leq a^{b^q}$ (for the first containment) and that $a$ is sufficiently large (for the second containment). As a consequence of this and the definition of $u_0$, \eqref{eq:middle:third} follows.

If $(u_0, \RR_0)$ do not satisfy \eqref{eq:inductive}, we note that \eqref{eq:euler:alpha1} is invariant under the rescaling $u(x ,t) \to \zeta u (x, \zeta t)$ for $\zeta \in \mathbb{R}$. Therefore, we can associate to $u^{(i)}$ defined on $[0,T]$ a function $u_{\zeta}^{(i)}(x,t) = \zeta u^{(i)}( x , \zeta t)$  defined on $[0, \zeta^{-1}T]$. By taking $\zeta$ sufficiently small depending on the choice of $a$ in \eqref{def:freq:amp}, we can ensure that the bounds in \eqref{eq:inductive} are satisfied for $u_{\zeta,0} := \eta_{\zeta} u^{(1)} + (1 - \eta_{\zeta}) u^{(2)}$, where $\eta_{\zeta}(t) = \eta(\zeta t)$, and the associated stress $\RR_{\zeta, 0}$. We can then inductively apply Proposition \ref{prop:main:iterative:proposition} as before to construct a weak solution $u_\zeta$ for \eqref{eq:euler:alpha} defined on $[0, \zeta^{-1} T]$. Rescaling the solution to produce $u$ defined on $[0,T]$ allows us to conclude the theorem.
\end{proof}

The remainder of this paper will be devoted to proving Proposition \ref{prop:main:iterative:proposition}.

\section{Mollification}
It will be necessary to mollify the equation in order to mitigate the loss of derivatives characteristic of convex integration schemes. We define the parameter $\ell$ in terms of $\lambda_q$ by
\begin{equation}\label{eq:ell:def}
    \ell=\lambda_q^{-8} \, .
\end{equation}

\begin{lemma}[Mollifying the equation]\label{lem:mollifying}

Assuming we have a pair $(u_q,\RR_q)$ satisfying \eqref{eq:inductive:equation} and \eqref{eq:inductive}, there exists a pair $(u_\ell, \RR_\ell)$ and a commutator stress $\RR_{\textnormal{comm}}$ satisfying the new equation 
\begin{subequations}
\label{eq:mollified:euler:alpha:statement}
\begin{align}
    \label{eq:velocity:mollified:equation:statement}
    \partial_t(u_\ell - \alpha^2 \Delta u_\ell )^l + \partial_k\left( u_\ell^k (u_\ell^l - \alpha^2  \Delta u_\ell^l) - \alpha^2 \partial_k u_\ell^j \partial_l u_\ell^j  \right) + \partial_l p_\ell &= \partial_k (\RR_\ell^{kl} + \RR_{\textnormal{comm}}^{kl})\\
    \label{eq:velocity:mollified:divergence:statement}
    \partial_l u_\ell^l &= 0 \, .
\end{align}
\end{subequations}
The new pair satisfies the inductive estimates \eqref{eq:inductive} as well as the higher-order estimates
\begin{equation}\label{eq:mollified:estimates}
    {\left\| \nabla_{t,x}^{m} u_\ell \right\|_{L^\infty} \lesssim {\ell^{-\sfrac{3}{2}-m}}} \, , \qquad \left\| \nabla_{t,x}^m \RR_\ell \right\|_{L^\infty} \lesssim {\ell^{-m-2}} \, , \qquad \left\| \RR_{\textnormal{comm}} \right\|_{L^1} \leq \lambda_q^{-1} \delta_{q+2}\lambda_{q+2}^2 \, 
\end{equation}
for $m\geq 0$.
\end{lemma}

\begin{proof}
Let $\phi_{\ell}$ be a family of spatial Friedrichs mollifiers at scale $\ell$, and let $\varphi_{\ell}$ be a family of temporal Friedrichs mollifiers at scale $\ell$. Then defining
\begin{subequations}
\label{eq:mollification}
\begin{align}
    \label{def:velocity:mollified}
u_{\ell} &:= (u_q * \phi_{\ell})* \varphi_{\ell} \\
\RR_{\ell} &:= (\RR_q * \phi_{\ell})* \varphi_{\ell} \, ,
\end{align}
\end{subequations}
we have that the mollified objects satisfy the equation
\begin{subequations}
\label{eq:mollified:euler:alpha}
\begin{align}
    \label{eq:velocity:mollified:equation}
    \partial_t(u_\ell - \alpha^2 \Delta u_\ell )^l + \partial_k\left( u_\ell^k (u_\ell^l - \alpha^2  \Delta u_\ell^l) - \alpha^2 \partial_k u_\ell^j \partial_l u_\ell^j  \right) + \partial_l p_\ell &= \partial_k (\RR_\ell^{kl} + R_{\text{comm}}^{kl})\\
    \label{eq:velocity:mollified:divergence}
    \partial_l u_\ell^l &= 0 \, .
\end{align}
\end{subequations}
In the above display, $R_{\textnormal{comm}}^{kl}$ is a symmetric (but not yet traceless) tensor given by the formula
\begin{align}
    \label{eq:commutator:stress}
    R_{\text{comm}}^{kl} &:= u_\ell^ku_\ell^l -  ((u_q^k u_q^l) * \phi_\ell ) * \varphi_{\ell}  - \alpha^2 \left(     \partial_k u_\ell^j \partial_l u_\ell^j - (( \partial_k u_q^j \partial_l u_q^j )* \phi_\ell )* \varphi_\ell  \right) \notag\\
    &\qquad - \alpha^2 \mathcal{R}^{kl} \Div \left( u_\ell \otimes \Delta u_\ell -  ((u_q \otimes \Delta u_q)* \phi_\ell)*  \varphi_\ell  \right) \, .
\end{align}

Using the double commutator estimate from \cite{CET94} (see also Proposition E.1 from \cite{BDLISZ15}) and \eqref{eq:inductive}, we may estimate the last term in $L^1$ by
\begin{align*}
    \| \mathcal{R}\Div (\left( u_\ell \otimes \Delta u_\ell -  ((u_q \otimes  \Delta u_q)* \phi_\ell)*  \varphi_\ell  \right) )\|_{L^1}
    &\lesssim \| \left( u_\ell \otimes \Delta u_\ell -  ((u_q \otimes  \Delta u_q)* \phi_\ell)*  \varphi_\ell  \right) \|_{L^p} \\
    &\lesssim  \| \left( u_\ell \otimes \Delta u_\ell -  ((u_q \otimes \Delta u_q)* \phi_\ell)*  \varphi_\ell  \right) \|_{L^{\infty}}\\ 
    &\lesssim \ell^2 \|\nabla^3 u_\ell \|_{L^{\infty}} \left\| \nabla u_\ell \right\|_{L^\infty}\\
    &\lesssim \ell^2 \lambda_q^8 \\
    &{\leq \lambda_q^{-1} \alpha^2 \cdot \calc_\RR \cdot \delta_{q+2}\lambda_{q+2}^2} \, ,
\end{align*}
where $p \in (1, \infty)$ is close to $1$, we used the bound for $\mathcal{R} \Div $ from Proposition~\ref{prop:fourier:inverse:div}, and we appealed to inequality \eqref{ineq:one}. The terms on the first line of \eqref{eq:commutator:stress} obey similar bounds, and we omit further details. Subtracting off the trace of $R_{\textnormal{comm}}^{kl}$ and absorbing it into the pressure concludes the proof of the estimate for $\RR_{\textnormal{comm}}$.  The estimates for $u_\ell$ and $\RR_\ell$ follow from Young's inequality for convolutions, and we omit further details.
\end{proof}

\section{Linear Algebra}
We begin with the linear algebra lemma which ensures that there are sets of vectors $v$ for which the linear combinations of simple tensors $v\otimes v$ can be used to ``span" a set of symmetric traceless matrices. The significance of the quantity $3v\otimes v - \Id$ in this Lemma will become clear after Lemma~\ref{lem:averages}. The proof can be found in Appendix \ref{app:linear:algebra}.

\begin{lemma}[Linear Algebra]\label{lem:linear:algebra}
Let $N\in\mathbb{N}$ be given. Then there exists a fixed positive constant $\calc_{\rm sum}>0$, a small positive number $\varepsilon>0$, and sets of {nine} distinct vectors $\mathcal{K}_n=\{k_i^n\}_{i=1}^9\subset \mathbb{S}^2\cap \mathbb{Q}^3$ indexed by $n\in\{0,\dots,N\}$ such that the following holds. First, if $n_1\neq n_2$, then $\mathcal{K}_{n_1}\cap\mathcal{K}_{n_2}=\emptyset$.  Secondly, let $\RR$ be a smooth, symmetric, traceless matrix with $\left\| \RR \right\|_{0}\leq {\varepsilon}$. Then there exist smooth, strictly positive functions $\left\{c_i^n\left(\RR\right)\right\}_{i=1}^9$ depending on $\RR$ and $n$ such that for all $n\in\{0,\dots,N\}$,
\begin{align}\label{eq:to:check}
    \sum_{i=1}^9 \left(c_i^n\right)^2\left(\RR\right)\left( 3 k_i^n \otimes k_i^n - \textnormal{Id} \right) = \RR \, , \qquad \sum_{i=1}^9 \left(c_i^n\right)^2\left(\RR\right) =  \calc_{\rm sum} \, .
\end{align}
\end{lemma}

\section{Technical Preliminaries}

\subsection{Decoupling}
\subsubsection{\texorpdfstring{$L^p$}{decouple} decoupling lemma}

The following lemma may be found in \cite{BV19} as Lemma 3.7. Although it is stated there for $\mathbb{T}^3$, the proof adapts easily to other dimensions (up to adjustments of certain constants). 

\begin{lemma}[$L^p$ decoupling]
\label{lem:Lp:independence}
Fix integers $\Ndec  \geq 1$, $\mu \geq \lambda \geq 1$ and assume that these integers obey
\begin{align} 
 \lambda^{\Ndec +4}    \leq \left(\frac{\mu}{2\pi\sqrt{3}}  \right)^{\Ndec} 
\,.
\label{eq:Lp:independence:assumption}
\end{align}
Let $p \in \{1,2\}$, and let $f$ be a $\T^3$-periodic function such that 
\begin{align}
\max_{0 \leq N \leq \Ndec+4} \lambda^{-N} \|D^N f\|_{L^p} \leq \calc_f
\label{eq:Lp:independence:assumption:2}
\end{align}
for a constant  $\calc_f > 0$.  Then for any $(\T/\mu)^{3}$-periodic function $g$, we have that 
\begin{align}\notag
 \|f g \|_{L^p} \lesssim \calc_f \|g\|_{L^p} \,,
\end{align}
where the implicit constant is universal (in particular, independent of $\mu$ and $\lambda$).
\end{lemma}

\subsubsection{Decoupling for intersections of deformed intermittent pipes}
Before stating the main decoupling result for products of deformed pipes in Proposition~\ref{prop:weak:decoupling}, we recall the following lemma from \cite{bmnv21}. The version written below is nearly identical to Lemma 4.7 from subsection 4.2 of \cite{bmnv21}. The only difference is that the analogue in \cite{bmnv21} of the Lipschitz bound in \eqref{e:axis:derivative:bounds} is quite sharp, whereas \eqref{e:axis:derivative:bounds} is lossy. {Since the timescale $\tau$ has been shortened commensurately in this version of the Lemma, the proof is identical in spirit, and we refer the reader to \cite{bmnv21} for more details.}

\begin{lemma}[\bf Control on Axes, Support, and Spacing]
\label{lem:axis:control}
Consider a {convex neighborhood} of space $\Omega\subset \mathbb{T}^3$.  Let $v$ be an incompressible velocity field, and define the flow $X(x,t)$ as the solution to
\begin{subequations}\label{e:flow}
\begin{align}
\partial_t X(x,t) &= v\left(X(x,t),t\right) \\
X|_{t=t_0} &= x\, ,
\end{align}
\end{subequations}
and inverse $\Phi(x,t)=X^{-1}(x,t)$ as the solution to
\begin{subequations}\label{e:transport}
\begin{align}
\partial_t \Phi + v\cdot\nabla \Phi &=0 \\
\Phi_{t=t_0} &= x\, .
\end{align}
\end{subequations}
Define $\Omega(t):=\{ x\in\mathbb{T}^3 : \Phi(x,t) \in \Omega \} = X(\Omega,t)$, and let ${\tau = \ell^3}$  be fixed.  Suppose that for $(x,t)\in\Omega(t)\times[t_0-\tau,t_0+\tau]$,
\begin{equation}\label{e:axis:derivative:bounds}
    | \nabla v(x,t) | \lesssim {\ell^{-\frac{3}{2}}} \, 
\end{equation}
{where the implicit constant is independent of $q$}. 
Let $\WW_{\lambda_{q+1},r,\xi}:\mathbb{T}^3\rightarrow\mathbb{R}^3$ be a set of straight pipe flows constructed as in Proposition~\ref{pipeconstruction} which are $\frac{\mathbb{T}^3}{\lambda_{q+1}r}$-periodic for $\frac{\lambda_q}{\lambda_{q+1}}\leq r\leq 1$ and concentrated around axes $\{A_i\}_{i\in\mathcal{I}}$ oriented in the vector direction $\xi$ for $\xi\in\mathcal{K}_n$.  Then $\WW:=\WW_{\lambda_{q+1},r,\xi}(\Phi(x,t)):\Omega(t)\times[t_0-\tau,t_0+\tau]$ satisfies the following conditions:
\begin{enumerate}
    \item If $x$ and $y$ belong to a particular axis $A_i\subset\Omega$, then 
    \begin{equation}\label{e:axis:variation}
    \frac{X(x,t)-X(y,t)}{|X(x,t)-X(y,t)|} = \xi + \delta_i(x,y,t)
    \end{equation}
    where $|\delta_i(x,y,t)|<\ell^{{\sfrac{1}{3}}}$.
    \item Let $x$ and $y$ belong to a particular axis $A_i\subset\Omega$.  Denote the length of the axis $A_i(t):=X(A_i\cap\Omega,t)$ in between $X(x,t)$ and $X(y,t)$ by $L(x,y,t)$.  Then
    \begin{equation}\label{e:axis:length}
    L(x,y,t) \leq \left(1+\ell^{\sfrac{1}{3}}\right)\left| x-y \right|.
    \end{equation}
    \item The support of $\WW$ is contained in a $\calc_\xi\lambda_{q+1}^{-1}$-neighborhood of 
    \begin{equation}\label{e:axis:union}
       \bigcup_{i} A_i(t) \, ,
    \end{equation}
    where $\calc_\xi$ is a dimensional constant depending only on $N$ and the sets $\mathcal{K}_n$ from Lemma~\ref{lem:linear:algebra}.
    \item $\WW$ is ``approximately periodic" in the sense that for distinct axes $A_i,A_j$ with $i\neq j$ and $\dist(A_i\cap\Omega,A_j\cap\Omega)=d$,
\begin{equation}\label{e:axis:periodicity:1}
    \left(1-\ell^{\sfrac{1}{3}}\right)d \leq \dist\left(A_i(t),A_j(t)\right)\leq \left(1+\ell^{\sfrac{1}{3}}\right)d.
\end{equation}
\end{enumerate}

\end{lemma}
We now state and prove the main proposition which allows us to obtain a decoupling-style estimate for the product of two pipe flows with equal periodicity but oriented around axes with non-parallel tangent vectors. 

\begin{proposition}[\bf Decoupling of Minimally Orthogonal Pipes] \label{prop:weak:decoupling}
Consider neighborhoods of space $\Omega^1$ and $\Omega^2$ {satisfying the assumptions of Lemma~\ref{lem:axis:control}}.  Let $v$ be an incompressible velocity field, and define 
\begin{subequations}\label{e:transport:1}
\begin{align}
\partial_t \Phi_1 + v\cdot\nabla \Phi_1 &=0 \\
\Phi|_{t=t_1} &= x\, .
\end{align}
\end{subequations}
and
\begin{subequations}\label{e:transport:2}
\begin{align}
\partial_t \Phi_2 + v\cdot\nabla \Phi_2 &=0 \\
\Phi|_{t=t_2} &= x\, .
\end{align}
\end{subequations}
Define $\Omega^1(t)$ and $\Omega^2(t)$ as in Lemma~\ref{lem:axis:control}, and assume that $v$ satisfies the usual estimate \eqref{e:axis:derivative:bounds} on $\Omega^1(t)\cup \Omega^2(t)$.  For $\xi_1,\xi_2$ belonging to $\mathcal{K}_{n_1}$, $\mathcal{K}_{n_2}$, with $n_1\neq n_2$, let $\WW^1_{\lambda_{q+1},r_2,\xi_1}:\Omega^1\rightarrow \mathbb{R}^3$ and $\WW^2_{\lambda_{q+1},r_2,\xi_2}:\Omega^2\rightarrow \mathbb{R}^3$ be pipe flows with \emph{shared} periodicity $(\lambda_{q+1}r_2)^{-1}$ but concentrated around axes with \emph{distinct} unit vector directions $\xi_1,\xi_2\in\mathbb{S}^2$ such that there exists $\epsilon_0>0$ satisfying
\begin{equation}\label{e:inner:product:control}
| \langle \xi_1,\xi_2 \rangle | \leq 1 - \epsilon_0 \, , \qquad {\ell^{\sfrac{1}{3}}}\leq\frac{\epsilon_0}{2} \, .
\end{equation}
Define $\WW^1:=\WW^1_{\lambda_{q+1},r_2,\xi_1}(\Phi_1)$ and $\WW^2:=\WW^2_{\lambda_{q+1},r_2,\xi_2}(\Phi_2)$ to be the transported pipe flows with shifting spatial domains $\Omega^1(t)$ and $\Omega^2(t)$, respectively. Set ${\tau=\ell^3}$ and assume that there exists 
$$t_0\in [t_1-\tau,t_1+\tau]\cap [t_2-\tau,t_2+\tau]$$
such that $\Omega^1(t_0)\cap\Omega^2(t_0) \neq \emptyset$.  Set
$$\Omega(t):=\Omega^1(t)\cap\Omega^2(t)\, . $$
Let $f:\Omega(t)\times[t_1-\tau,t_1+\tau]\cap[t_2-\tau,t_2+\tau]$ be a function such that
\begin{equation}\label{e:derivatives:on:f}
    \left\| D^n f \right\|_{L^\infty_t\left(L^1(\Omega(t))\right)} \leq {\mathcal{C}_f} \left( \lambda_{q+1} r_1 \right)^n \, , \qquad  n\leq \Ndec+4 
\end{equation}
{where $m_1,m_2,\Ndec\in\mathbb{N}$ and $\Ndec,r_1,r_2$ satisfy}
\begin{equation}\label{e:r:decoupling}
   0\leq r_1 \leq r_2 \leq 1\, , \qquad \frac{r_1}{r_2}\leq \frac{2}{3}\, , \qquad (\lambda_{q+1}r_1)^{\Ndec +4} \leq \left(\lambda_{q+1}r_2\right)^\Ndec \, .
\end{equation}
Furthermore, assume that for $i=1,2$ and $\tau\in[t_i-\tau,t_i+\tau]$, $\Phi_i$ satisfies
\begin{equation}
    \label{eq:deformation:bounds:decoupling}
  {  |\nabla \Phi_i - \Id| \leq \ell \, , \qquad \| \nabla^{n}\Phi_i \|_{L^{\infty}} \leq \calc_\Phi(\lambda_{q+1}r_1)^{n-1} \, , \qquad 1 \leq n \leq \max( m_1, m_2) \, . }
\end{equation}
Then for $t\in[t_1-\tau,t_1+\tau]\cap[t_2-\tau,t_2+\tau]$, 
\begin{equation}\label{e:weak:decoupling:conclusion}
    \left\| f(t,\cdot) \nabla^{m_1}\WW^1(t,\cdot) \otimes \nabla^{m_2}\WW^2(t,\cdot) \right\|_{L^1(\Omega(t))} \lesssim  {r_2\cdot\mathcal{C}_f\cdot\lambda_{q+1}^{m_1 + m_2}}  
\end{equation}
for a universal implicit constant depending on $\epsilon_0$, the sets $\mathcal{K}_n$ for $n\in\{0,...,N\}$, and $\calc_\Phi$, but independent of $f$, $r_1$, $r_2$, or $\tau$.
\end{proposition}
\begin{remark}\label{rem:2d:one}{{We note that the entire Proposition applies as long as the vector directions $\xi_1$ and $\xi_2$ are minimally orthogonal in the sense of \eqref{e:inner:product:control}.  Furthermore, the scaling in two dimensions of \eqref{e:weak:decoupling:conclusion} is identical, since $L^2$-normalized objects have magnitude $r_2^{-\sfrac{1}{2}}$ in $L^\infty$, and the area of the support of an intersection of two orthogonal intermittent pipe flows is $r_2^2$. The proof goes through \emph{mutatis mutandis} after making these slight adjustments, and we omit the details. In the case of the 2D Euler-$\alpha$ equations, Mikado flows may not be taken to be disjoint. This lemma allows Theorem~\ref{thm:main:rough} to apply for the 2D Euler-$\alpha$ equations as well, since the all intersection terms (not just the ones arising from pipes originating from different Lagrangian coordinate systems) may be estimated using this Proposition.}}
\end{remark}
\begin{proof}[Proof of Proposition~\ref{prop:weak:decoupling}]
Define the set of intersection points $I(t)$ at time $t$ for the flowed pipes $\WW^1$ and $\WW^2$ by
\begin{equation}\label{e:It:definition}
    I(t) := \{ x: x\in \supp \WW^1(t) \cap \supp \WW^2(t) \} \, .
\end{equation}
We will show that at time $t_0$, 
\begin{equation}\label{e:It:containment}
    I(t_0) \subset \bigcup_i B_i \, ,
\end{equation}
where the $B_i$'s are pairwise disjoint balls of radius $\calc(\xi_1,\xi_2)\left(\lambda_{q+1}r_2\right)^{-1}$ for some dimensional constant {$\calc(\xi_1,\xi_2)$.}  Furthermore, we will show that
\begin{equation}\label{e:It:proportion}
    \left| I(t) \cap B_i \right| {\lesssim}_{\epsilon_0, \calc(\xi_1,\xi_2)} \lambda_{q+1}^{-3}
\end{equation}
for a universal constant which we emphasize depends only on $\epsilon_0$ and $\calc(\xi_1,\xi_2)$. After proving this, we will be able to follow the proof of the usual decoupling lemma to prove that the desired $L^1$ estimate holds at $t_0$. Since the intersection points are transported by the incompressible velocity field $v$, the same decoupling estimate will hold for all times. One could equivalently rerun the argument for all $t$ for which $\Omega_1(t)\cap\Omega_2(t) \neq \emptyset$.

Let $x\in I(t_0)$.  Then $x\in \supp \WW^1(t_0)\cap \supp \WW^2(t_0)$, and by \eqref{e:axis:union} from Lemma~\ref{lem:axis:control}, there exist $x_1,x_2$ belonging to flowed axes $A_{i_1}(t_0)$, $A_{i_2}(t_0)$, respectively, where we haved mimicked the notation from \eqref{lem:axis:control} to define the flowed axes for $\WW^1$ and $\WW^2$, such that
\begin{equation}\label{e:weak:decoupling:1}
    |x_1-x| \lesssim \lambda_{q+1}^{-1}, \quad |x_2-x| \lesssim \lambda_{q+1}^{-1} \, .
\end{equation}
The implicit constants in the above inequality depend only on $\calc_{\xi_1}$ and $\calc_{\xi_2}$ but nothing else. By \eqref{e:axis:periodicity:1} from Lemma~\ref{lem:axis:control} and the shared periodicity of $\WW^1$ and $\WW^2$ to scale $(\lambda_{q+1}r_2)^{-1}$, $A_{i_1}(t_0)$ and $A_{i_2}(t_0)$ are the only axes of $\WW^1(t)$ and $\WW^2(t)$, respectively, which are closer than a dimensional constant $\rho_{1,2}^{-1}$ multiplied by $(\lambda_{q+1}r_2)^{-1}$ to any of the points $x$, $x_1$, or $x_2$.  We therefore define $B_{x}$ by
\begin{equation}\label{e:decoupling:Bx}
    B_x := \left\{ y\in\mathbb{T}^3: |x-y| \leq \left(\rho_{1,2} \lambda_{q+1}r_2\right)^{-1} \right\}.
\end{equation}
We emphasize that $\rho_{1,2}$ depends only on $\xi_1$ and $\xi_2$, or equivalently on the sets $\mathcal{K}_n$ from Proposition~\ref{lem:linear:algebra}.  We also note that after setting $\calc_\xi=\max\{\calc_{\xi_1},\calc_{\xi_2}\}$, $I(t_0) \cap B_x$ only contains points which are within $\calc_{\xi}\lambda_{q+1}^{-1}$ of either $A_{i_1}(t_0)$ or $A_{i_2}(t_0)$ by \eqref{e:axis:union}.

We will now prove that 
\begin{equation}\label{e:decoupling:I:Bx}
\left| I(t_0) \cap B_x \right| \lesssim_{\epsilon_0} \lambda_{q+1}^{-3} \, .  
\end{equation}
Towards this end, consider the cones $C^j$ defined for $j=1,2$ by
\begin{equation}\label{e:decoupling:cone}
    C^j := \{ x\in\mathbb{T}^3 : x=x_j+r\xi \, \textnormal{ for } \,  |\xi-\xi_j| \leq \ell^{\sfrac{1}{3}}, r\in\mathbb{R} \} \, .
\end{equation}
By \eqref{e:axis:variation}, these cones contain the axes $A_{i_1}(t_0)$ and $A_{i_2}(t_0)$. Let 
\begin{equation}\label{e:decoupling:extended:cone}
    \tilde{C}^j:= \{ x\in \mathbb{T}^3 : |x-C^j| \leq \calc_\xi\lambda_{q+1}^{-1} \}.
\end{equation}
By \eqref{e:axis:union} and the previous observation that only $A_{i_1}(t_0)$ and $A_{i_2}(t_0)$ have non-empty intersection with $B_x$, $\tilde{C}^j$ contains the support of $\WW^j\cap B_x$ for $j=1,2$. Thus to provide an upper bound on the measure of $I(t_0)\cap B_x$, it suffices to estimate the measure of $\tilde{C}^1\cap \tilde{C}^2$. 

Let $x\in \tilde{C}^1\cap \tilde{C}^2$. Then for $j=1,2$, there exist $r_j\in\mathbb{R}$, $\xi_{j,x}\in\mathbb{S}^2$ such that $|\xi_j-\xi_{j,x}|<\ell^{\sfrac{1}{3}}$, and $y_j$ such that $|y_j-x_j|\leq \calc_\xi\lambda_{q+1}^{-1}$, with
\begin{equation}\label{e:decoupling:equality}
 x= y_1+r_1\xi_{1,x} = y_2 + r_2\xi_{2,x}\,.
\end{equation}
By \eqref{e:weak:decoupling:1}, we have that then
$$  |y_1-y_2| \lesssim_{\calc_\xi} \lambda_{q+1}^{-1} \, .  $$
As a consequence of \eqref{e:inner:product:control},
\begin{align}
    \lambda_{q+1}^{-2} &\gtrsim |r_1\xi_{1,x}-r_2\xi_{2,x}|^2 \nonumber\\
    &= r_1^2+r_2^2-2r_1r_2\langle \xi_{1,x},\xi_{2,x}\rangle\nonumber\\
    &\geq r_1^2+r_2^2 - \left(1-\frac{\epsilon_0}{2}\right)\left( r_1^2+r_2^2 \right)\nonumber\\
    &\geq \frac{\epsilon_0}{2}\left( r_1^2+r_2^2 \right)\, , \notag
\end{align}
which is only possible when 
$$  r_1^2+r_2^2 \lesssim \frac{1}{\epsilon_0}\lambda_{q+1}^{-2} \, .  $$
We can then brutally bound the volume of $\tilde{C}^1\cap\tilde{C}^2$ by the volume of the union of two balls centered at $x_j$ whose radius is twice the maximum value of $r_j$, yielding 
\begin{equation}\label{e:decoupling:proportion}
    \left|\tilde{C}^1\cap\tilde{C}^2\right| \lesssim_{\epsilon_0,\calc_\xi} \lambda_{q+1}^{-3}.
\end{equation}
Since $I(t_0)\cap B_x\subset \tilde{C}^1\cap\tilde{C}^2$, we have verified \eqref{e:It:proportion} for $B_x$.  Repeating the above steps by finding $x\in I(t_0)\setminus B_x$ and so forth verifies \eqref{e:It:containment}.

To demonstrate the desired decoupling estimate at $t_0$, consider
\begin{align}
    \int_{\Omega(t_0)} \left| f   \nabla^{m_1} \WW^1  \otimes \nabla^{m_2}\WW^2 \right| \,dx &= \int_{I(t_0)} \left| f \nabla^{m_1}\WW^1 \otimes \nabla^{m_2}\WW^2 \right| \,dx \nonumber\\
    &\leq \sum_{i} \int_{B_i} \left| f \nabla^{m_1} \WW^1 \otimes \nabla^{m_2} \WW^2 \right| \,dx\,. \label{e:decoupling:end:1}
\end{align}
Letting $\bar{h}_i$ denote the mean of any function $h$ on $B_i$, for $x\in B_i$, we can write that
\begin{align}
    |f(x)| &= |(\bar{f}_i + f(x) - \bar{f}_i)| \nonumber\\
    &\leq |\bar{f}_i| + (2\rho_{1,2}\lambda_{q+1}r_2)^{-1} \sup_{B_i} | D f | \notag
\end{align}
since the diameter of $B_i$ is $(2\rho_{1,2}\lambda_{q+1}r_2)^{-1}$. Iterating $\Ndec$ times, we obtain the pointwise estimate for $x\in B_i$
\begin{equation}
|f(x)| \leq \sum_{m=0}^{\Ndec-1} \left(2\rho_{1,2}\lambda_{q+1}r_2\right)^{-m} \left|\overline{D^{m }f}_i\right| + \left(\lambda_{q+1}r_2\right)^{-\Ndec}\left\| D^{\Ndec } f \right\|_{L^\infty(\Omega(t_0))}\,. \notag
\end{equation}
Multiplying by $\nabla^{m_1}\WW^1 \otimes \nabla^{m_2}\WW^2$, integrating, and summing over $i$, \eqref{e:decoupling:end:1} is less than or equal to
\begin{align}
    \sum_i \int_{B_i} &|\nabla^{m_1}\WW^1\otimes \nabla^{m_2}\WW^2|\left( \sum_{m=0}^{\Ndec-1} \left(2\rho_{1,2}\lambda_{q+1}r_2\right)^{-m} |\overline{D^{m} f}_i| \right)\notag \\
    &\qquad + \left(2\rho_{1,2}\lambda_{q+1}r_2\right)^{-\Ndec}\left\| D^{\Ndec}f \right\|_{L^\infty(\Omega(t_0))} \left\|\nabla^{m_1} \WW^1 \otimes \nabla^{m_2} \WW^2 \right\|_{L^1(\Omega(t_0))}\nonumber\\
    &\leq \sum_i  \int_{B_i} | \nabla^{m_1}\WW^1 \otimes \nabla^{m_2}\WW^2| \left( \sum_{m=0}^{\Ndec-1} \left(2\rho_{1,2}\lambda_{q+1}r_2\right)^{-m} \frac{1}{|B_i|} \left\|D^{m} f_i\right\|_{L^1(B_i)} \right) \nonumber\\
    &\qquad + \left(2\rho_{1,2}\lambda_{q+1}r_2\right)^{-\Ndec}\left\| D^{\Ndec } f \right\|_{L^\infty(\Omega(t_0))} \left\| \nabla^{m_1} \WW^1\otimes \nabla^{m_2} \WW^2 \right\|_{L^1(\Omega(t_0))}  . \label{eq:integrating:1}
\end{align}
We note that by \eqref{e:pipe:estimates:2}, \eqref{eq:deformation:bounds:decoupling}, repeated application of the chain rule, and the inequality $\lambda_{q + 1} \gg \lambda_{q+1} r_1$, we have 
\begin{equation}
     \| \nabla^{m_i} \WW^i \|_{L^{\infty}} \lesssim r_2^{-1}\lambda_{q+1}^{m_i} \, . \notag
\end{equation}
Since  
$$  \left| I(t_0) \cap B_i \right| \lesssim \lambda_{q+1}^{-3} \, ,  $$
we have
\begin{equation}
    \int_{B_i} | \nabla^{m_1}\WW^1 \otimes \nabla^{m_2}\WW^2 | \leq \left| I(t_0) \cap B_i \right| \left\|\nabla^{m_1} \WW^1  \otimes \nabla^{m_2}\WW^2 \right\|_{L^\infty(B_i)}\lesssim  r_2^{-2}\lambda_{q+1}^{m_1 + m_2 -3}\,. \label{e:decoupling:ballL1}
\end{equation}
Using the fact that there exist at most $\approx\left(\lambda_{q+1}r_2\right)^3$ disjoint balls of radius $(\rho_{1,2}\lambda_{q+1}r_2)^{-1}$ in $\mathbb{T}^3$ (by volume constraints), we have that using Sobolev embedding, \eqref{e:derivatives:on:f}, \eqref{e:r:decoupling}, and \eqref{e:decoupling:ballL1}, we can estimate the second term in \eqref{eq:integrating:1} by
\begin{align}
    \left(2\rho_{1,2}\lambda_{q+1}r_2\right)^{-\Ndec}&\left\| D^{\Ndec } f \right\|_{L^\infty(\Omega(t_0))} \left\| \nabla^{m_1} \WW^1  \otimes \nabla^{m_2}\WW^2 \right\|_{L^1(\Omega(t_0))}\nonumber\\
    &\lesssim \left( \lambda_{q+1}r_2 \right)^{-\Ndec} \cdot\mathcal{C}_f\cdot \left(\lambda_{q+1}r_1\right)^{\Ndec +4} \| \nabla^{m_1} \WW^1  \otimes \nabla^{m_2}\WW^2 \|_{L^1\left(\Omega(t_0)\right)}\nonumber\\
    &\leq C(\epsilon_0,\mathcal{K})\cdot {\mathcal{C}_f}\cdot \left(\lambda_{q+1}r_2\right)^3 \lambda_{q+1}^{m_1 + m_2-3} r_2^{-2}\nonumber\\
    &\leq C(\epsilon_0,\mathcal{K})\cdot {\mathcal{C}_f}\cdot r_2  \lambda_{q+1}^{m_1 + m_2 } \, .\nonumber
\end{align}
The constant in the last line depends only on $\epsilon_0$, the sets of vector directions $\mathcal{K}_n$, and geometric quantities. Thus it remains to estimate the first term. We may write that
\begin{align}
\sum_i \int_{B_i} &|\nabla^{m_1}\WW^1 \otimes \nabla^{m_2}\WW^2| \left( \sum_{m=0}^{\Ndec-1} \left(2\rho_{1,2}\lambda_{q+1}r_2\right)^{-m} \frac{1}{|B_i|} \left\|D^{m } f \right\|_{L^1(B_i)} \right)    \nonumber\\
&\leq C(\epsilon_0,\mathcal{K}) \sum_{m=0}^{\Ndec-1} \lambda_{q+1}^{m_1 + m_2 -3} r_2^{-2}\left(\lambda_{q+1}r_2\right)^3 \left(\frac{\lambda_{q+1}r_1}{\lambda_{q+1}r_2}\right)^m  {\mathcal{C}_f}\nonumber\\
&\leq C(\mathcal{K},\epsilon_0) \lambda_{q+1}^{m_1 + m_2} r_2  {\mathcal{C}_f} \nonumber
\end{align}
using again \eqref{e:derivatives:on:f},  \eqref{e:r:decoupling}, \eqref{e:decoupling:Bx}, and \eqref{e:decoupling:ballL1}. Summing estimates for the two terms proves \eqref{e:weak:decoupling:conclusion} at the fixed time $t_0$. 

To prove the estimate for all times $t$, first notice that by the incompressibility of the flow, the $B_i$'s remain disjoint upon advection.  Furthermore, since $\WW^1$ and $\WW^2$ are constant along the characteristics of $v$, the $L^1$ norm of their product on each $B_i$ remains constant as well.  Thus, the only part of the argument which requires adjustment is the estimate on the diameters of the flowed balls $B_i$, which is easily seen to increase by a geometric factor on the timescale in question after using the results of Lemma~\ref{lem:axis:control}. 
\end{proof}

\subsection{Inverse Divergence}
The inverse divergence we use in this paper must be flexible enough to handle both intermittency and composition with diffeomorphisms.  For these purposes, we use an inverse divergence predicated on ``differentation by parts" rather than Fourier multipliers. The inverse divergence operator from \cite{bmnv21} is built upon this principle; in fact, the estimates proven there are much more detailed than anything necessary in this paper.  However, we reproduce identically the ``iterative step" from \cite{bmnv21} but state a significantly streamlined and simplified version of the estimates satisfied by the inverse divergence operator.  For the reader's convenience, we include an outline of the proof of the claimed estimates for the inverse divergence operator.

\begin{proposition}[Inverse divergence iteration step]
\label{prop:Celtics:suck}
Let $n\geq 2$ and fix two zero-mean $\T^n$-periodic functions $\varrho$ and $\vartheta$ with $\varrho  =  \Delta \vartheta $. Let $\Phi$ be a volume preserving transformation of $\T^n$ such that $\norm{\nabla \Phi - \Id}_{L^\infty(\T^n)} \leq \sfrac 12$. Define the matrix $A = (\nabla \Phi)^{-1}$. Given a smooth vector field $G^i$, we have
\begin{align}
G^i  \varrho\circ \Phi = \partial_m \RR^{im} + \partial_i P + E^i \, ,
\label{eq:Celtics:suck:total}
\end{align}
where the traceless symmetric stress $R^{im}$ is given by
\begin{align}
\RR^{im}
&= \left(  G^i A^m_\ell  + G^m A^i_\ell  -A^i_k A^m_k   G^p \partial_p \Phi^\ell \right) (\partial_\ell \vartheta) \circ \Phi  -  P \delta_{im}
\,,
\label{eq:Celtics:suck:stress}
\end{align}
the pressure term is given by 
\begin{align}
P
&= \left(2   G^m A^m_\ell   -A^m_k A^m_k   G^p \partial_p \Phi^\ell \right) (\partial_\ell \vartheta) \circ \Phi   \, ,
\label{eq:Celtics:suck:pressure}
\end{align}
and the error term $E^i$ is given by 
\begin{align}
E^i
&=  \left(  \partial_m \left(G^p A^i_k A^m_k - G^m A^i_k A^p_k\right)  \partial_p \Phi^\ell
- \partial_m G^i A^m_\ell  \right) (\partial_\ell \vartheta) \circ \Phi 
\, .
\label{eq:Celtics:suck:error}
\end{align}
\end{proposition}

Before defining and collecting estimates on the primary inverse divergence operator, we recall the Fourier-multiplier frequently used in convex integration schemes, for example in \cite{BDLISZ15}.

\begin{proposition}[Fourier-multiplier inverse divergence]\label{prop:fourier:inverse:div}
Let $v:\mathbb{T}^3\rightarrow\mathbb{T}^3$ be a $C^\infty$ vector field. Define
\begin{equation}\label{eq:divR}
    (\mathcal{R}v)^{ij} = \left(-\frac{1}{2}\Delta^{-2} \partial_i \partial_j \partial_k - \frac{1}{2}\Delta^{-1}\partial_k \delta_{ij} + \Delta^{-1}\partial_i \delta_{jk} + \Delta^{-1} \partial_j \delta_{ik} \right) \left(v^k - \fint_{\T^3} v^k \right) \, .
\end{equation}
Then $\Div\left(\mathcal{R}v\right)=v - \fint_{\T^3} v$, and the operator $\nabla\circ\mathcal{R}$ is a singular integral operator which is bounded from $L^p(\T^3)$ to itself for all $1<p<\infty$ (with $p$-dependent bounds).
\end{proposition}

With the iterative step and nonlocal inverse divergence operator in hand, we can now construct the full inverse divergence operator and record and prove the estimates satisfied by the output. For the purposes of the statement of this Proposition, we use the notation
\begin{equation}\label{eq:MM:def}
    \MM{n,N,\lambda,\Lambda} = \lambda^{\min\{n,N\}} \Lambda^{\max\{n-N,0\}} \, .
\end{equation}

\begin{proposition}[Inverse divergence with estimates]
\label{prop:intermittent:inverse:div}
Fix an incompressible vector field $v(t,x):\mathbb{R}\times\T^n\rightarrow\R^n$ for $n=2,3$ and denote its material derivative by $D_t = \partial_t + v\cdot\nabla$. 
Let $G:\R\times\T^n\rightarrow\R^n$ be a vector field, and assume there exists a constant $\calc_{G} > 0$, $\lambda\geq 1$, and a large integer $\dpot$ such that
\begin{align}
\norm{D^N G}_{L^{1}(\T^n)}\lesssim \calc_{G} \lambda^N 
\label{eq:inverse:div:DN:G}
\end{align}
{for all $N\leq 2\dpot$ with implicit constants which may depend on $N$ but not $q$.} Let $\Phi$ be a volume preserving transformation of $\T^n$ such that 
\begin{equation}\label{eq:diff:bounds}
D_t \Phi = 0 \,
\qquad \mbox{and} \qquad
\norm{\nabla \Phi - \Id}_{L^\infty(\supp G)} \leq \sfrac 12 \,.
\end{equation} 
Denote by $\Phi^{-1}$ the inverse of the flow $\Phi$,  which is the identity at a time slice which intersects the support of $G$.
Assume that the velocity field $v$ and the flow functions $\Phi$ and $\Phi^{-1}$ satisfy the bounds
\begin{align}
\norm{D^{N+1}   \Phi}_{L^{\infty}(\supp G)} + \norm{D^{N+1}   \Phi^{-1}}_{L^{\infty}(\supp G)} + \left\| D^N \left(\nabla\Phi\right)^{-1} \right\|_{L^\infty(\supp G)}
&\les \lambda^{N} \, 
\label{eq:DDpsi}
\end{align}
for all $N\leq 2\dpot$, where the implicit constants may again depend on $N$ but not $q$. Lastly, let $\varrho,\vartheta \colon \T^n \to \R$ be two zero mean functions with  the following properties:
\begin{enumerate}
\item\label{item:inverse:i} There exists a parameter $\zeta\geq 1$ such that $\varrho (x) = \zeta^{-2\dpot } \Delta^\dpot \vartheta(x)$.
\item\label{item:inverse:ii} There exists a parameter $\mu\geq 1$ such that $\varrho$ and $\vartheta$ are $(\sfrac{\T}{\mu})^n$-periodic.
\item\label{item:inverse:iii} There exist parameters $\Lambda\geq \zeta$ and $\calc_{*} \geq 1$ such that 
\begin{align}
\norm{D^N \varrho}_{L^1} \les \calc_{*} \Lambda^{N}
\qquad \mbox{and} \qquad
\norm{D^N \vartheta}_{L^1} \les \calc_{*} \MM{N,2\dpot,\zeta,\Lambda} 
\label{eq:DN:Mikado:density}
\end{align} 
for all {$N \leq 4\dpot$}
except for the case $N = 2\dpot$ when the Calder\'on-Zygmund inequality fails. In this exceptional case, the second inequality in \eqref{eq:DN:Mikado:density} is allowed to be weaker by a factor of $\Lambda^\alpha$, for an arbitrary $\alpha \in (0,1]$; that is, we only require that $\norm{D^{2\dpot} \vartheta}_{L^1} \les \calc_{*} \Lambda^\alpha\zeta^{2\dpot} $.
\end{enumerate}

If the parameters satisfy
\begin{align}\label{eq:exchange}
{\left(\frac{\lambda}{\zeta}\right)^{\dpot}\lambda^4 \leq \zeta^{-1} \, ,}
\end{align}
then we have that 
\begin{align}
G \; \varrho\circ \Phi  &=  \Div \RR + \nabla P + \fint_{\mathbb{T}^3} G \varrho\circ \Phi \, . \label{eq:inverse:div}
\end{align}
Furthermore, the traceless symmetric stress $\RR$ and pressure $P$ satisfy the bound
\begin{align}
\norm{\RR}_{L^{1}} + \norm{P}_{L^{1}}
 &\les \calc_{G} {\lambda^4} \calc_{*}   \zeta^{-1} \, ,
\label{eq:inverse:div:stress:1}
\end{align}
with an implicit constant which is independent of $G$, $\varrho$, or $\Phi$.

\end{proposition}

\begin{proof}
The proof mirrors that of \cite{bmnv21} in spirit. Since there is no need to propagate any sharp derivative estimates, however, the proof is significantly simpler. In fact there is no need to ensure that decoupling lemmas apply as in \cite{bmnv21}, leading to the $\lambda^4$ Sobolev loss in \eqref{eq:inverse:div:stress:1}. These shortcomings are irrelevant for applications in this paper but could be avoided by appealing to the more sophisticated estimates in \cite{bmnv21}.  In any case, the strategy is to apply Proposition~\ref{prop:Celtics:suck} to produce a sequence of matrices $\RR_{(j)}$ and scalar functions $P_{(j)}$ for {$0\leq j \leq \dpot-1$} which each satisfy the estimates in \eqref{eq:inverse:div:stress:1}.  Every iteration will produce a symmetric tensor $\RR_{(j)}$ and pressure $P_{(j)}$ which are each smaller than $G\varrho\circ\Phi$ by at least $\zeta^{-1}$ and an error term $E_{(j)}$ which has gained one factor of $\zeta^{-1}\lambda$ relative to its predecessor; that is, we repeatedly differentiate by parts to produce a symmetric tensor and an error term which has ``exchanged expensive derivatives for cheap derivatives." We may halt this exchanging process after reaching the threshold indicated in \eqref{eq:exchange}, at which point we simply apply the usual Fourier multiplier inverse divergence and use its trivial $L^2(\mathbb{T}^n)\rightarrow L^2(\mathbb{T}^n)$ bound.

\newcommand{\varrhozero}{\varrho_{(0)}}
\newcommand{\varrhok}{\varrho_{(k)}}
\newcommand{\varrhokminus}{\varrho_{(k-1)}}

To set notation let us define 
\begin{equation}\label{eq:rhos}
    \varrhozero = \varrho \, , \qquad \varrhok = \left(\zeta^{-2}\Delta\right)^{\dpot-k} \vartheta \quad \forall
    1 \leq k \leq \dpot \, .
\end{equation}
From these definitions, we have that 
\begin{equation}\label{eq:rho:}
\varrhokminus = \zeta^{-2} \Delta \varrhok \, .
\end{equation}
We begin by applying Proposition~\ref{prop:Celtics:suck} with $\varrho=\varrho_{(0)}$ and $\vartheta=\varrho_{(1)}$ so that $\varrho=\zeta^{-2}\Delta \varrho_{(1)}$. From \eqref{eq:diff:bounds} we have that the assumption on $\Phi$ is satisfied, and so we obtain
\begin{equation}\notag
    G^i \varrho_{(0)} \circ \Phi = \partial_m \RR^{im}_{(0)} + \partial_i P_{(0)} + E^i_{(0)} \, .
\end{equation}
From \eqref{eq:Celtics:suck:stress} and \eqref{eq:Celtics:suck:pressure}, we have that both the symmetric traceless stress $\RR^{im}_{(0)}$ and the pressure $P^i_{(0)}$ are defined in terms of the symmetric tensor
\begin{equation}\label{eq:to:estimate}
    \left( G^i A^m_\ell + G^m A^i_\ell - A^i_k A^m_k G^p \partial_p \Phi^\ell \right) \left(\zeta^{-2}\partial_\ell \varrho_{(1)}\right)\circ \Phi \, ,
\end{equation}
and so it will suffice to estimate \eqref{eq:to:estimate}. From \eqref{eq:DDpsi}, \eqref{eq:inverse:div:DN:G}, and the Sobolev embedding $W^{n+1,1}(\T^n)\hookrightarrow L^\infty(\T^n)$, we have that
\begin{equation}\notag
\left\| G^i A^m_\ell + G^m A^i_\ell - A^i_k A^m_k G^p \partial_p \Phi^\ell \right\|_{L^\infty} \lesssim \calc_G {\lambda^4} \, .
\end{equation}
From item~\eqref{item:inverse:i} and \eqref{eq:DN:Mikado:density}, we have that 
$$  \left\|  \zeta^{-2} \partial_\ell \varrho_{(1)}\circ \Phi \right\|_{L^1} = \left\| \zeta^{-2} \partial_\ell \varrho_{(1)}\right\|_{L^1}  = \left\| \zeta^{-2} \zeta^{-2(\dpot-1)} \partial_\ell \Delta^{\dpot-1}\vartheta \right\|_{L^1} \lesssim \zeta^{-1} \calc_* \, . $$
Thus we may bound the tensor in \eqref{eq:to:estimate} in $L^1(\T^n)$ by 
$$  \calc_G {\lambda^4} \zeta^{-1} \calc_* $$
as desired in \eqref{eq:inverse:div:stress:1}. The error term $E^i_{(0)}$ is given from \eqref{eq:Celtics:suck:error} by 
\begin{align} \notag
    E^i_{(0)}
&=  \left(  \partial_m \left(G^p A^i_k A^m_k - G^m A^i_k A^p_k\right)  \partial_p \Phi^\ell
- \partial_m G^i A^m_\ell  \right) (\zeta^{-2}\partial_\ell \varrho_{(1)}) \circ \Phi  \, .
\end{align}
Upon decomposition of the summation over $m$, $p$, $k$, and $\ell$, what remains is a sum of terms, each of which is a vector field with components indexed by $i$.  Each of these vector fields is of the form 
$$  G^i_{(1)} \zeta^{-2} \partial_\ell \varrho_{(1)} \circ \Phi \, ,  $$
where $\ell\in\{1,\dots,n\}$. In the above expression, the vector field $G^i_{(1)}$ is a single term of the form 
$$   \partial_m \left(G^p A^i_k A^m_k - G^m A^i_k A^p_k\right)  \partial_p \Phi^\ell
- \partial_m G^i A^m_\ell  \, , $$
where $m,p,k,\ell\in\{1,\dots,n\}$. From \eqref{eq:inverse:div:DN:G} and \eqref{eq:diff:bounds}, we have that 
\begin{equation}\notag
\left\| D^N \left( \partial_m \left(G^p A^i_k A^m_k - G^m A^i_k A^p_k\right)  \partial_p \Phi^\ell
- \partial_m G^i A^m_\ell \right) \right\|_{L^1} \lesssim \calc_G \lambda^{N+1} \, 
\end{equation}
for all $N\leq 2\dpot-1$.  However, the $L^1$ norm of $\left(\zeta^{-2}\partial_\ell\varrho_{(1)}\circ\Phi\right)$ is $\zeta^{-1}\calc_*$ as calculated above. Hence we are in a scenario similar to that in which we started, except we have ``traded a cheap derivative for an expensive derivative" - i.e. lost a factor of $\lambda$ but gained a factor of $\zeta^{-1}$ when estimating the $L^1$ norm of $E_{(0)}^i$ relative to that of the original product $G \varrho \circ \Phi$. So applying Proposition~\ref{prop:intermittent:inverse:div} with the new functions $G^i_{(1)}$ and $\zeta^{-2}\partial_\ell \varrho_{(1)}\circ \Phi$ and estimating as before will produce a new symmetric stress $\RR_{(1)}$, pressure $P_{(1)}$, and error $E_{(1)}$.  The bounds for these terms will differ from those of their respective predecessors by a factor of $\lambda\zeta^{-1}$. Repeating this process $\dpot$ times produces an explicitly computable sum of the form
\begin{align}\label{eq:sum:sum}
    G^i \varrho\circ \Phi = \sum_{j=0}^{\dpot -1} \left( \partial_m  \RR^{im}_{(j)} + \partial_i P_{(j)} \right) + E^i_{(\dpot-1)} \, .
\end{align}
The final error term $E_{(\dpot-1)}^i$ has the same mean as the left-hand size of \eqref{eq:sum:sum} and has gained the factor of smallness $\left(\lambda\zeta\right)^{\dpot-1}$ relative to $E_{(0)}$, which had already gained one factor of $\lambda\zeta^{-1}$. Hence we may estimate the size of $E^i_{(\dpot-1)}$ in $L^2$ by 
$$  \calc_G \lambda^{4+\dpot} \calc_* \zeta^{-\dpot} \, .  $$
Applying $\mathcal{R}$ to $\mathbb{P}_{\neq 0}E^i_{(d-1)}$, appealing to its $L^2\hookrightarrow L^2$ bound, and utilizing \eqref{eq:exchange} concludes the proof.
\end{proof}

\section{Construction of the Perturbation}
We next define the intermittent Mikado flows used in this paper.  The following Proposition is part of Proposition 4.3 from \cite{bmnv21}, to which we refer the reader for a proof of everything except \eqref{item:pipe:3}, as well as further properties about intermittent Mikado flows which are not relevant to this paper.
\begin{proposition}[Construction and Properties of Intermittent Pipe Flows]\label{pipeconstruction}
Given a vector $\xi$ belonging to the one of the sets of rational vectors $\mathcal{K}_n\subset\mathbb{Q}^{3}$ from Lemma~\ref{lem:linear:algebra}, $r^{-1},\lambda \in \mathbb{N}$ with $\lambda r\in \mathbb{N}$, and large integers $N$ and $d$, there exist vector fields $\WW_{\xi,\lambda,r}:\mathbb{T}^3\rightarrow\mathbb{R}^3$ and implicit constants depending on $N$ and $d$ but not $\lambda$ or $r$ such that the following hold:
\begin{enumerate}
    \item\label{item:pipe:1} There exists $\varrho:\mathbb{R}^2\rightarrow\mathbb{R}$ which is radially symmetric and given by the iterated Laplacian $\Delta^d \vartheta =: \varrho$ of a radially symmetric potential $\vartheta:\mathbb{R}^2\rightarrow\mathbb{R}$ with compact support in a ball of radius $\frac{1}{4}$ such that the following holds. There exists $\UU_{\xi,\lambda,r}:\mathbb{T}^3\rightarrow\mathbb{R}^3$ such that
    $$\displaystyle{\curl \UU_{\xi,\lambda,r} = \xi \lambda^{-2d}\Delta^d \left(\vartheta_{\xi,\lambda,r}\right) = \xi \varrho_{\xi,\lambda,r} =: \WW_{\xi,\lambda,r}}\, , \qquad \xi \cdot \nabla \varrho_{\xi,\lambda,r} = \xi\cdot\nabla\vartheta_{\xi,\lambda,r} = 0 \, . $$
    \item\label{item:pipe:3} $\WW_{\xi,\lambda,r}$ is a stationary, pressureless solution to the Euler equations, i.e.
    $$  \Div \WW_{\xi,\lambda,r} = 0, \qquad \Div\left( \WW_{\xi,\lambda,r} \otimes \WW_{\xi,\lambda,r} \right) = 0 \, . $$
    In addition, $\WW_{\xi,\lambda,r}$ is a stationary solution to the Euler-$\alpha$ equations with an explicitly computable pressure\footnote{The pressure corresponding to a stationary solution of the Euler-$\alpha$ equations is not uniquely determined unless a mean-zero condition is imposed.} - see \eqref{eq:pressure:formula}.
    \item\label{item:pipe:5} For all $n\leq N$, 
    \begin{equation}\label{e:pipe:estimates:1}
    {\left\| \nabla^n\vartheta_{\xi,\lambda,r} \right\|_{L^p(\mathbb{T}^3)} \lesssim \lambda^{n}r^{\left(\frac{2}{p}-1\right)} }, \qquad {\left\| \nabla^n\varrho_{\xi,\lambda,r} \right\|_{L^p(\mathbb{T}^3)} \lesssim \lambda^{n}r^{\left(\frac{2}{p}-1\right)} }
    \end{equation}
    and
    \begin{equation}\label{e:pipe:estimates:2}
    {\left\| \nabla^n\UU_{\xi,\lambda,r} \right\|_{L^p(\mathbb{T}^3)} \lesssim \lambda^{n-1}r^{\left(\frac{2}{p}-1\right)} }, \qquad {\left\| \nabla^n\WW_{\xi,\lambda,r} \right\|_{L^p(\mathbb{T}^3)} \lesssim \lambda^{n}r^{\left(\frac{2}{p}-1\right)} }.
    \end{equation}
    Furthermore, each of the functions listed above is $\frac{\T^3}{\lambda r}$-periodic.
    \item\label{item:pipe:6} Let $\Phi:\mathbb{T}^3\times[0,T]\rightarrow \mathbb{T}^3$ be the periodic solution to the transport equation
\begin{subequations}\label{e:phi:transport}
\begin{align}
\partial_t \Phi + v\cdot\nabla \Phi &=0, \\
\Phi_{t=t_0} &= x\, ,
\end{align}
\end{subequations}
with a smooth, divergence-free, periodic velocity field $v$. Then
\begin{equation}\label{eq:pipes:flowed:1}
\nabla \Phi^{-1} \cdot \left( \WW_{\xi,\lambda,r} \circ \Phi \right) = \curl \left( \nabla\Phi^T \cdot \left( \mathbb{U}_{\xi,\lambda,r} \circ \Phi \right) \right).
\end{equation}
\end{enumerate}
\end{proposition}
\begin{proof}[Proof of Proposition~\ref{pipeconstruction}]
We refer to \cite{bmnv21} for the proofs of each item, save for the stationarity with respect to the Euler-$\alpha$ equations. Due to the presence of the Laplacian in the Euler-$\alpha$ equations, we must define the Mikado flows using a radially symmetric flow profile. We shall define the objects $\UU$, $\WW$, $\varrho$, and $\vartheta$ at unit scale and refer to \cite{bmnv21} for the concentration and periodizing. Thus, let $h:\R^+ \to \R $ be a compactly supported smooth function satisfying
\begin{equation}
    \label{eq:zero:mean}
    \int_0^{\infty} h(r) {r} \,  dr = 0 \,, \qquad h(r) = \left( \partial_{r}^2 + \frac{1}{r} \partial_r \right)^d H(r) \,  .
\end{equation}
Note that the second property asserts that in Cartesian coordinates, $h$ is the iterated Laplacian of $H$, and so the second property implies the first. Then define
\begin{equation}
    \label{def:radial:profile}
    \varrho(y_1, y_2) := h\left(\sqrt{y_1^2 + y_2^2}\right) \, , \qquad \vartheta(y_1,y_2) := H\left(\sqrt{y_1^2 + y_2^2}\right) \, .
\end{equation}
Furthermore, define the rotated version of \eqref{def:radial:profile} by 
\begin{equation}
\label{def:rotated:radial:profile}
    \varrho_k(x) := \varrho(k_1 \cdot x, k_2 \cdot x)\, , \quad x \in \R^3 \, ,
\end{equation}
where $(k_1, k_2, k)$ forms a rational orthonormal basis of $\R^3$. Define the radial Mikado flow associated to \eqref{def:rotated:radial:profile} as 
\begin{equation}
    \label{def:radial:Mikado}
    \WW_k(x) = \varrho_k(x) k. 
\end{equation}
We note that one may convert to cylindrical coordinates adapted to the support of $\varrho_k$; after doing so, it is clear that $\Delta u^j \nabla u^j$ is curl-free.  Since it is also mean-zero, it must be given by a pressure gradient.  It is easy to check that the other terms in \eqref{eq:euler:alpha} vanish for the radial Mikado flow, or are equal to a pressure gradient, showing the stationarity.  For the sake of completeness, we still compute explicitly the pressure associated with the radial flow profile, using the formulation \eqref{eq:euler:alpha}. 

We have that 
\begin{equation}
    \Div(\WW_k) = k \cdot \nabla \varrho_k = 0 \, .\notag
\end{equation}
A stationary solution of \eqref{eq:euler:alpha} must satisfy
\begin{equation}
\label{eq:stationary:euler:alpha}
    \curl \optwo \times u + \nabla p = 0 \, .
\end{equation}
We first note that 
\begin{align}
    \curl_x \WW_k \times \WW_k &= \curl_x(\varrho_k k) \times \varrho_k k \notag\\
    &= (-\partial_{y_1} \varrho(k_1\cdot x, k_2\cdot x) k_2 + \partial_{y_2} \varrho(k_1\cdot x, k_2\cdot x) k_1 
    ) \times \varrho_k k \notag\\
    &= - (k_1 \partial_{y_1} \varrho(k_1\cdot x, k_2\cdot x) \varrho_k + k_2 \partial_{y_2} \varrho(k_1\cdot x, k_2\cdot x) \varrho_k) \notag\\
    &= -\frac{1}{2} \nabla_x \varrho_k^2 \, . \notag
\end{align}
To compute the pressure arising from the term 
$$
\curl \Delta \WW_k \times \WW_k \, ,
$$
we will use polar coordinates adapted to the $k_1, k_2$ plane. Hence $ r = \sqrt{(k_1 \cdot x)^2 + (k_2 \cdot x)^2 }$ will represent the distance from the $k$ axis, and $\theta$ will measure rotation around the $k$ axis. We also define $g(y) = \Delta_y \varrho(y)$ and  $g_k(x) = g(k_1 \cdot x, k_2 \cdot x)$ so that 
$$
g_k(x) = g(k_1 \cdot x, k_2 \cdot x ) = \Delta_x( \varrho_k(x)) = (\Delta_y\varrho)(k_1 \cdot x, k_2 \cdot x)
$$
Using these conventions we have
\begin{align}
    \curl \Delta &\WW_k \times \WW_k \notag\\
    &= (-k_2 \partial_{y_1} g(k_1\cdot x, k_2\cdot x) + k_1 \partial_{y_2} g(k_1\cdot x, k_2\cdot x) ) \times \varrho_k k  \notag\\
   & = -k_1 \partial_{y_1} g(k_1\cdot x, k_2\cdot x) \varrho_k - k_2 \partial_{y_2} g(k_1\cdot x, k_2\cdot x) \varrho_k \notag\\
    &= -k_1 \cos \theta \partial_r \left( \partial_r^2 h\left(\sqrt{(k_1\cdot x)^2 + (k_2\cdot x)^2} \right) + \frac{1}{r} \partial_r h\left(\sqrt{(k_1\cdot x)^2 + (k_2\cdot x)^2} \right)  \right) h \left(\sqrt{(k_1\cdot x)^2 + (k_2\cdot x)^2} \right) \notag\\
    &\quad - k_2 \sin \theta \partial_r \left( \partial_r^2 h\left(\sqrt{(k_1\cdot x)^2 + (k_2\cdot x)^2} \right) + \frac{1}{r} \partial_r h\left(\sqrt{(k_1\cdot x)^2 + (k_2\cdot x)^2} \right)  \right)h\left(\sqrt{(k_1\cdot x)^2 + (k_2\cdot x)^2} \right) \notag \\
&= -e_r \partial_r \left( \partial_r^2 h\left(\sqrt{(k_1\cdot x)^2 + (k_2\cdot x)^2} \right) + \frac{1}{r} \partial_r h\left(\sqrt{(k_1\cdot x)^2 + (k_2\cdot x)^2} \right)  \right)h\left(\sqrt{(k_1\cdot x)^2 + (k_2\cdot x)^2} \right) \, , \notag
\end{align}
where we used that the Laplacian applied to radial functions is given as $\partial_r^2 + \frac{\partial_r}{r}$ to go from the second to the third equality. Note that   $h$ and its derivatives are evaluated at $\sqrt{(k_1 \cdot x)^2 + (k_2 \cdot x)^2} $. 
We can write this as 
$$
-e_r \partial_r \left( \partial_r^2 h + \frac{1}{r} \partial_r h  \right)h = -\nabla_x \int_0^{\sqrt{(k_1 \cdot x)^2 + (k_2 \cdot x)^2} } \partial_r \left( \partial_r^2 h + \frac{1}{r} \partial_r h  \right)h \, dr \, .
$$
Therefore, the pressure is a radial function given by
\begin{equation}
p\left(\sqrt{(k_1\cdot x)^2 + (k_2\cdot x)^2} \right) = \frac{\left(h\left(\sqrt{(k_1\cdot x)^2 + (k_2\cdot x)^2} \right)\right)^2}{2} - \alpha^2\int_0^{\sqrt{(k_1 \cdot x)^2 + (k_2 \cdot x)^2} } \partial_r \left( \partial_r^2 h + \frac{1}{r} \partial_r h \right)h \, dr \, . \label{eq:pressure:formula}
\end{equation}
\end{proof}

\begin{lemma}[Calculating the Averages]\label{lem:averages}
Let $\xi$ be a rational unit direction vector, and let $\WW_{\xi}=\varrho_\xi \xi$ be the unit scale, un-concentrated Mikado flow in the direction $\xi$ as in \eqref{def:radial:Mikado}. Then if $\left\| \nabla \varrho_\xi \right\|_{L^2(\mathbb{T}^3)}^2=\calc$, we have that
\begin{align}\label{eq:average}
\int_{\mathbb{T}^3} \WW_\xi^k \partial_{mm} \WW_\xi^\ell + \partial_k \WW_\xi^j \partial_\ell \WW_\xi^j = \frac{\calc}{2} \left( \delta^{k\ell} - 3\xi^k\xi^\ell \right) \, . 
\end{align}
\end{lemma}
\begin{proof}
Since $\xi$ is fixed throughout the proof, we simply write $\varrho$ rather than $\varrho_\xi$ to lighten the notation. Then to calculate the first integral, we write
\begin{align}
    \int_{\mathbb{T}^3} \WW_\xi^k \partial_{mm} \WW_\xi^\ell \, dx&= \int_{\mathbb{T}^3} \varrho(x\cdot\xi_1,x\cdot\xi_2) \partial_{mm} \left( \varrho(x\cdot\xi_1,x\cdot\xi_2) \right) \xi^k \xi^\ell \, dx  \notag\\
    &= - \xi^k \xi^\ell \int_{\mathbb{T}^3} \partial_p \varrho(x\cdot\xi_1,x\cdot\xi_2) \xi_{p}^m \partial_n \varrho(x\cdot\xi_1,x\cdot\xi_2) \xi_n^m \,dx \notag\\
    &= -\xi^k\xi^\ell \int_{\mathbb{T}^3} \delta_{np} \partial_p \varrho(x\cdot\xi_1,x\cdot\xi_2) \partial_n \varrho(x\cdot\xi_1,x\cdot\xi_2) \,dx \notag\\
    &= -\xi^k \xi^\ell \int_{\mathbb{T}^3} \partial_n \varrho(x\cdot\xi_1,x\cdot\xi_2) \partial_n \varrho(x\cdot\xi_1,x\cdot\xi_2) \,dx \notag\\
    &:= - \calc \xi^k \xi^\ell \, . \label{eq:average:one}
\end{align}
Note that in the above calculation, $\partial_{mm}$ denotes partial differentation computed with respect to the standard coordinate basis, while $\partial_p \varrho$ denotes the derivative of $\varrho$ with respect to $x\cdot \xi_p$. We have also used that 
$$ \calc = \left\| \nabla_{e_1,e_2,e_3} \varrho_\xi \right\|_{L^2(\mathbb{T}^3)}^2 = \int_{\mathbb{T}^3} \left((\partial_1 \varrho)^2 + (\partial_2 \varrho)^2\right)(x\cdot\xi_1,x\cdot\xi_2) \,dx \, .$$
For the second integral, we must calculate 
$$  \int_{\mathbb{T}^3} \partial_k \WW_\xi^j \partial_\ell \WW_\xi^j  = \int_{\mathbb{T}^3} \xi^j \,  \partial_n \varrho(x\cdot\xi_1,x\cdot\xi_2) \,  \xi_n^k \,  \xi^j \,  \partial_p \varrho(x\cdot\xi_1,x\cdot\xi_2) \,  \xi_p^\ell \, . $$
Since we are summing over both $n$ and $p$, we have four cases corresponding to 
$$(n,p)\in\{(1,1),(1,2),(2,1),(2,2)\} \, . $$
If $n=p=1$, then the expression is equal to
\begin{align}\label{eq:average:two}
     \int_{\mathbb{T}^3} \xi^j \,  \partial_1 \varrho(x\cdot\xi_1,x\cdot\xi_2) \,  \xi_1^k \,  \xi^j \,  \partial_1 \varrho(x\cdot\xi_1,x\cdot\xi_2) \,  \xi_1^\ell  &= \left\| \partial_1 \varrho \right\|_{L^2(\mathbb{T}^3)}^2 \xi_1^k \xi_1^\ell \, ,
\end{align}
and for $n=p=2$, we obtain
\begin{equation}\label{eq:average:three}
\left\| \partial_2 \varrho \right\|_{L^2(\mathbb{T}^3)}^2 \xi_2^k \xi_2^\ell \, .
\end{equation}
However, since $\WW_\xi$ is in fact radial around its axis, we have that by the assumption that $\left\| \nabla \varrho \right\|_{L^2}^2=\calc$,
\begin{equation}\label{eq:average:four}
\left\| \partial_1 \varrho \right\|_{L^2(\mathbb{T}^3)}^2 = \left\| \partial_2 \varrho \right\|_{L^2(\mathbb{T}^3)}^2 = \frac{\calc}{2} \,.
\end{equation}
Moving to the cases $n\neq p$, for example $n=1$ and $p=2$, the expression becomes 
\begin{align}\label{eq:pipe:zero}
     \int_{\mathbb{T}^3} \xi^j \,  \partial_1 \varrho(x\cdot\xi_1&,x\cdot\xi_2) \,  \xi_1^k \,  \xi^j \,  \partial_2 \varrho(x\cdot\xi_1,x\cdot\xi_2) \,  \xi_2^\ell \, .
\end{align}
Consider a single cross section of the pipe with diameter $2d$; in cylindrical coordinates, this is the set 
$$  \{z=z_0,\, r \leq d, \, 0 \leq \theta \leq 2\pi \} \, . $$
The coordinate transformations for partial derivatives in cylindrical (or polar) coordinates for functions which do not depend on $\theta$ are 
$$  \partial_1 \rightarrow \cos(\theta) \partial_r, \qquad \partial_2 \rightarrow \sin(\theta) \partial_r \, . $$
Since $\varrho$ does not depend on $\theta$, the integrand over this cross-section of pipe becomes
$$ \xi_1^k \xi_2^\ell \int_0^{2\pi} \int_0^d \cos(\theta)\partial_r \varrho(r) \sin(\theta) \partial_r \varrho(r) r \,dr\,d\theta = 0 \, , $$
since $\cos(\theta)\sin(\theta) = \sfrac{1}{2}\sin(2\theta)$ integrates to $0$ from $0$ to $2\pi$, which shows that the entire term in \eqref{eq:pipe:zero} vanishes.

Finally, we have that
$$\xi\otimes\xi+\xi_1\otimes\xi_1+\xi_2\otimes\xi_2=\textnormal{Id}, \qquad \xi_1^k\xi_1^\ell+\xi_2^k\xi_2^\ell= \textnormal{Id}-\xi^k\xi^\ell \, $$
for any orthonormal basis $\{\xi,\xi_1,\xi_2\}$ of $\mathbb{R}^3$. Combining this with \eqref{eq:average:one}, \eqref{eq:average:two}, \eqref{eq:average:three}, and \eqref{eq:average:four}, we deduce the equality in \eqref{eq:average}.
\end{proof}
\subsection{Definition of and estimates on the perturbation}

We define a compactly supported temporal cutoff function $\eta:(-1,1)\rightarrow[0,1]$ which induces a $C^\infty$ partition of unity of $\mathbb{R}$ according to
\begin{equation}\label{eq:time:partition}
  \sum_{i\in\mathbb{Z}} \eta^2 (\cdot - i ) \equiv 1 \, .
\end{equation}
To set the scale for our cutoffs we define a parameter $\tau_q$ as
\begin{equation}
    \label{def:tau}
    \tau_q := \ell^3\, .
\end{equation}
Letting $t_i = i \tau_q$ for $i\in\mathbb{Z}$, we define the rescaled and translated cutoff functions by
\begin{equation}\label{eq:eta:q:i}
    \eta_{q,i}(t) = \begin{dcases}
\eta\left(\tau_q^{-1}(t-t_i)\right) & \supp_t \RR_q \cap \left[t_i-\tau_q,t_i+\tau_q\right] \neq \emptyset \\
0 & \textnormal{otherwise} \, .
\end{dcases} 
\end{equation}
We have that $\eta_{q,i}$ satisfies the properties
\begin{equation}\label{eq:eta:q:i:props}
    \sum_{i\in\mathbb{Z}} \eta_{q,i}^2(t) \equiv 1 \quad  \forall t \in \supp_t \RR_q \, , \qquad \left| \partial_t^m \eta_{q,i} \right| \lesssim \tau_q^{-m} \,, \qquad \eta\qi \eta\qiprime \neq 0 \implies |i-i'|\leq 1 \, .
\end{equation}
Define $\Phi_{q,i}$ to be the solution to the transport equation
\begin{subequations}\notag
\begin{align}
  \left(\partial_t + u_\ell \cdot \nabla\right) \Phi_{q,i} &= 0 \\
   \Phi_{q,i}(t_i,x) &= x \, .
\end{align}
\end{subequations}

\begin{lemma}[Deformation bounds]\label{lem:deformation}
The following estimates hold for the deformation maps $\Phi_{q,i}$:
\begin{enumerate}
    \item For each $t\in\mathbb{R}$, $\Phi_{q,i}(t,\cdot):\mathbb{T}^d\rightarrow\mathbb{T}^d$ is a diffeomorphism, and we denote the (time-dependent) inverse map by $X_{q,i}$, so that $\Phi_{q,i}(t,\cdot)\circ X_{q,i}(t,\cdot)=\Id$ for all $t\in\mathbb{R}$.
    \item For all $t\in \supp \eta_{q,i}$ and $N\geq 2$, 
    \begin{subequations}
    \begin{align}
        \left\| \nabla \Phi_{q,i}(t) - \Id \right\|_{0} \, ,         \left\| \nabla X_{q,i}(t) - \Id \right\|_{0} &\leq {\ell} \label{eq:flow:one} \\
        \left\| \nabla^N \Phi_{q,i}(t) \right\|_0\, , \left\| \nabla^N X_{q,i}(t) \right\|_0 &\lesssim \ell^{-(N-1)} \label{eq:flow:two} \\
        \left\| \nabla^{N-2} \partial_t \Phi_{q,i} \right\|_0\, , \left\| \nabla^{N-2} \partial_t X_{q,i} \right\|_0 &\lesssim \ell^{-(N+1)} \, . \label{eq:flow:three}
    \end{align}
    \end{subequations}
\end{enumerate}
\end{lemma}
\begin{proof}
The fact that $\Phi_{q,i}$ is a diffeomorphism follows from the divergence-free property of the vector field $u_\ell$, and then automatically $\Phi_{q,i}^{-1}=X_{q,i}$ is a diffeomorphism as well. The proof of \eqref{eq:flow:one} for $\Phi(t)$ follows precisely the proof of Proposition D.1, (135) in \cite{BDLISZ15}.  We use \eqref{eq:inductive} and  Young's inequality to deduce that $\| \nabla u_{\ell}\|_0 \lesssim \ell^{-\frac{3}{2}}$.  {The choice of $\tau_q=\ell^3$} then implies  that the deviation of $\nabla\Phi_{q,i}$ from the identity is bounded by $\ell^{-\frac{3}{2}}\cdot\ell^3\leq \ell$.  The same estimate for $\nabla X$ follows from the inverse function theorem. The estimate \eqref{eq:flow:two} for $\Phi_{q,i}$ follows \cite{BDLISZ15} again, while the same estimate for $X_{q,i}$ follows from the Fa'a di Bruno formula and the fact that $\nabla X_{q,i}$ is a $C^\infty$ function of the entries of $\nabla \Phi_{q,i}$. The final estimate in \eqref{eq:flow:three} for $\Phi\qi$ follows from writing $\nabla^{N-2}\partial_t \Phi_{q,i} = \nabla^{N-2}\left(u_\ell\cdot\nabla\Phi_{q,i}\right) $ and utilizing the spatial derivative bounds on $u_\ell$ and $\Phi\qi$ and standard H\"older estimates for products. The bound for $\partial_t X\qi$ follow from the fact that $\partial_t X\qi(t,x)=u_\ell(X\qi(t,x),t)$ and repeated application of the chain rule (or Fa'a di Bruno formula); see for example Proposition A.1, estimate A.5 from \cite{BDLSV17}. 
\end{proof}

Define the cutoff function $\chi(z):[0,\infty)\rightarrow\mathbb{R}$ to be a smooth function satisfying
\begin{equation}\chi(z) = 
\begin{dcases}
1 & 0 \leq z \leq 1 \\
z & z \geq 2
\end{dcases} \, , \qquad z \leq 2\chi(z) \leq {4z} \quad \forall z\in(1,2) \, . \label{eq:def:chi}
\end{equation}
Recall that in Lemma~\ref{lem:linear:algebra}, we chose a value of $N\in \mathbb{N}$, which in turn sets a value of $\varepsilon$.  We choose $N=2$ and let $\varepsilon$ be as in the statement of the Lemma. We define
\begin{equation}\label{eq:def:rho}
    \rho_{q}(x,t) = \left(2 \delta_{q+1} {\lambda_{q+1}^2}\varepsilon^{-1} \calc_\RR\right) \cdot \chi \left( (\calc_\RR \delta_{q+1} {\lambda_{q+1}^2}\alpha^2)^{-1} \left| \RR_{\ell}(x,t) \right| \right)  \, .
\end{equation}
Simple calculations give that
\begin{equation}\label{eq:R:rho}
  \frac{\left|\alpha^{-2}\RR_\ell \right|}{\rho_{q}(x,t)} \leq \varepsilon \, , \qquad  \left\| \rho_{q} \right\|_{L^p} \leq \frac{3}{\varepsilon} \left( \mathcal{C}_{\RR} (8 \pi^3)^{\frac{1}{p}} \delta_{q+1}\lambda_{q+1}^2  + \left\| \alpha^{-2}\RR_\ell \right\|_{L^p} \right) \, .
\end{equation}

For ease of notation, let us define
\begin{equation}\label{eq:rell:def}
R_\ell = \alpha^{-2} \frac{\RR_\ell}{\rho_q} \, .    
\end{equation}
From \eqref{eq:R:rho} we have that $c_k^n\left(R_\ell\right)$ is well-defined for $1\leq k \leq 9$, $n=0,1$, and $i\in\mathbb{Z}$. After referring to Proposition~\ref{pipeconstruction} to set the notations
\begin{equation}\label{eq:flow:prof}
    \WW_{q+1,k} := \WW_{k,\lambda_{q+1},r_q} \, , \qquad \varrho_{q+1,k}:=\varrho_{k,\lambda_{q+1},r_q} \, ,
\end{equation}
we may then define the principal part of the perturbation
\begin{align}\label{eq:princ}
    w_{q+1}^{(p)} &= {\frac{1}{\lambda_{q+1}}} \sum_{i\in 2\mathbb{Z}} \sum_{k\in \mathcal{K}_0} c_k^{0}\left(R_\ell\right) \eta_{q,i}(t) \rhohalf_{q}(x,t) \left( \nabla \Phi_{q,i}(x,t) \right)^{-1} \WW_{q+1,k} \circ \Phi_{q,i}(x,t)  \notag\\
    &\qquad + {\frac{1}{\lambda_{q+1}}} \sum_{i\in 2\mathbb{Z}+1} \sum_{k\in \mathcal{K}_1} c_k^{1}\left(R_\ell\right) \eta_{q,i}(t) \rhohalf_{q}(x,t) \left( \nabla \Phi_{q,i}(x,t) \right)^{-1} \WW_{q+1,k} \circ \Phi_{q,i}(x,t) \, ,
\end{align}
where as indicated in \eqref{eq:flow:prof}, $\WW_{q+1,k}$ is an intermittent Mikado flow with parameter choices $\xi=k$, $\lambda=\lambda_{q+1}$, $r=r_q=\left(\frac{\lambda_q}{\lambda_{q+1}}\right)^\Gamma$. 
We may also define the divergence corrector by
\begin{align}
     w_{q+1}^{(c)} &= {\frac{1}{\lambda_{q+1}}} \sum_{i\in 2\mathbb{Z}} \sum_{k\in \mathcal{K}_0} \nabla \left(c_k^{0}\left(R_\ell\right) \eta_{q,i}(t) \rhohalf_{q}(x,t)\right) \times \left( \left( \nabla \Phi_{q,i}(x,t) \right)^{T} \ \UU_{q+1,k} \circ \Phi_{q,i}(x,t)  \right) \notag\\
    &\qquad + {\frac{1}{\lambda_{q+1}}} \sum_{i\in 2\mathbb{Z}+1} \sum_{k\in \mathcal{K}_1} \nabla\left(c_k^{1}\left(R_\ell\right) \eta_{q,i}(t) \rhohalf_{q}(x,t)\right) \times \left( \left( \nabla \Phi_{q,i}(x,t) \right)^{T}  \UU_{q+1,k} \circ \Phi_{q,i}(x,t) \right) \, . \notag
\end{align}
We then define the entire perturbation as
\begin{align}
     \label{eq:full:perturbation}
     w_{q + 1} &= w_{q + 1}^{(p)} + w_{q + 1}^{(c)} \notag\\
     &= {\frac{1}{\lambda_{q+1}}} \sum_{i\in\mathbb{Z}}\sum_{k\in \mathcal{K}} \curl \left( c_k\left(R_\ell\right) \eta_{q,i}(t) \rhohalf_{q}(x,t) \left( \nabla \Phi_{q,i}(x,t) \right)^{T}  \UU_{q+1,k} \circ \Phi_{q,i}(x,t)\right) \, ,
\end{align}
where in a slight abuse of notation, we have condensed the separate sums over even and odd integers into a single sum over \emph{all} integers, while suppressing the superscript on $c_k$ and the subscript on $\mathcal{K}$. Notice that by the definition of $\eta_{q,i}$ in \eqref{eq:eta:q:i} and the assumption on the temporal support of $\RR_q$ from \eqref{eq:supp:Rq}, we have that 
\begin{equation}\label{eq:supp:w}
    \supp_t w_{q+1} \subset (t_1-2\tau_q^{-1},t_2+2\tau_q^{-1}) \, .
\end{equation}

\begin{lemma}[Estimates on the perturbation]\label{lem:estimates}
The perturbation $w_{q+1}$ as defined in \eqref{eq:full:perturbation} satisfies the following estimates.
\begin{enumerate}
    \item For $j\geq 1 $, we have that
      {\begin{subequations}
    \begin{align}
      \left\| \rho \right\|_0 \lesssim \delta_{q+1} \lambda_{q+1}^2\ell^{-3} \, ,  &\qquad  \left\|  \nabla_{x,t}^j \rho \right\|_0 \lesssim \delta_{q+1} \lambda_{q+1}^2 \ell^{-4j} \label{eq:rho:bound:one}\\
        \left\| \rhohalf \right\|_0 \lesssim \delta_{q+1}^{\sfrac{1}{2}} \ \lambda_{q+1} \ell^{-2} \, ,  &\qquad \left\| \nabla_{x,t}^j \rhohalf \right\|_0 \lesssim \delta_{q+1}^{\sfrac{1}{2}} \lambda_{q+1}\ell^{- 5j}  \, . 
        \label{eq:rho:bound:two}
    \end{align}
    \end{subequations}}
    \item For $j\geq 1$, we have that 
      {\begin{subequations}
    \begin{align}
    \left\| c_k\left(R_\ell\right)\eta_{q,i}\rhohalf_q \right\|_{L^2} & \leq \frac{1}{2} \delta_{q+1}^{\sfrac{1}{2}} \lambda_{q+1} \, , \label{eq:a:bound:one}\\
    \left\| \nabla_{x,t}^j\left( c_k\left(R_\ell\right)\eta_{q,i}\rhohalf_q \right) \right\|_{0} &\lesssim \delta_{q+1}^{\sfrac{1}{2}} \lambda_{q+1}\ell^{-5j-2} \lesssim \ell^{-5j -2}\, . \label{eq:a:bound:two}
    \end{align}
    \end{subequations}}
    \item For   {$j\geq 0$}, we have that
    \begin{subequations}
    \begin{align}
        \left\| w_{q+1}^{(p)} \right\|_{L^2} + \lambda_{q+1}^{-1} \left\| \nabla w_{q+1}^{(p)} \right\|_{L^2} + \lambda_{q+1}^{-2} \left\| \nabla^2 w_{q+1}^{(p)} \right\|_{L^2} &\leq \frac{1}{2}\delta_{q+1}^{\sfrac{1}{2}} \label{eq:wp} \\
        \left\| w_{q+1}^{(c)} \right\|_{L^2} + \lambda_{q+1}^{-1} \left\| \nabla w_{q+1}^{(c)} \right\|_{L^2} + \lambda_{q+1}^{-2} \left\| \nabla^2 w_{q+1}^{(c)} \right\|_{L^2} &\leq \frac{1}{2}\delta_{q+1}^{\sfrac{1}{2}} \left(\ell^8 \lambda_{q+1}\right)^{-1} \label{eq:wc}\\
        \left\| w_{q+1}^{(p)} \right\|_{\dot{C}^j} + \left\| w_{q+1}^{(c)} \right\|_{\dot{C}^j} &\leq  \lambda_{q+1}^{j+2} \, . \label{eq:wlazy}
    \end{align}
    \end{subequations}
\end{enumerate}
\end{lemma}
\begin{proof}
\eqref{eq:rho:bound:one} and \eqref{eq:rho:bound:two} follow from Young's inequality, \eqref{eq:inductive}, and standard H\"older estimates for compositions, cf. \cite{BDLSV17} Appendix A.1.  \eqref{eq:a:bound:one} follows from \eqref{eq:inductive}, \eqref{eq:R:rho}, \eqref{eq:eta:q:i:props},  the {definition of $c_j(R_\ell)$ as a smooth bounded function}, and a sufficiently small choice of $\calc_\RR$. \eqref{eq:a:bound:two} follows from {$\tau = \ell^3 $},  \eqref{eq:rho:bound:one},  \eqref{eq:rho:bound:two}, standard H\"older estimates, and {$\ell \leq \lambda_{q+1}^2 \delta_{q+1}$}, which is inequality \eqref{ineq:two}.

To prove \eqref{eq:wp} we fix an arbitrary time $t_0 \in \supp_t w_{q+1}$. At $t_0$, at most two of the $\eta_i$ are nonzero, so it is enough to focus on bounding each term in the sum over $i$ in \eqref{eq:princ} individually.  We note that from \eqref{eq:a:bound:two}, \eqref{eq:flow:one} \eqref{eq:flow:two}, \eqref{e:pipe:estimates:2}, the chain and product rule, and the ordering {$\ell^{-12} \ll \lambda_{q+1}$} from inequality \eqref{ineq:three}, it is enough to consider the case where the derivatives all land on $\WW \circ \Phi_{q, i}$. Before we can apply Lemma \ref{lem:Lp:independence}, we perform the change of variables $y := \Phi_{q,i}(x)$: doing so converts $\nabla^j\WW \circ \Phi_{q, i}$ to $\nabla^j \WW$ which is a genuinely $\left(\mathbb{T}/ (\lambda_{q+1}r_q)\right)^3 $ periodic function. Then we let $g = \nabla^j \WW$ and $f$ be the product of {$\lambda_{q+1}^{-1}$},   $c_k^{i \mod 2}\left(R_\ell\right) \eta_{q,i}(t) \rhohalf_{q}(x,t) \left( \nabla \Phi_{q,i} \right)^{-1}$ , $j$ copies of $\nabla \Phi_{q,i}$ (all composed with $X_{q,i}$) , and the Jacobian $\det \nabla X_{q,i}(y)$. Using \eqref{eq:a:bound:one}, \eqref{eq:a:bound:two}, Lemma \ref{lem:deformation}, and  H\"older estimates for compositions, we have that 
\begin{equation}
  \ell^{8j}  \| D^j f \|_0 \leq {c} \delta_{q+1}^{\sfrac{1}{2}} \, , \notag
\end{equation}
where ${c}$ is sufficiently small enough to absorb any dimensional constants; note that this is allowed because we have {used} an extra factor of $\ell$, and we may assume $\calc_\RR$ from \eqref{eq:inductive} is sufficiently small. From {the parameter inequality \eqref{ineq:four}, which states that
\begin{equation}
      \ell^{-8(\Ndec + 4 )} \leq \lambda_{q+1}^{\Ndec(1 - \Gamma - \frac{\Gamma}{b} )} \, , \notag
\end{equation}}
we may thus apply Lemma~\ref{lem:Lp:independence} with $\mu=\lambda_{q+1}r_q$ and $\lambda=\ell^{-8}$, producing \eqref{eq:wp}.

To prove $\eqref{eq:wc}$ we once again use the fact that it is enough to consider all the derivatives landing on $\UU_{\lambda_{q+1},\xi,r_q}$, where $\UU$ is the potential defined in Proposition~\ref{pipeconstruction}. In addition, by  \eqref{eq:a:bound:two} and  \eqref{item:pipe:5} from Proposition \ref{pipeconstruction} and a similar application of Lemma~\ref{lem:Lp:independence}, the bound \eqref{eq:wc} follows. Finally, to prove \eqref{eq:wlazy} we simply use \eqref{item:pipe:5} from Proposition \ref{pipeconstruction} and inequality~\eqref{ineq:five}, which implies that {$r_q^{-1} \ll \lambda_{q+1}$.}
\end{proof}

\subsection{The new equation at level \texorpdfstring{{$q+1$}}{qplusone}}

We define 
\begin{equation}
    \label{eq:u:q+1}
    u_{q + 1} :=  u_{\ell} + w_{q+1}.
\end{equation}
Adding $w_{q+1}$ to \eqref{eq:velocity:mollified:equation:statement}, we find that $u_{q+1}$ satisfies
\begin{align}
    \partial_t(u_{q+1} - &\alpha^2 \Delta u_{q+1} )^l + \partial_k\left( u_{q+1}^k (u_{q+1}^l - \alpha^2 \Delta u_{q+1}^l) - \alpha^2 \partial_k u_{q+1}^j \partial_l u_{q+1}^j  \right) + \partial_l p_\ell \notag\\
    &= \partial_k \RR_{\text{comm}}^{kl} \notag\\
    &\quad +\partial_k \left( \RR_\ell^{kl} + w_{q+1}^k\left(w_{q+1}^l-\alpha^2\Delta w_{q+1}^l\right) - \alpha^2 \partial_k w_{q+1}^j \partial_l w_{q+1}^j \right)\label{eq:osc}\\
    &\quad +\partial_t (w_{q + 1} - \alpha^2 \Delta w_{q+1})^l  + u_{\ell} \cdot \nabla (w_{q + 1} - \alpha^2 \Delta w_{q+1})^l \label{eq:transport}\\
    &\quad + \partial_k \left(w_{q+1}^k \left(u_\ell^l - \alpha^2 \Delta u_\ell^l\right) -\alpha^2 \partial_k u_\ell^j \partial_l w_{q+1}^j - \alpha^2 \partial_k w_{q+1}^j\partial_l u_\ell^j \right) \label{eq:Nash} \\
    &=:  \partial_k \RR_{\text{comm}}^{kl} + \partial_k\RR_{\text{osc}}^{kl} \, + \, \partial_k\RR_{\text{transport}}^{kl} \, + \, \partial_k \RR_{\text{Nash}}^{kl} \, .
\end{align}
{We do not need to address the commutator stress, and so we analyze the three main error terms in the next sections.}

\section{Stress Estimates}
With Proposition~\ref{prop:Celtics:suck} and Proposition~\ref{prop:weak:decoupling}, we are now ready to estimate the remaining stress terms from \eqref{eq:osc}-\eqref{eq:Nash}. In what follows we will use that $\mathcal{R} \circ \Div$ is bounded on $L^p$ for $p \in (1, \infty)$. In particular, $\mathcal{R} \circ \Div$ is \emph{not} bounded on $L^1$. To circumvent this we will choose a $p$ slightly larger than 1 and bound the stress in $L^p$ by interpolating between $L^1$ and $L^{\infty}$. This will be accounted for by a  parameter $\gamma$ satisfying $0 < \gamma \ll 1$ appearing as a power on $\lambda_{q+1}$.

\subsection{Oscillation Error}
The oscillation error was given in \eqref{eq:osc} by
$$ \partial_k\RR_{\text{osc}}^{kl} := \partial_k \left( \RR_\ell^{kl} + w_{q+1}^k\left(w_{q+1}^l-\alpha^2\Delta w_{q+1}^l\right) - \alpha^2 \partial_k w_{q+1}^j \partial_l w_{q+1}^j \right) \, .  $$
Eschewing index notation for the time being, we can split this sum as
\begin{align}
   \Div&\left( w_{q+1}\otimes w_{q+1} - \alpha^2 \left( w_{q+1}^{(c)}\otimes \Delta w_{q+1}^{(c)} + w_{q+1}^{(c)}\otimes \Delta w_{q+1}^{(p)} + w_{q+1}^{(p)}\otimes \Delta w_{q+1}^{(c)} \right) \right) \notag\\
   & \quad -\alpha^2 \Div \left(  \left(\nabla w_{q+1}^{(c)}\right)^T\nabla w_{q+1}^{(c)} + \left(\nabla w_{q+1}^{(p)}\right)^T\nabla w_{q+1}^{(c)} + \left(\nabla w_{q+1}^{(c)}\right)^T\nabla w_{q+1}^{(p)} \right) \notag\\
   &\quad - \alpha^2 \Div \left( w_{q+1}^{(p)}\otimes \Delta w_{q+1}^{(p)} + \left(\nabla w_{q+1}^{(p)}\right)^T\nabla w_{q+1}^{(p)} - \alpha^{-2}\RR_\ell \right) \, . \label{eq:hailstorm:1}
\end{align}
We estimate now the first two lines of the above display, which are all of lower order. We may bound $w_{q+1}\otimes w_{q+1}$ in $L^1$ by $\delta_{q+1}$ after appealing to \eqref{eq:wp} and \eqref{eq:wc} and in $L^\infty$ by $\lambda_{q+1}^2$. The subsequent terms on the first line, which we abbreviate $R_{\textnormal{corr},1}$ may be bounded in $L^1$ by $\delta_{q+1}\lambda_{q+1}^2\left(\ell^{8}\lambda_{q+1}\right)^{-1}$ after appealing to the higher order bounds in \eqref{eq:wp} and \eqref{eq:wc}, and in $L^\infty$ by $\lambda_{q+1}^6$ using \eqref{eq:wlazy}. The terms on the second line, which we abbreviate by $R_{\textnormal{corr},2}$, all obey analogous $L^1$ and $L^\infty$ bounds compared to the terms in $R_{\textnormal{corr},1}$. We obtain that
\begin{align}
    \left\| \mathcal{R} \circ \Div \left( w_{q+1}\otimes w_{q+1} + R_{\textnormal{corr},1} + R_{\textnormal{corr},2} \right) \right\|_{L^1} &= \left\| \mathcal{R} \circ \Div \left( \mathbb{P}_{\neq 0} \left( w_{q+1}\otimes w_{q+1} + R_{\textnormal{corr},1} + R_{\textnormal{corr},2}  \right) \right) \right\|_{L^1} \notag\\
    &\lesssim \delta_{q+1}\lambda_{q+1}^{\gamma} + \delta_{q+1}\lambda_{q+1}^{2+\gamma} \left(\ell \lambda_{q+1} \right)^{-1}\notag\\
    &\leq {\lambda_{q+2}^{2-\gamma}\delta_{q+2}}\, . \label{eq:corr:bound:one}
\end{align}
Therefore these errors are symmetric traceless stresses of the proper magnitude and may be absorbed into the new stress $\RR_{q+1}$. Note that in the last line, we appealed to inequality~\eqref{ineq:six}.

For the remainder of this section, we focus on the analysis of the third line
\begin{equation}
    \alpha^2 \Div \left( w_{q+1}^{(p)}\otimes \Delta w_{q+1}^{(p)} + \left(\nabla w_{q+1}^{(p)}\right)^T\nabla w_{q+1}^{(p)} - \alpha^{-2}\RR_\ell \right) \label{eq:third:redux} 
\end{equation}
from \eqref{eq:hailstorm:1}. Identical arguments to those which produced the bounds \eqref{eq:wc} imply that when the differential operators $\Delta$ or $\nabla$ from \eqref{eq:third:redux} land on the low-frequency coefficient functions or flow maps in the definition of $w_{q+1}^{(p)}$ in \eqref{eq:princ}, the resulting terms may be bounded by  \eqref{eq:corr:bound:one}. So we may immediately absorb these terms into $\RR_{q+1}$ and focus on the terms in which the differential operator falls on the high-frequency objects.

At this point we revert back to index notation and index the tensors by $\theta$ and $l$ rather than $k$ and $l$ (to avoid the overload of notation with vector directions $k$) and ignore the prefactor $\alpha^2$ in front of the last line of \eqref{eq:third:redux}.  Utilizing the notation from \eqref{eq:flow:prof} and the identities
\begin{align}
    \partial_m \left( \WW_{q+1,k}\circ\Phi\qi \right) &= k \partial_n \varrho_{q+1,k}(\Phi\qi) \partial_m \Phi\qi^n \notag\\
    \partial_{mm}\left(\WW_{q+1,k}\circ \Phi\qi \right) &= k\partial_{np} \varrho_{q+1,k}(\Phi\qi) \partial_m \Phi\qi^n \partial_m \Phi\qi^p + k\partial_n \varrho_{q+1,k}(\Phi\qi) \partial_{mm} \Phi\qi^n \, , \label{eq:shorthands}
\end{align}
we find that we must bound (after throwing away the lower order term from the Laplacian in \eqref{eq:shorthands})
\begin{align}
    &\frac{\partial_\theta}{\lambda_{{q+1}}^2} \left( \sum_{\substack{i,i'\in\mathbb{Z} \\ k\in\mathcal{K}_{i\, \textnormal{mod}\, 2} \\ k'\in\mathcal{K}_{i'\, \textnormal{mod}\, 2}}} \left(c_k c_{k'} \eta_{q,i} \eta_{q,i'} \rho_{q}^{\sfrac{1}{2}} \right)^2 \left(\nabla\Phi_{q,i}\right)^{-1}_{\beta\theta} k^\beta \varrho_{q+1,k}(\Phi\qi) \left(\nabla\Phi\qiprime\right)^{-1}_{jl}k^j \partial_{np}\varrho_{q+1,k'}(\Phi\qiprime) \partial_m \Phi\qiprime^n \partial_m \Phi\qiprime^p \right) \label{eq:osc:hail:1} \\
    &+ \frac{\partial_\theta}{\lambda_{{q+1}}^2} \left( \sum_{\substack{i,i'\in\mathbb{Z} \\ k\in\mathcal{K}_{i\, \textnormal{mod}\, 2} \\ k'\in\mathcal{K}_{i'\, \textnormal{mod}\, 2}}} \left(c_k c_{k'} \eta_{q,i} \eta_{q,i'} \rho_{q}^{\sfrac{1}{2}} \right)^2 \left(\nabla\Phi_{q,i}\right)^{-1}_{\beta j} k^\beta \partial_n\varrho_{q+1,k}(\Phi\qi) \partial_\theta\Phi\qi^n \left(\nabla\Phi\qiprime\right)^{-1}_{mj}k^m \partial_{p}\varrho_{q+1,k'}(\Phi\qiprime) \partial_l \Phi\qiprime^p \right) \,  \label{eq:osc:hail:two}\\
    &\qquad \qquad \qquad \qquad \qquad \qquad - \alpha^{-2} \partial_\theta \RR_\ell^{\theta l} \, . \notag
\end{align}
By the properties of the time cutoffs $\eta\qi$, the only nonzero terms in the above sum occur when $|i-i'|\leq 1$. The Type 1 oscillation error is the standard oscillation error in most convex integration schemes and is precisely the sum of $\alpha^{-2}\partial_\theta\RR_{\ell}^{\theta l}$ and the terms in the above sums in which $i=i'$; for such terms, we have that all terms besides those with $k=k'$ vanish, since Mikado flows which belong to the same set $\mathcal{K}_i$ but have different vector directions may be taken to be disjoint.\footnote{In two dimensions, these error terms do \emph{not} vanish but may be analyzed using Proposition~\ref{prop:weak:decoupling} in the same way as the terms for which $i\neq i'$.} These terms will be analyzed using Proposition~\ref{prop:intermittent:inverse:div}.  The Type 2 oscillation errors are the terms in the above sums for which $i\neq i'$; these terms will be analyzed using Proposition~\ref{prop:weak:decoupling}.

\subsubsection{Type 1 oscillation errors}
We first analyze the terms in \eqref{eq:osc:hail:1} for which $i=i'$, which (after a slight obfuscation of notation in which we denote the sets of vector directions with $\mathcal{K}$) may be written as
\begin{align} \notag
   \lambda_{q+1}^{-2}\partial_\theta \left( \sum_{i\in\mathbb{Z},k\in\mathcal{K}} c_k^2 \eta_{q,i}^2 \rho_{q} \left(\nabla\Phi_{q,i}\right)^{-1}_{\beta\theta} k^\beta \varrho_{q+1,k}(\Phi_{q,i}) \left(\nabla\Phi_{q,i}\right)^{-1}_{jl}k^j \partial_{np}\varrho_{q+1,k}(\Phi_{q,i}) \partial_m \Phi_{q,i}^n \partial_m \Phi_{q,i}^p \right) \, .
\end{align}
In order to simplify notation further in the second term above, for the time being we omit all subscripts $q$, $i$, and $k$ on the derivative matrices $\nabla\Phi\qi$ and flow profiles $\varrho_{q+1,k}$, and we will abbreviate $c_k^2 \eta_{q,i}^2 \rho_q$ with simply $a^2$. With these conventions, the above display may be rewritten as
\begin{align}
    \lambda_{q+1}^{-2}\partial_\theta \sum_{i,k} a^2 &\bigg{[} k^\theta k^l \varrho(\Phi) \partial_{mm} \varrho(\Phi) \label{eq:canc:hailstorm:one}\\
    &\quad + k^\beta \left( (\nabla\Phi)_{\beta\theta}^{-1} - \delta_{\theta\beta} \right) k^l \varrho(\Phi) \partial_{mm} \varrho(\Phi)\notag\\
    &\quad + k^\beta \left( \nabla\Phi\right)^{-1}_{\beta\theta}\left((\nabla\Phi)_{jl}^{-1}-\delta_{jl}\right)k^j \varrho(\Phi) \partial_{mm} \varrho(\Phi) \notag\\
    &\quad + k^\beta \left( \nabla\Phi\right)^{-1}_{\beta\theta}(\nabla\Phi)_{jl}^{-1} k^j \left( \partial_m \Phi^n - \delta_{mn} \right) \varrho(\Phi) \partial_{nm} \varrho(\Phi) \notag\\
     &\quad + k^\beta \left( \nabla\Phi\right)^{-1}_{\beta\theta}(\nabla\Phi)_{jl}^{-1} k^j \partial_m \Phi^n \left(\partial_m \Phi^p - \delta_{mp} \right) \varrho(\Phi) \partial_{np} \varrho(\Phi) \bigg{]} \, . \notag 
\end{align}
The first term above will be kept and used to cancel the stress. We must analyze the second through fifth terms.  The arguments for each are quite similar, and so we estimate the second and leave further details to the reader. Ignoring the $\partial_\theta$ but holding onto the rescaling $\lambda_{q+1}^{-2}$ for the time being, we recall \eqref{eq:flow:one}, \eqref{eq:a:bound:one}, \eqref{e:pipe:estimates:1} with $\lambda=\lambda_{q+1}$, $p=2,\infty$ and $r=\left(\frac{\lambda_q}{\lambda_{q+1}}\right)^\Gamma\geq \lambda_{q+1}^{-1}$, and \eqref{eq:a:bound:two}.  Then applying Lemma~\ref{lem:Lp:independence} in the same way that produced the bounds \eqref{eq:wp}, we may estimate the second term in $L^1$ by
\begin{align}\label{eq:hailstorm:est:one}
    \lambda_{q+1}^{-2}\left\| a^2 \left(\nabla\Phi^{-1}-\Id\right) \varrho(\Phi) \partial_{mm}\varrho(\Phi) \right\|_{L^1} \lesssim \lambda_{q+1}^{-2}\left\| a^2 \right\|_{L^1} \left\| \nabla\Phi^{-1} - \textnormal{Id} \right\|_{0} \left\| \varrho \Delta \varrho \right\|_{L^1} &\lesssim \ell \delta_{q+1} \lambda_{q+1}^2 \, ,
\end{align}
and in $L^\infty$ by
\begin{align}\label{eq:hailstorm:est:one:alt}
    \left\| a^2 \right\|_{0} \left\| \nabla\Phi - \textnormal{Id} \right\|_{0} \left\| \varrho \Delta \varrho \right\|_{0} \lesssim \ell^{-4} \ell \lambda_{q+1}^4 \, .
\end{align}
Applying $\mathcal{R}\circ\partial_\theta\circ \Pneq$, we may absorb this error term into $\RR_{q+1}$ after appealing to the parameter inequality~\eqref{ineq:seven}, which implies that
\begin{equation}\label{eq:telly:good}
   { \ell \delta_{q+1} \lambda_{q+1}^{2+\gamma} \leq \delta_{q+2} \lambda_{q+2}^{2-\gamma} \, .}
\end{equation}

Returning to \eqref{eq:canc:hailstorm:one}, we may rewrite the first line as
\begin{align}\label{eq:set:aside}
    \lambda_{q+1}^{-2}\partial_\theta \sum_{i,k} a^2 &\left( k^\theta k^l \Pneq \left( \varrho \partial_{mm} \varrho \right)(\Phi) + k^\theta k^l \Peq \left( \varrho \partial_{mm} \varrho \right) \right) \, .
\end{align}
The second portion of this term will be used to cancel the stress, and the first part is an error itself.  In order to analyze it, we have to telescope back to a term which has the correct factors of $\nabla\Phi$ so that we again have an approximately stationary solution to the equations.  Thus we rewrite the first term from the above display as
\begin{align}
    \lambda_{q+1}^{-2}\partial_\theta \sum_{i,k} a^2 &\bigg{[} k^\beta \left( \delta_{\theta\beta} - (\nabla\Phi)^{-1}_{\beta\theta} \right) k^l \Pneq \left( \varrho \partial_{mm} \varrho \right)(\Phi) \notag\\
    &\quad + k^\beta (\nabla\Phi)^{-1}_{\beta\theta} k^j\left(\delta_{jl} - (\nabla\Phi)^{-1}_{jl}\right)\Pneq \left( \varrho \partial_{mm} \varrho \right)(\Phi) \notag\\
    &\quad + k^\beta (\nabla\Phi)^{-1}_{\beta\theta} k^j (\nabla\Phi)^{-1}_{jl}  \Pneq \left( \varrho \partial_{mm} \varrho \right)(\Phi) \bigg{]} \, . \label{eq:hailstorm:telly}
\end{align}
The first and second of these terms are estimated in precisely the same fashion as the terms estimated in \eqref{eq:hailstorm:est:one}, and we omit further details. In the third term, the product rule gives that $\partial_\theta$ lands either on the low-frequency objects $a^2$ or $\nabla\Phi$ or the high-frequency, mean-zero object $\Pneq(\varrho\partial_{mm})$. When $\partial_\theta$ lands on the low-frequency portion, we then must estimate
\begin{equation}\label{eq:hail:low}
    \lambda_{q+1}^{-2}\sum_{i,k}\partial_\theta\left(a^2 k^\beta (\nabla\Phi)_{\beta\theta}^{-1} k^j (\nabla\Phi)^{-1}_{jl} \right) \Pneq\left(\varrho\partial_{mm}\varrho\right)(\Phi) \, .
\end{equation}
We apply the inverse divergence operator from Proposition~\ref{prop:intermittent:inverse:div} with the choices 
\begin{align}
v=u_\ell\,, \qquad G^l=\partial_\theta\left(a^2 k^\beta (\nabla\Phi)_{\beta\theta}^{-1} k^j (\nabla\Phi)^{-1}_{jl} \right)\,, \qquad\calc_G=\ell^{-9}\,, \qquad\lambda=\ell^{-5}\,, \qquad\Phi=\Phi\qi\,,\notag\\
\qquad\varrho=\Pneq\left(\varrho_{q+1,k}\partial_{mm}\varrho_{q+1,k}\right)\,, \qquad\vartheta=\left(\lambda_{q+1}r_q\right)^{2\dpot}\Delta^{-\dpot}\varrho\,,\notag\\
\qquad \zeta=\mu=(\lambda_{q+1}r_q)\,, \qquad \Lambda=\lambda_{q+1}\,, \qquad {\calc_*=\lambda_{q+1}^2}\, . \label{eq:choicez}
\end{align}
Then \eqref{eq:a:bound:two} shows that \eqref{eq:inverse:div:DN:G} is satisfied, \eqref{eq:flow:one} shows that \eqref{eq:diff:bounds} is satisfied, \eqref{eq:flow:two} shows that \eqref{eq:DDpsi} is satisfied, item~\eqref{item:inverse:i} is satisfied by the $\frac{\T^3}{\lambda_{q+1}r_q}$-periodicity guaranteed by item~\eqref{item:pipe:5} from Proposition~\ref{pipeconstruction}, item~\eqref{item:inverse:ii} and \eqref{item:inverse:iii} are satisfied by item~\eqref{item:pipe:5} from Proposition~\ref{pipeconstruction}, and \eqref{eq:exchange} is satisfied by {the parameter inequality~\eqref{ineq:eight}, which states that
\begin{equation}\notag
\left(\frac{\ell^{-5}}{\lambda_{q+1}r_q}\right)^\dpot \ell^{-4} \leq (\lambda_{q+1}r_q)^{-1}   \, .
\end{equation}
}
Since we may assume that all terms inside the operator $\partial_\theta$ in \eqref{eq:hailstorm:telly} have zero mean, and we shall see shortly that the complementary term to \eqref{eq:hail:low} in which the $\partial_\theta$ lands on the high-frequency object vanishes, the error term we are currently estimating may be taken to have zero mean.  Then applying \eqref{eq:inverse:div} and \eqref{eq:inverse:div:stress:1}, we have rewritten this error term as the sum of the divergence of a traceless symmetric stress $\RR$ and gradient of a pressure $P$ which obey the $L^1$ bounds (recall the factor of $\lambda_{q+1}^{-2}$ in front of the expression in \eqref{eq:hail:low})
{\begin{equation}\label{eq:div:good}
    \left\| \RR \right\|_{L^1} + \left\| P \right\|_{L^1} \lesssim \lambda_{q+1}^{-2} \calc_G \lambda^{4} \calc_* \zeta^{-1} = \lambda_{q+1}^{-2} \ell^{-9} \ell^{-20} \lambda_{q+1}^2 \left(\lambda_{q+1}r_q\right)^{-1} \leq \delta_{q+2} \lambda_{q+2}^{2-\gamma} \, .
\end{equation}}
In the last inequality, we have used inequality~\eqref{ineq:ten}.

We now must calculate what happens in \eqref{eq:hailstorm:telly} when the $\partial_\theta$ lands on the high frequency portion. By the product rule, we have two terms, which we may write out as
\begin{align}
    k^\beta (\nabla\Phi)^{-1}_{\beta\theta} k^j (\nabla\Phi)^{-1}_{jl}\Pneq \left( \partial_n\varrho \partial_{mm} \varrho \right)(\Phi) \partial_\theta \Phi^n &= k^n k^j (\nabla\Phi)_{jl}^{-1} \partial_n \varrho(\Phi) \partial_{mm}\varrho(\Phi) = 0 \, . \notag
\end{align}
and
\begin{align}\notag
    k^\beta (\nabla\Phi)^{-1}_{\beta\theta} k^j (\nabla\Phi)^{-1}_{jl}\Pneq \left( \varrho \partial_{nmm} \varrho \right)(\Phi) \partial_\theta \Phi^n &= k^n k^j (\nabla\Phi)_{jl}^{-1} \varrho(\Phi) \partial_{nmm}\varrho(\Phi) = 0 \, ,
\end{align}
{To conclude that these terms vanish, we have used that {$k^n \partial_n (D^N\varrho)=0$} for all $N\geq 0$, which follows from \eqref{item:pipe:1} of Proposition~\ref{pipeconstruction}.}  This concludes the analysis of the terms in \eqref{eq:osc:hail:1}, save for the low frequency terms which we set aside in \eqref{eq:set:aside} to cancel the stress later.

We now return to \eqref{eq:osc:hail:two} and handle the terms for which $i=i'$, which after adopting the same abbreviations as in the analysis of \eqref{eq:osc:hail:1} are written as
\begin{align}
    \lambda_{q+1}^{-2}\partial_\theta &\left[ \sum_{i,k} a^2 (\nabla\Phi)^{-1}_{\beta j}k^\beta \partial_n \varrho(\Phi) \partial_\theta \Phi^n (\nabla\Phi)^{-1}_{mj} k^m \partial_p \varrho(\Phi) \partial_l \Phi^p \right] \notag\\
    &= \lambda_{q+1}^{-2}\partial_\theta \left[ \sum_{i,k} a^2 \partial_\theta \varrho(\Phi) \partial_l\varrho(\Phi) \right] + T \, , \notag
\end{align}
where $T$ is a sum of telescoping terms exactly analogous to those in \eqref{eq:canc:hailstorm:one} and is estimated in exactly the same way, and we omit further details. Focusing on the first term after the equality in the above display, we again set aside the low-frequency portion to cancel the stress.  Then the high-frequency portion of this term is given by 
\begin{align}
    \lambda_{q+1}^{-2}\partial_\theta \left( \sum_{i,k} a^2 \Pneq \left( \partial_\theta \varrho \partial_l \varrho \right)(\Phi) \right) \, . \notag
\end{align}
Again by the product rule, the derivative may land on low-frequency objects as in \eqref{eq:hail:low}, and this term is fed into the inverse divergence exactly as before. We omit further details, as the only differences are that the high-frequency object $\Pneq(\varrho\partial_{mm}\varrho)$ from \eqref{eq:hail:low} has been replaced with $\Pneq(\partial_\theta\varrho\partial_l \varrho)$, and a few indices have been shifted in the definition of the low frequency object $G^l$ from \eqref{eq:choicez}.

The relatively more troublesome term arises when $\partial_\theta$ lands on the high-frequency portion $\Pneq \left( \partial_\theta \varrho \partial_l \varrho \right)(\Phi)$. We analyze this term as follows.  Since radial Mikado flows are stationary solution of the Euler-$\alpha$ equations (with pressure calculated in \eqref{eq:pressure:formula}), we have that
\begin{align}\notag
    \partial_\theta \left( \Pneq \left(\partial_\theta\varrho\partial_l\varrho\right) \right) = \partial_l P \, ,
\end{align}
where $P$ is a mean-zero pressure. We want to show that
\begin{align}\label{eq:hailstorm:simple}
    \partial_\theta \left( a^2 T_{\theta l} (\Phi) \right) = \partial_l \tilde P + \partial_\theta (\RR_{\theta l})
\end{align}
where $T_{\theta l} := \Pneq \left(\partial_\theta\varrho\partial_l\varrho\right)$. We now highlight the properties of $a^2$ and $T_{\theta l}$ which are analogous to those in \eqref{eq:choicez}, which later allow applications of the inverse divergence operator with choices essentially identical to those of \eqref{eq:choicez}.  The mean-zero tensor $T_{\theta l}$ is periodic to scale $\left(\lambda_{q+1}r_q\right)^{-1}$ and given by the iterated Laplacian of a potential (similar to the choices of $\varrho$ and $\vartheta$ in the analysis of \eqref{eq:hail:low}), has derivatives which cost $\Lambda=\lambda_{q+1}$, satisfies $\partial_\theta (T_{\theta l})=\partial_l P$, and has $L^1$ norm comparable to $\lambda_{q+1}^2$ (just as in the choice for $\calc_*$ from \eqref{eq:hail:low}). Note that by virtue of the equality $\partial_\theta T_{\theta l} = \partial_l P$ and the $L^1$ norm of $T_{\theta l}$, we have that
\begin{align}\label{eq:L1:p}
    \left\| P \right\|_{L^1} = \left\| \Delta^{-1}\partial_{ll} P \right\|_{L^1} =  \left\| \Delta^{-1}\partial_{l \theta} T_{\theta l} \right\|_{L^1} \leq \lambda_{q+1}^{2+\gamma} \, .
\end{align}
 The diffeomorphism $\Phi=\Phi\qi$ in \eqref{eq:hailstorm:simple} is the usual one, and the low-frequency coefficient function $a^2$ has $L^1$ norm given by a power of $\ell^{-1}$, and each subsequent derivative costs $\ell^{-5}$ (exactly like the term $G^l$ in the analysis of \eqref{eq:hail:low}).  The tensor $\RR_{\theta l}$ in \eqref{eq:hailstorm:simple} will be symmetric and traceless and, from the assumptions on $a^2$ and $T_{\theta l}$, satisfy estimates which allow it to be absorbed into $\RR_{q+1}$.

Feeding the term in \eqref{eq:hailstorm:simple} in which $\partial_\theta$ lands on the low frequency function $a$ into the inverse divergence from Proposition~\ref{prop:Celtics:suck} to obtain an error term $\RR_{\textnormal{low}}^{\theta l}$ and pressure $P_{\textnormal{low}}$, we have that
\begin{align}
    \partial_\theta \left( a^2 T_{\theta l}(\Phi) \right) &= \partial_\theta(a^2) T_{\theta l} (\Phi) + a^2 \partial_\theta \left(T_{\theta l}(\Phi)\right) \notag\\
    &=\partial_\theta \RR_\textnormal{low}^{\theta l} + \partial_l P_\textnormal{low} + \fint_{\T^3} \partial_\theta(a^2) T_{\theta l}(\Phi) + a^2 \partial_m T_{\theta l}(\Phi) \partial_\theta \Phi^m \notag\\
    &= \partial_\theta \RR_\textnormal{low}^{\theta l} + \partial_l P_\textnormal{low}+ \fint_{\T^3} \partial_\theta(a^2) T_{\theta l}(\Phi) + a^2 \partial_m T_{\theta l}(\Phi) \left( \partial_\theta \Phi^m - \delta_{\theta m} \right) + a^2 \left( \partial_\theta T_{\theta l} \right)(\Phi) \, . \label{eq:hailstorm:damnit}
\end{align}
The first two terms on the second line, which come from the application of the inverse divergence, are estimated entirely analogously to how \eqref{eq:hail:low} was estimated and satisfy analogous bounds due to the properties of $a$, $T_{\theta l}$, and $\Phi$. The third term, which contains the mean, will cancel with the means resulting from subsequent applications of the inverse divergence to the remaining terms, since the left-hand side of \eqref{eq:hailstorm:damnit} is zero mean. Therefore, it suffices to show these two terms are sufficiently small. The fourth term in the above display may be rewritten using the Piola identity and differentiation by parts as
\begin{align}
    a^2 \partial_m T_{\theta l}(\Phi)(\partial_\theta \Phi^m - \delta_{\theta m}) &= a^2 (\nabla\Phi)^{-1}_{mn} \partial_n \left( T_{\theta l}(\Phi) \right)(\partial_\theta \Phi^m - \delta_{\theta m}) \notag\\
    &=  \partial_n \left(  a^2 (\nabla\Phi)^{-1}_{mn}  T_{\theta l}(\Phi) (\partial_\theta \Phi^m - \delta_{\theta m}) \right) - \partial_n \left( a^2 (\nabla\Phi)^{-1}_{mn}   (\partial_\theta \Phi^m - \delta_{\theta m}) \right) T_{\theta l}(\Phi) \, . \label{eq:above}
\end{align}
The first term in \eqref{eq:above} is the divergence of a tensor indexed by $n$ and $l$ and satisfies bounds identical to those in \eqref{eq:hailstorm:est:one}, although it is not symmetric and traceless. Applying the same methodology that converted the telescoping term in \eqref{eq:canc:hailstorm:one} into a symmetric traceless stress, we see that this term may be absorbed into $\RR_{q+1}$ {with a bound given by \eqref{eq:telly:good}}.  The second term is of the exact same form as the term in \eqref{eq:hail:low} and may be fed into the inverse divergence from Proposition~\ref{prop:Celtics:suck} to produce bounds identical to those in \eqref{eq:div:good}.

The final term on the second line of \eqref{eq:hailstorm:damnit} may be rewritten using the Piola identity and differentiation by parts as
\begin{align}
    a^2 \left( \partial_\theta T_{\theta l} \right)(\Phi) &= a^2 (\partial_l P)(\Phi) \notag\\
    &= a^2 \partial_m \left( (\nabla\Phi)^{-1}_{lm} P\circ \Phi \right) \notag\\
    &= a^2 \partial_l (P \circ \Phi) + a^2 \partial_m \left( \left( (\nabla\Phi)^{-1}_{lm} - \delta_{lm}\right) P \circ \Phi \right) \notag\\
    &= \partial_l \left( a^2 P \circ \Phi \right) - \partial_l a^2 P \circ \Phi + \partial_m \left( a^2 \left( (\nabla\Phi)_{lm}^{-1} - \delta_{lm} \right) P \circ \Phi \right) - \partial_m a^2  \left( (\nabla\Phi)^{-1}_{lm} - \delta_{lm}\right) P \circ \Phi \notag\\
    &= \partial_l \left( a^2 P \circ \Phi \right) + \partial_m \left( a^2 \left( (\nabla\Phi)_{lm}^{-1} - \delta_{lm} \right) P \circ \Phi \right) - \partial_m a^2   (\nabla\Phi)^{-1}_{lm} P \circ \Phi  \, . \notag
\end{align}
The first term on the last line above is pressure.  The second term is the divergence of a tensor, but not a symmetric traceless one. We estimate it in the same fashion as the previous telescoping terms, noting along the way that the $L^1$ bound on $P$ costs an extra factor of $\lambda_{q+1}^{\gamma}$. The final term goes into the inverse divergence from Proposition~\ref{prop:Celtics:suck} and obeys similar bounds as before, albeit with an extra $\lambda_{q+1}^\gamma$ loss again due to the $L^1$ bound on $P$.

It remains to show that the low-frequency portions of the terms analyzed above cancel the stress.  We write that
\begin{align}
    \lambda_{q+1}^{-2} \partial_\theta &\left( \sum_{i} \eta_{q,i}^2(t) \sum_k c_k^2(R_\ell) \rho_q(x,t) \left( k^\theta k^l \Peq\left(\varrho \partial_{mm} \varrho\right) + \Peq\left(\partial_\theta \varrho \partial_l \varrho \right) \right) \right) - \alpha^{-2}\partial_\theta \RR_{\ell}^{\theta l} \notag\\
    &= \lambda_{q+1}^{-2} \partial_\theta \left( \sum_{i} \eta_{q,i}^2(t) \sum_k c_k^2(R_\ell) \rho_q(x,t) \left( k^\theta k^l \Peq\left(\varrho \partial_{mm} \varrho\right) + \Peq\left(\partial_\theta \varrho \partial_l \varrho \right) \right) \right)\notag\\
    &\qquad - \sum_{i} \eta_{q,i}^2(t) \alpha^{-2} \partial_\theta \RR_{\ell}^{\theta l} \notag\\
    &= \partial_\theta \sum_{i}\eta_{q,i}^2(x,t) \left( \sum_k c_k^2(R_\ell) \rho_q(x,t) \left(3k^\theta k^l - \delta_{\theta l}\right) - \alpha^{-2} \RR_\ell^{\theta l} \right) \notag\\
    &= \partial_\theta \sum_{i}\eta_{q,i}^2(x,t) \left(   \rho_q(x,t) \left(\alpha^{-2} \frac{\RR_\ell^{\theta l}}{\rho_q(x,t)}\right) - \alpha^{-2} \RR_\ell^{\theta l} \right) \notag\\
    &= 0 \, . \notag
\end{align}
{We have used Lemma \ref{lem:averages} and the normalization $\| \nabla \varrho_{q+1,k} \|_{L^2} =  \lambda_{q+1}$}, which is commensurate with the estimates in Proposition~\ref{pipeconstruction}.

\subsubsection{Type 2 oscillation errors}
From \eqref{eq:osc:hail:1} and \eqref{eq:osc:hail:two} and the ensuing discussion, the only non-zero terms for which $i\neq i'$ are those in which $|i-i'|\leq 1$. As only two time cutoffs are non-zero at any point in time, it therefore suffices to analyze a single term of the form
\begin{equation}
    \label{eq:osc:error:type2:worst:terms}
\mathcal{E}_{i}=a_{i}a_{i+1} \WW_{q+1,k} \circ \Phi_{q,i} \otimes  \Delta \left( \WW_{q+1,k'} \circ \Phi_{q,i+1} \right) +  a_i a_{i+1} \nabla \left( \WW_{q+1,k} \circ \Phi_{q,i} \right)^T \nabla \left( \WW_{q+1,k'} \circ \Phi_{q,i+1}\right) \, ,
\end{equation}
where we have used the simpler notation from \eqref{eq:shorthands} for the Mikado flows and the shorthand
\begin{equation}
    a_i:=a_{q,i,k} := \lambda_{q+1}^{-1} c_k^{i \text{ mod }2} \eta_{q ,i} \rhohalf\qi (\nabla \Phi_{q , i})^{-1} \, . \notag
\end{equation}
To estimate \eqref{eq:osc:error:type2:worst:terms}, we will use Proposition \ref{prop:weak:decoupling} with the following choices:
\begin{align}
\WW^1 = \WW_{q+1,k} \circ \Phi_{q,i}\, ,\qquad \WW^2 = \WW_{q+1,k'} \circ \Phi_{q,i+1}\, , \qquad v = u_{\ell}\, ,\qquad \Phi_1 = \Phi\qi \, , \qquad \Phi_2=\Phi_{q,i+1} \, , \notag\\ 
f = a_i a_{i+1}\, ,
\qquad r_1 = \frac{\ell^{-8}}{\lambda_{q+1}}\, , \qquad r_2 = r_q\, , \qquad {\mathcal{C}_f = \delta_{q+1} }\, . \notag
\end{align}
By the definition of $\mathcal{K}_0$ and $\mathcal{K}_1$ and Lemma~\ref{lem:linear:algebra}, we have that \eqref{e:inner:product:control} is satisfied. From Lemma \ref{lem:deformation}  we also have that \eqref{eq:deformation:bounds:decoupling} is satisfied for an appropriate choice of $\calc_\Phi$.  We now just need to check the bounds for $f$, which was defined as
$$
f = a_i a_{i+1} = \lambda_{q+1}^{-2}c_k^0 c_{k'}^1 \eta_{q ,i} \eta_{q,i+1}  \rho_{q, i}^{\sfrac{1}{2}}\rho_{q,i+1}^{\sfrac{1}{2}} (\nabla \Phi_{q , i})^{-1} (\nabla \Phi_{q , i+1})^{-1} \, .
$$
By the product rule \eqref{eq:a:bound:one}, and \eqref{eq:a:bound:two}, we have that 
\begin{equation}\notag
    \| \nabla^n f \|_{L^1} \leq  \delta_{q+1}  \ell^{-8n} \, ,
\end{equation}
and so \eqref{e:derivatives:on:f} is satisfied. By the inequality~\eqref{ineq:four}, we have that
\begin{equation}\notag
    {\ell^{-8(\Ndec + 4 )} \leq \lambda_{q+1}^{\Ndec(1 - \Gamma - \frac{\Gamma}{b} )}} \, ,
\end{equation}
and so \eqref{e:r:decoupling} is satisfied. Therefore, \eqref{e:weak:decoupling:conclusion} yields that
\begin{equation}
    \| \mathcal{E}_i \|_{L^1} \lesssim \delta_{q+1} r_q \lambda_{q+1}^2 \, . \label{eq:sharp:ell:one}
\end{equation}
Since $\mathcal{E}_i$ is not necessarily traceless and symmetric, we define the stress
\begin{equation}\notag
    R_{\textnormal{osc}, 2} = \sum_{i} \mathcal{R} \circ \Div \mathcal{E}_i \, .
\end{equation}
As before, we may interpolate between the sharp $L^1$ bound in \eqref{eq:sharp:ell:one} and a lossy $L^\infty$ bound to deduce that,
\begin{equation}
    \left\| R_{\textnormal{osc},2} \right\|_{L^1} \leq {\delta_{q+1}r_q \lambda_{q+1}^{2+\gamma} \leq \delta_{q+2}\lambda_{q+2}^{2-\gamma} \, ,}
\end{equation}
which concludes the analysis of the Type 2 error terms.  In the last line we have used inequality~\eqref{ineq:seven:prime}.

\begin{remark}[Different strategies for solving the ``intersection problem"]\label{rem:why}
Following \cite{bmnv21} (see section 2.5.2 for heuristics, or Proposition 4.8 for a precise statement), discretizing in space at scale $\lambda_q^{-1}$ and using $\left(\frac{\lambda_q}{\lambda_{q+1}}\right)^{\sfrac{3}{4}}$ worth of intermittency ensures that pipes from different transport maps do not intersect at all. The problem with this method is that it requires a large amount of intermittency.  This will not impede the maximization of the regularity in $L^2$ of a solution, but will cause extremely lossy $L^\infty$ estimates. So although this method prevents intersections entirely and would suffice for our current purposes, it appears to impede the successful resolution of the Onsager conjecture for the Euler-$\alpha$ equations.

A second option is that of De Lellis and Kwon from the paper \cite{dlk20}.  In this paper, the authors restrict the timescale to be much shorter than the inverse Lipschitz timescale. On this short timescale, they can also prevent intersections entirely. However, the heuristic Nash error bounds are \emph{already} sharp when utilizing the inverse Lipschitz timescale; see Remark~\ref{rem:nash:explanation}.  Shortening the timescale will damage these error terms and force the usage of more intermittency, thus preventing $b$ from approaching $1$. Thus it appears this technique also will not allow for a resolution of the Onsager conjecture for the Euler-$\alpha$ equations.

A third option would be the gluing technique of Isett \cite{Isett2018}. However, the gluing technique as employed by Isett requires global Lipschitz bounds which are not available in an intermittent convex integration scheme.  Since the usage of intermittency seems \emph{necessary} in order to obtain estimates for the transport error which allow $\nabla u\in L^2$, and there is no clear way to pass to the limit in the weak formulation of \eqref{eq:inductive:equation} unless $\nabla u \in L^2_{t,x}$, the gluing technique does not appear to be consistent with a convex integration scheme which produces a function satisfying a meaningful weak formulation of the equations. Based on these heuristics, none of the current methodologies for preventing intersections of deformed pipe flows will allow for a resolution of the Onsager conjecture for Euler-$\alpha$. Since Proposition~\ref{prop:weak:decoupling} requires no minimum amount of intermittency to function, no artificial restriction of the timescale, and no global Lipschitz bounds, it appears to be more amenable to sharper regularity estimates which would allow the scheme to reach the $L^3$-threshold.  
\end{remark}

\subsection{Nash Error}
Per \eqref{eq:Nash}, the Nash error is given by 
\begin{equation}\label{eq:Nash:redux}
   \partial_{\theta} \RR_{\text{Nash}}^{\theta l}:= \partial_\theta \left(w_{q+1}^\theta \left(u_\ell^l - \alpha^2\Delta u_\ell^l\right) -\alpha^2 \partial_\theta u_\ell^j \partial_l w_{q+1}^j - \alpha^2 \partial_\theta w_{q+1}^j\partial_l u_\ell^j \right) \, . 
\end{equation}
It is clear that the most costly terms in this expression are the second and third terms; more specifically, the instances in which the differential operator $\partial_\theta$ falls on the principal part of $\partial_l w_{q+1}^j$ or $\partial_\theta w_{q+1}^j$. For this reason we begin by analyzing one such term. In the second term in \eqref{eq:Nash:redux}, the most expensive term is then
\begin{align}
    \partial_\theta u_\ell^j \partial_\theta \partial_l \left( w_{q+1}^{(p)}\right)^j &= \partial_\theta u_\ell^j \partial_\theta \partial_l \left( \frac{1}{\lambda_{q+1}} \sum_{i\in\mathbb{Z}}\sum_{k\in\mathcal{K}} c_k(R_\ell) \eta_{q,i} \rho_q^{\sfrac{1}{2}}(x,t) (\nabla\Phi_{q,i})_{\beta j}^{-1} k^\beta \varrho_{q+1,k}(\Phi_{q,i}) \right) \, .\notag
\end{align}
We simplify again to the case in which the differential operators $\partial_\theta \partial_l$ fall on $\varrho_{q+1,k}$, which again will clearly be the most costly error term from the above expression.  Thus we are tasked with bounding
\begin{align}
     \partial_\theta u_\ell^j \left( \frac{1}{\lambda_{q+1}} \sum_{i\in\mathbb{Z}}\sum_{k\in\mathcal{K}} c_k(R_\ell) \eta_{q,i} \rho_q^{\sfrac{1}{2}}(x,t) (\nabla\Phi_{q,i})_{\beta j}^{-1} k^\beta \partial_{mn}\varrho_{q+1,k}(\Phi_{q,i}) \partial_\theta \Phi\qi^m \partial_l \Phi\qi^n \right) \, . \label{eq:Nash:messy}
\end{align}
We shall apply Proposition~\ref{prop:Celtics:suck} with the following choices, after fixing choices of $m,n\in\{1,2,3\}$:
\begin{align}
v=u_\ell\,,\qquad G^l=\partial_\theta u_\ell^j c_k(R_\ell) \eta\qi \rho_q^{\sfrac{1}{2}} (\nabla\Phi\qi)_{\beta j}^{-1} k^\beta \partial_\theta \Phi\qi^m \partial_l\Phi\qi^n \, , \qquad \calc_G = {\ell^{-7}} \notag\\
\lambda = \ell^{-5}\, , \qquad \Phi= \Phi\qi \, ,\qquad \varrho = \partial_{mn}\varrho_{q+1,k} \, ,\qquad \vartheta = \partial_{mn}\vartheta_{k,\lambda_{q+1},r_q} \, , \notag\\
\qquad \zeta = \Lambda = \lambda_{q+1} \, , \qquad \calc_* = {\lambda_{q+1}^2} r_q \, , \qquad \mu = \lambda_{q+1} r_q \, . \label{eq:choices:Nash}
\end{align}
Then from \eqref{eq:a:bound:two}, Lemma~\ref{lem:deformation}, and \eqref{eq:mollified:estimates}, we may bound $G^l$ by 
\begin{equation}
    \left\| \partial_\theta u_\ell^j \right\|_{0} \left\| c_k \rho_q^{\sfrac{1}{2}} \eta\qi \right\|_{0} \left\| (\nabla\Phi\qi)^{-1} \right\|_0 \left\| \nabla \Phi\qi \right\|_0^2 \lesssim \ell^{-7} \, .
\end{equation}
Furthermore, from \eqref{eq:a:bound:two}, the most expensive spatial derivatives arise when differentiating $c_k \rho_q^{\sfrac{1}{2}}\eta\qi$, and cost $\ell^{-5}$ from \eqref{eq:a:bound:two}. Thus we have satisfied \eqref{eq:inverse:div:DN:G} with our choices of $\calc_G$ and $\lambda$. The deformation bounds in \eqref{eq:diff:bounds} and \eqref{eq:DDpsi} follow from Lemma~\ref{lem:deformation}. The equality in \eqref{item:inverse:i} follows from applying $\partial_{mn}$ to both sides of the first equality in \eqref{item:pipe:1} from Proposition~\ref{pipeconstruction}. The desired periodicity in \eqref{item:inverse:ii} follows as well from \eqref{item:pipe:5} from Proposition~\ref{pipeconstruction}. Finally, since $\zeta=\Lambda=\lambda_{q+1}$, \eqref{eq:DN:Mikado:density} follows also from \eqref{item:pipe:5} from Proposition~\ref{pipeconstruction}. The parameter inequality in \eqref{eq:exchange} translates into
{\begin{equation}
    \left(\frac{\ell^{-5}}{\lambda_{q+1}}\right)^{\dpot} \ell^{-20} \leq \lambda_{q+1}^{-1} \, , \notag
\end{equation}}
which is inequality~\eqref{ineq:nine}. Applying \eqref{eq:inverse:div:stress:1} and recalling the factor $\frac{1}{\lambda_{q+1}}$ in front of \eqref{eq:Nash:messy}, we find that \eqref{eq:Nash:messy} satisfies \eqref{eq:inverse:div} for a traceless symmetric stress $\RR$ and pressure $P$ which satisfy the bound
\begin{equation}\label{eq:Nash:bound}
\left\| \RR \right\|_{L^1} + \left\| P \right\|_{L^1} \lesssim \lambda_{q+1}^{-1} \ell^{-7} \ell^{-20} \lambda_{q+1}^2 r_q \lambda_{q+1}^{-1} \leq {\delta_{q+2}\lambda_{q+2}^{2-\gamma}} \, ,
\end{equation}
which is a consequence of inequality~\eqref{ineq:eleven}.
Finally, since the entirety of the Nash error in \eqref{eq:Nash:redux} has mean zero since it is the divergence of a tensor, we may ignore the mean in equality \eqref{eq:inverse:div}. Furthermore, since we have bounded a term which had one spatial derivative landed on $u_\ell$ and two on the flow profile in the definition of $w_{q+1}^{(p)}$, all terms in \eqref{eq:Nash:redux} are either similar in structure and will satisfy an identical bound, or have \emph{less} derivatives landing on the flow profile and will thus satisfy better bounds.  We omit further details.

\begin{remark}[Usage of intermittency and the $L^3$ rigidity threshold]\label{rem:nash:explanation}
To the authors knowledge, the 3D Euler equations provide the only known example of a proven Onsager-type threshold between rigidity and flexibility.  In the proofs of flexibility from \cite{Isett2018}, \cite{BDLSV17}, estimates are made in $L^\infty$ rather than $L^3$ since no intermittency is employed, and so we compute the following $L^\infty$ bound for the Nash error term, which is then written as
\begin{equation}\label{eq:nash:heuristic}
    \left\| \Div^{-1}\left[ w_{q+1}^k \partial_k\left(u_q-\alpha\Delta u_q\right)^\ell -\alpha \partial_k \left( \partial_k u_q^j \partial^\ell w_{q+1}^j + \partial_k w_{q+1}^j \partial^\ell u_q^j \right) \right] \right\|_{L^\infty} \, .
\end{equation}
The first part of this error term is quite benign since no derivatives have landed on $w_{q+1}$.  The second part of the error term is less benign, since it is possible that all the derivatives can land on $w_{q+1}$.  The worst case scenario is that in which all the derivatives land on $w_{q+1}$.  Assuming that \emph{no} intermittency is being employed, we find that
\begin{align}
    \lambda_{q+1}^{-1} \delta_q^{\sfrac{1}{2}}\lambda_q \delta_{q+1}^{\sfrac{1}{2}}\lambda_{q+1}^2 &\leq \delta_{q+2} \lambda_{q+2}^2 \notag\\
    \iff \delta_{q+2}^{-1} \delta_{q+1}^{\sfrac{1}{2}} \delta_q^{\sfrac{1}{2}} &\leq \lambda_{q+2}^2 \lambda_{q+1}^{-1}\lambda_q^{-1} \notag\\
    \iff 2\beta b^2 - \beta b - \beta &\leq 2b^2 - b - 1 \notag\\
    \iff \beta &\leq 1  \, . \notag
\end{align}
We thus conclude that when estimating the Nash error, the apparent optimal $L^\infty$ bound would not allow the final solution to enjoy $C^1$ regularity. Any intermittency would \emph{weaken} this $L^\infty$ estimate, and a significant amount of intermittency would produce a bound which would be far from $C^1$ regularity, and likely also far from the sharp $L^3$ threshold.
\end{remark}

\subsection{Transport Error}
Per \eqref{eq:transport}, the transport error is given by 
\begin{align}
  \partial_{\theta} \RR_{\text{transport}}^{\theta l} :=  \partial_t (w_{q + 1} - \alpha^2 \Delta w_{q+1})^l  + u_{\ell} \cdot \nabla (w_{q + 1} - \alpha^2 \Delta w_{q+1})^l \, . \label{eq:transport:redux}
\end{align}
It is clear that the terms $(\partial_t + u_\ell \cdot \nabla)w_{q+1}$ and $(\partial_t + u_\ell \cdot \nabla)(\Delta w_{q+1}^{(c)})$ are lower order when compared with $(\partial_t + u_\ell \cdot \nabla)(\Delta w_{q+1}^{(p)})$, and so we shall bound only the latter and omit further details. Thus we are tasked with bounding
\begin{align}
   \left( \partial_t + u_\ell \cdot \nabla  \right) \left[ \Delta \left( \frac{1}{\lambda_{q+1}} \sum_{i\in\mathbb{Z}}\sum_{k\in\mathcal{K}} c_k(R_\ell) \eta_{q,i} \rho_q^{\sfrac{1}{2}}(x,t) (\nabla\Phi_{q,i})_{\beta l}^{-1} k^\beta \varrho_{q+1,k}(\Phi_{q,i}) \right) \right]  \, . \notag
\end{align}
As usual, the terms in which $\Delta$ lands anywhere besides $\rho_{q+1,k}$ are lower order, and so we shall simply bound
\begin{align}
   \left( \partial_t + u_\ell \cdot \nabla  \right)  \left( \frac{1}{\lambda_{q+1}} \sum_{i\in\mathbb{Z}}\sum_{k\in\mathcal{K}} c_k(R_\ell) \eta_{q,i} \rho_q^{\sfrac{1}{2}}(x,t) (\nabla\Phi_{q,i})_{\beta l}^{-1} k^\beta \partial_{mn}\varrho_{q+1,k}(\Phi_{q,i}) \partial_j\Phi\qi^m \partial_j \Phi\qi^n \right)   \, . \notag
\end{align}
After fixing choices of $m,n\in\{1,2,3\}$, we shall apply Proposition~\ref{prop:intermittent:inverse:div} with the following choices:
\begin{align}
    v= u_\ell \, , \qquad G^l = \left( \partial_t + u_\ell \cdot \nabla \right) \left( c_k \eta_{q,i} \rho_q^{\sfrac{1}{2}} (\nabla\Phi_{q,i})_{\beta l}^{-1} k^\beta \partial_j\Phi\qi^m \partial_j \Phi\qi^n \right) \, ,\qquad  \calc_G = {\ell^{-13}} \, , \notag\\
    {\lambda=\ell^{-5}} \, ,\qquad \Phi=\Phi\qi \, , \qquad \varrho=\partial_{mn} \varrho_{q+1,k} \, , \qquad \vartheta = \partial_{mn} \vartheta_{k,\lambda_{q+1},r_q} \, , \notag\\
    \zeta=\Lambda=\lambda_{q+1} \, , \qquad \calc_* = \lambda_{q+1}^2 r_q \, , \qquad \mu=\lambda_{q+1} r_q \, . \label{eq:choices:transport}
\end{align}
In the above choices, we have used that the material derivative of any function which has been precomposed with $\Phi\qi$ is zero, and so the material derivative falls on everything except $\partial_{mn}\rho_{q+1,k}(\Phi\qi)$. To bound $G^l$, we split up the material derivative into $\partial_t$ and $u_\ell \cdot \nabla$ and appeal to \eqref{eq:mollified:estimates}, \eqref{eq:a:bound:two} with $j=1$, \eqref{eq:eta:q:i:props}, and Lemma~\ref{lem:deformation} to bound $G^l$ by 
\begin{equation}
    (1+\left\| u_\ell\right\|_{0}) \left\| \nabla_{x,t} \left( c_k \eta\qi \rho_q^{\sfrac{1}{2}} \right) \right\|_{0} \left\| \nabla_{x,t}\left( \nabla\Phi\qi^{-1} \otimes \nabla \Phi\qi \otimes \nabla \Phi\qi \right) \right\|_0 \lesssim \ell^{-5}\cdot{\ell^{-7}} \cdot \ell^{-1} = \ell^{-13} \, . \notag
\end{equation}
Appealing to the same estimates, we see that spatial derivatives on $G^l$ cost at most $\ell^{-5}$, and so \eqref{eq:inverse:div:DN:G} is satisfied. The rest of the choices in \eqref{eq:choices:transport} are \emph{identical} to those in \eqref{eq:choices:Nash}, and so the rest of the assumptions in Proposition~\ref{prop:Celtics:suck} are satisfied. Thus \eqref{eq:inverse:div} gives a traceless symmetric stress $\RR$ and pressure $P$ satisfying bounds identical to those in \eqref{eq:Nash:bound}, save for {6 extra factors of $\ell^{-1}$ from the slightly worse choice of $\calc_G$.}  Since the highest order error terms satisfies a bound which matches the desired estimate on $\RR_{q+1}$, it only remains to address the mean of the transport error.  From the definition of $u_{q+1}$ in \eqref{eq:u:q+1} and the identification of error terms immediately below \eqref{eq:u:q+1}, we see that
\begin{equation}
    \fint \left( \partial_t + u_\ell\cdot\nabla \right) \left( w_{q+1} - \alpha^2 \Delta w_{q+1} \right) = \fint \partial_t  \left( u_{q+1} - \alpha^2 \Delta u_{q+1} \right) \, . \notag
\end{equation}
Since $u_{q+1}$ was constructed as an explicit sum of curls of vector fields, cf. \eqref{eq:full:perturbation}, the right-hand side in the above equality has zero mean, and thus the left-hand side does as well.  Therefore we may ignore the terms arising from the mean in \eqref{eq:inverse:div} as usual, concluding the analysis of the transport error.

\subsection{Parameters and inequalities}\label{ss.parameters}
We need to choose values for the parameters
\begin{align}
\beta \, , \quad b \, , \quad \Gamma \, , \quad \gamma \, ,  \quad \ell \, , \quad \Ndec \,, \quad \dpot \, , \quad \calc_\RR \, , \quad a \,  \notag
\end{align}
such that all the inequalities throughout the paper are satisfied.  We chosen the parameters in the following order, and according to the following methodology:
\begin{enumerate}
    \item[(1)] Let $\beta'=1$ be an auxiliary parameter which we use temporarily to simplify the arithmetic in a number of inequalities.  After demonstrating that each inequality is strict, we shall later choose a value of $\beta>1$ which preserves each inequality.
    \item[(2)] Set $\Gamma=\sfrac{1}{2}$.
    \item[(3)] Choose $b>600$.
    \item[(4)] Choose $\Ndec$ such that $64(\Ndec+4) < b \Ndec(1-\Gamma-\sfrac{\Gamma}{b})$.  This is possible since $\Gamma=\sfrac{1}{2}$ and $b>600$.
    \item[(5)] Choose $\dpot$ such that so that $\dpot(40-b+\Gamma(b-1)) < - b + \Gamma(b-1) - 32$.  This is possible since $b>600$ and $\Gamma=\sfrac{1}{2}$.
    \item[(6)] Choose $\gamma\in(0,1)$ such that 
    \begin{enumerate}
        \item[(i)] $-2b+(1+\gamma)b+8 < - \gamma b^2$.
        \item[(ii)] $\gamma b - 8 < - \gamma b^2$.
        \item[(iii)] $232 - b + \Gamma(b-1) < - 2\gamma b^2$.
        \item[(iv)] $264 + \Gamma(1-b) < -\gamma b^2$.
        \item[(v)] $\Gamma(1-b) + \gamma b < -\gamma b^2$.
    \end{enumerate}
    This is possible since we may choose $\gamma$ very close to $0$.
    \item[(7)] Choose $\calc_\RR$ to ensure that \eqref{eq:a:bound:one} and \eqref{eq:wp} are satisfied; this is possible since $\calc_\RR$ appears as a prefactor in the inductive assumption in \eqref{eq:inductive} for $\RR_q$, and by extension in \eqref{eq:def:rho}.
\end{enumerate}
Given the choices of $\beta '=1$, $\Gamma$, $b$, $\Ndec$, $\dpot$, and $\gamma$ above, we may now verify a number of strict inequalities. We leave room in each inequality so that later we may choose $\beta>1$, and $a$ large enough to absorb implicit constants throughout the argument. We introduce the auxiliary parameter
\begin{equation}\notag
    \delta_q' = \lambda_q^{-2\beta'} = \lambda_q^{-2} \, 
\end{equation}
in order to simplify the arithmetic, and we show at the end that one can in fact substitute $\beta>1$ for $\beta'=1$ while preserving all the inequalities. For each inequality, we write out everything in terms of $\lambda_q$. We remind the reader that, per definitions \eqref{def:freq:amp} and \eqref{eq:ell:def}, 
\begin{equation}\label{eq:def:redux}
  \lambda_q = a^{b^q} \, , \qquad \delta_q = \lambda_q^{-2\beta} \, , \qquad \ell=\lambda_q^{-8} \, .
\end{equation}
\begin{enumerate}
    \item\label{ineq:one}  From the choice of $\ell=\lambda_q^{-8}$, and a sufficiently large choice of $a$ so that $\lambda_q^{-1}\leq \alpha^2 \calc_\RR$, we have that
    $$\ell^2\lambda_q^{9} < \alpha^2 \calc_\RR \delta'_{q+2} \lambda_{q+2}^2 \impliedby -16 + 10 < -2\beta' b^2 + 2b^2 \, .  $$
    \item\label{ineq:two} From the temporary choice of $\beta'=1$, we have that
    $$\ell < \lambda_{q+1}^2\delta'_{q+1} \iff -8 < 2b^2 - 2\beta' b^2 \, .$$
    \item\label{ineq:three} From the choices $\ell=\lambda_q^{-8}$ and $b>600$, we have that
    $$\ell^{-12}< \lambda_{q+1} \iff 96 < b \, . $$
    \item\label{ineq:four}  From the choice of $\Ndec$, we have that
    $$\ell^{-8(\Ndec+4)} < \lambda_{q+1}^{\Ndec(1-\Gamma-\sfrac{\Gamma}{b})} \iff 64(\Ndec +4) < b\Ndec(1-\Gamma-\sfrac{\Gamma}{b}) \, . $$
    \item\label{ineq:five} From $b>600$ and $\Gamma=\sfrac{1}{2}$, we have that
    $$r_q^{-1} < \lambda_{q+1} \, \iff \Gamma(b-1) < b \, .$$
    \item\label{ineq:six}  - From the temporary choice of $\beta'=1$ and the choice of $\gamma$, we have that
    $$\delta'_{q+1}\lambda_{q+1}^{2+\gamma}(\ell\lambda_{q+1})^{-1} < \lambda_{q+2}^{2-\gamma}\delta'_{q+2} \iff -2\beta' b + (1+\gamma)b + 8 < -2\beta' b^2 + (2-\gamma)b^2 \, .$$
    \item\label{ineq:seven} From the temporary choice of $\beta'=1$ and the choice of $\gamma$, we have that 
    $$\delta'_{q+1}\lambda_{q+1}^{2+\gamma}\ell < \lambda_{q+2}^{2-\gamma}\delta'_{q+2} \iff -2\beta' b + (2+\gamma)b - 8 < -2\beta' b^2 + (2-\gamma)b^2 \, . $$
    \item\label{ineq:seven:prime} From the temporary choice of $\beta'=1$, the choice of $b$, and the choice of $\gamma$, we have that
    $$ \delta'_{q+1} r_q \lambda_{q+1}^{2+\gamma} < \delta'_{q+2} \lambda_{q+2}^{2-\gamma} \iff -2\beta' b + 2b + \Gamma(1-b) + \gamma b < - 2\beta' b^2 + 2b^2 -\gamma b^2 \, . $$
    \item\label{ineq:eight} From the choice of $\dpot$, we have that
    $$\left(\frac{\ell^{-5}}{\lambda_{q+1}r_q}\right)^\dpot \ell^{-4} < \left(\lambda_{q+1}r_q\right)^{-1} \iff \dpot(40-b+\Gamma(b-1)) < -b + \Gamma(b-1) - 32 \, . $$
    \item\label{ineq:nine} As in the last inequality, the choice of $\dpot$ ensures that
    $$\left(\frac{\ell^{-5}}{\lambda_{q+1}}\right)^\dpot\ell^{-20} < \lambda_{q+1}^{-1} \iff \dpot(40-b)+160 < -b \, .$$ 
    \item\label{ineq:ten} From the temporary choice $\beta'=1$, the inequality $b>600$, and the choice of $\gamma$, we have that
    $$ \ell^{-9} \ell^{-20} \left(\lambda_{q+1}r_q\right)^{-1} < \delta'_{q+2} \lambda_{q+2}^{2-2\gamma} \iff 232-b+\Gamma(b-1) < -2\beta' b^2 + (2-2\gamma)b^2\, . $$
    \item\label{ineq:eleven} From the temporary choice of $\beta'=1$, the choice $\Gamma=\sfrac{1}{2}$, and the choice of $b>600$, we have that
    $$\ell^{-33} r_q < \delta'_{q+2}\lambda_{q+2}^{2-\gamma} \iff 264 + \Gamma(1-b) < -2\beta' b^2 + (2-\gamma)b^2 \, .$$
\end{enumerate}

Now observe that each of the inequalities above is strict, and all quantities are continuous with respect to $\beta'$.  Therefore we may choose $\beta>1$ so that all inequalities \emph{remain} strict after substitution of $\beta>1$ for $\beta'=1$.  At this point, we may now choose $a$ large enough so that all implicit constants throughout the paper do not exceed the extra room still present in each inequality.

\begin{remark}\label{rem:2d:two}
Aside from estimating the intersection of Mikado flows which belong to the same Lagrangian coordinate system by Proposition~\ref{prop:weak:decoupling}, the only significant change required to apply our arguments to the 2D Euler-$\alpha$ equations is the fact that the intermittency gain from $L^2$ to $L^1$ for intermittent Mikado flows is only of order $r_q^{\sfrac{1}{2}}$ instead of $r_q$. When choosing parameters, this only affects the Nash and transport error estimates, which held from inequality~\eqref{ineq:eleven}. Substituting $\sfrac{1}{2}\cdot\Gamma(1-b)$ for $\Gamma(1-b)$ necessitates a larger choice of $b$, but the rest of the parameter choices may be made identically, and so we achieve analogous results.
\end{remark}

\appendix

\section{Proof of Lemma \ref{lem:energy:conservation}}
\label{app:energy:conservation}
In this section we provide a proof of Lemma \ref{lem:energy:conservation}.

\begin{proof}
 We follow the proof from \cite{CET94}. Mollifying \eqref{eq:alpha:euler:transport} in space with a standard radial, compactly supported mollifier $\varphi_{\varepsilon}$, {assuming differentiability in time,} and integrating in space against $u_i^{\varepsilon}=u \ast \varphi_\varepsilon$ gives
\begin{align}
\label{eq:mollifed:energy}
  &\frac{1}{2} \frac{d}{dt} \left(\| u_i^{\varepsilon} \|_{L^2}^2 + \alpha^2 \| \nabla u_i^{\varepsilon}\|_{L^2}^2 \right)  \notag\\
  &= \int_{\T^n } (u_j u_i)^{\varepsilon} \partial_j u_i^{\varepsilon} + \alpha^2 (u_j \partial_k u_i )^{\varepsilon} \partial_j \partial_k u_i^{\varepsilon} + \alpha^2 (\partial_k u_j \partial_k u_i)^{\varepsilon} \partial_ju_i^{\varepsilon} - \alpha^2 (\partial_k u_j \partial_i u_j)^{\varepsilon}  \partial_k u_i^{\varepsilon} \, dx \, 
\end{align}
where we have used Remark \ref{rem:weak}.

We wish to show that the right-hand side vanishes when $\varepsilon$ is sent to 0. 
We recall some standard estimates for mean-zero $f \in B_{3, \infty}^{s}$ where ${1<s<2}$; see for example \cite{CET94,bcd}.
\begin{subequations}
\label{eq:Besov:estimates}
\begin{align}
    &\| \nabla f(\cdot ) - \nabla f(\cdot  - y)\|_{L^3} \lesssim |y|^{s -1} \| \nabla f \|_{B_{3 , \infty}^{s - 1}}\\
   & \| f - f^{\varepsilon} \|_{L^3} \lesssim \varepsilon \| f \|_{B_{3, \infty}^{s}}  \\ 
    & \| \nabla f - \nabla f
    ^{\varepsilon} \|_{L^3} \lesssim \varepsilon^{s -1} \| \nabla f \|_{B_{3, \infty}^{s-1}} \\ 
    &\| \grad^2 f^{\varepsilon} \|_{L^3} \lesssim \varepsilon^{s-2} \| \nabla f \|_{B_{3, \infty}^{s - 1}} \\ 
    &{\left\| \nabla f \right\|_{L^3} \lesssim \left\| \nabla f \right\|_{B^{s-1}_{3,\infty}}}\\ 
    & \left\| \nabla f \right\|_{B^{s-1}_{3,\infty}} \lesssim \left\| f \right\|_{B^{s}_{3,\infty}}\\ 
    &{\left\| f^\varepsilon \right\|_{B^{s}_{3,\infty}}\lesssim \left\|  f \right\|_{B^{s}_{3,\infty}}} \, . 
\end{align}
\end{subequations}
Finally, recall the double commutator identity from \cite{CET94}:
\begin{equation}
    \label{eq:double:commutator:estimate}
    (fg)^{\varepsilon} = f^{\varepsilon}g^{\varepsilon} - (f - f^{\varepsilon})(g -g^{\varepsilon}) + r_{\varepsilon}(f,g) \, ,
\end{equation}
where $f^{\varepsilon}(x) = f * \varphi_{\varepsilon}(x)$ and
\begin{equation}
    \label{eq:error}
    r_{\varepsilon}(f,g) = \int_{\T^n} \varphi_{\varepsilon}(y) (f(x - y) - f(x))(g(x - y) - g(x)) \, dy \, .
\end{equation}
Then the right-hand side of \eqref{eq:mollifed:energy} can be written as
\begin{align}
     &\int_{\T^n } (u_j u_i)^{\varepsilon} \partial_j u_i^{\varepsilon} + \alpha^2 (u_j \partial_k u_i )^{\varepsilon} \partial_j \partial_k u_i^{\varepsilon} + \alpha^2 (\partial_k u_j \partial_k u_i)^{\varepsilon} \partial_ju_i^{\varepsilon} - \alpha^2 (\partial_k u_j \partial_i u_j)^{\varepsilon}  \partial_k u_i^{\varepsilon} dx 
    \notag\\
    &= \int_{\T^n} u_j^{\varepsilon}u_i^{\varepsilon} \partial_j u_i^{\varepsilon} + \alpha^2\left[ u_j^{\varepsilon}\partial_ku_i^{\varepsilon}  \partial_j \partial_k u_i^{\varepsilon}   + \partial_ku_j^{\varepsilon} \partial_k u_i^{\varepsilon} \partial_j u_i^{\varepsilon} - \partial_ku_j^{\varepsilon} \partial_i u_j^{\varepsilon} \partial_k u_i^{\varepsilon} \right] dx 
    \notag\\ 
    & - \int_{\T^n}  (u_j - u_j^{\varepsilon})(u_i - u_i^{\varepsilon}) \partial_j u_i^{\varepsilon} dx 
    \notag\\
   &- \alpha^2\int_{\T^n} \left[( u_j -  u_j^{\varepsilon})(\partial_k u_i - \partial_k u_i^{ \varepsilon})\partial_j \partial_ku_i^{\varepsilon} + (\partial_k u_j  - \partial_k u_j^{\varepsilon}) (\partial_k u_i - \partial_k u_i^{\varepsilon}) \partial_j u_i^{\varepsilon} - (\partial_k u_j - \partial_k u_j^{\varepsilon} )(\partial_i u_j - \partial_i u_j^{\varepsilon}) \partial_k u_i^{\varepsilon}   \right] dx \notag\\
   &+  \int_{\T^n} r_{\varepsilon}(u_j, u_i) \partial_j u_i^{\varepsilon} + \alpha^2\left[ r_{\varepsilon } (u_j, \partial_ku_i ) \partial_j \partial_k u_i^{\varepsilon}   + r_{\varepsilon} (\partial_ku_j ,\partial_k u_i) \partial_j u_i^{\varepsilon} - r_{\varepsilon}(\partial_ku_j , \partial_i u_j) \partial_k u_i^{\varepsilon} \right] dx \, . \label{eq:comm:messy}
    \end{align}
We have 
\begin{equation*}
    \int_{\T^n} u_j^{\varepsilon} u_i^{\varepsilon} \partial_j u_i^{\varepsilon} dx = \int_{\T^n} u_j^{\varepsilon} \partial_k u_i^{\varepsilon} \partial_j \partial_k u_i^{\varepsilon} = 0
\end{equation*}
using integration by parts and that $\Div u = 0$. Furthermore, the equality
\begin{equation*}
       \partial_k u_j^{\varepsilon} \partial_ku_i^{\varepsilon} \partial_j  u_i^{\varepsilon} = \partial_k u_j^{\varepsilon} \partial_iu_j^{\varepsilon} \partial_k u_i^{\varepsilon}
\end{equation*} 
follows from switching the roles of $i$ and $j$. Therefore the first four terms after the equal sign on the right-hand side of \eqref{eq:comm:messy} vanish, and we can bound the remaining terms by  
   \begin{align*}
     &\lesssim \| u - u^{\varepsilon}\|_{L^3}^2 \| \nabla u^{\varepsilon} \|_{L^3} + \| u - u^{\varepsilon}\|_{L^3} \| \nabla u - \nabla u^{\varepsilon} \|_{L^3} \| \nabla^2 u^{\varepsilon} \|_{L^3} + \| \nabla u - \nabla u^{\varepsilon} \|_{L^3}^2 \| \nabla u^{\varepsilon} \|_{L^3}  
      \\
      &+ \| r_{\varepsilon} (u, u) \|_{L^{\frac{3}{2}}} \| \nabla u^{\varepsilon} \|_{L^3} + \|r_{\varepsilon} (u, \nabla u)\|_{L^{\frac{3}{2}}} \|  \nabla^2 u^{\varepsilon}\|_{L^3} + \| r_{\varepsilon}( \nabla u, \nabla u) \|_{L^{\frac{3}{2}}} \| \nabla u^{\varepsilon} \|_{L^3
      }
      \\
      &\lesssim \varepsilon^{{2}} \| u \|_{B_{3, \infty}^{s}  }^3 + \varepsilon^{ {2(s-1)}} \| u \|_{B_{3 ,\infty}^{s}}^3.
   \end{align*}
  Thus, taking $s > 1$ shows that the right-hand side of \eqref{eq:mollifed:energy} converges to zero when $\varepsilon\rightarrow 0$.  
  {Concluding the proof in the case of continuity in time may be done analogously as for the classical Euler equations, and we omit further details}.
\end{proof}

\section{Proof of Lemma~\ref{lem:linear:algebra}}
\label{app:linear:algebra}
In this section we provide a proof of Lemma \ref{lem:linear:algebra}.
\begin{proof}
We first construct $\mathcal{K}_0$ and $c_i^0$ by hand while allowing $\left\| \RR \right\|_{0}\leq 1$, afterwards constructing $\mathcal{K}_n$ for $1\leq n \leq N$ and choosing $\varepsilon$.
Consider a symmetric traceless matrix $\RR$, which without loss of generality may be written as
\begin{equation}\label{eq:matrix}
\RR = \begin{pmatrix}
a & c & d\\
c & b & e\\
d & e & -a-b 
\end{pmatrix} \, .
\end{equation}
Notice that the set of such matrices is 5-dimensional, and combined with the condition on the sum of $(c_i^0)^2$ in \eqref{eq:to:check} will require a set of at least six vectors. The extra three vectors will ensure that the coefficients are all positive. Ignoring for now the upper subscript $0$ on the vectors $k_i^0\in\mathcal{K}_0$, we set 
\begin{equation}\notag
    k_1 = e_1 \, , \qquad k_2 = e_2 \, , \qquad k_3 = e_3 \, .
\end{equation}
For ease of notation, let us define
\begin{equation}\notag
    f(k_i) = 3k_i\otimes k_i - \textnormal{Id} \, .
\end{equation}
Then it is clear that $f(k_1)$, $f(k_2)$, and $f(k_3)$ are symmetric traceless matrices which only contain entries on the diagonal; specifically, we have
\begin{equation}\label{eq:mixed:matrices}
f(k_1) = \begin{pmatrix}
2 & 0 & 0\\
0 & -1 & 0\\
0 & 0 & -1 
\end{pmatrix}\, , \qquad 
f(k_2) =\begin{pmatrix}
-1 & 0 & 0\\
0 & 2 & 0\\
0 & 0 & -1 
\end{pmatrix}\, , \qquad 
f(k_3) = \begin{pmatrix}
-1 & 0 & 0\\
0 & -1 & 0\\
0 & 0 & 2 
\end{pmatrix}\, . \qquad 
\end{equation}
We shall use $f(k_1)$ and $f(k_2)$ to engender the entries of the matrix \emph{on} the diagonal in \eqref{eq:matrix}, while we shall use a ``balanced" sum of the form $c(f(k_1)+f(k_2)+f(k_3))$ to ensure the second condition in \eqref{eq:to:check}. We further set 
\begin{align}
        k_4&= \frac{3e_1+4e_2}{5}\, , \qquad k_5=\frac{3e_1+4e_3}{5}\, ,\qquad  k_6=\frac{3e_2+4e_3}{5} \notag\\
        k_7&= \frac{3e_1-4e_2}{5}\, , \qquad k_8=\frac{3e_1-4e_3}{5}\, ,\qquad  k_9=\frac{3e_2-4e_3}{5}\, . \notag
\end{align}
Then 
\begin{align}\label{eq:mixed:up:matrices}
f(k_4) &= \frac{1}{25}\begin{pmatrix}
2 & 36 & 0\\
36 & 23 & 0\\
0 & 0 & -25 
\end{pmatrix}\, , \qquad 
f(k_5) = \frac{1}{25}\begin{pmatrix}
2 & 0 & 36\\
0 & -25 & 0\\
36 & 0 & 23 
\end{pmatrix}\, , \qquad 
f(k_6) = \frac{1}{25}\begin{pmatrix}
-25 & 0 & 0\\
0 & 2 & 36\\
0 & 36 & 23 
\end{pmatrix}\\
f(k_7) &= \frac{1}{25}\begin{pmatrix}
2 & -36 & 0\\
-36 & 23 & 0\\
0 & 0 & -25 
\end{pmatrix}\, , \quad 
f(k_8) = \frac{1}{25}\begin{pmatrix}
2 & 0 & -36\\
0 & -25 & 0\\
-36 & 0 & 23 
\end{pmatrix}\, , \quad 
f(k_9) = \frac{1}{25}\begin{pmatrix}
-25 & 0 & 0\\
0 & 2 & -36\\
0 & -36 & 23 
\end{pmatrix}\, .  \notag
\end{align}
We shall use $f(k_4)$, $f(k_5)$, $f(k_6)$, $f(k_7)$, $f(k_8)$, and $f(k_9)$ to engender the entries of \eqref{eq:matrix} \emph{off} the diagonal. 

It is simple to check that the set $\{f(k_1),f(k_2),f(k_4),f(k_5),f(k_6)\}$ is a linearly independent set in the space of symmetric traceless matrices.  Therefore, there exist smooth functions $\{\tilde c_1, \tilde c_2, \tilde c_4, \tilde c_5, \tilde c_6\}$ given by the solutions of a linear system of equations such that for all symmetric traceless matrices $\RR$ satisfying $\left\| \RR \right\|_0\leq 1$,
\begin{equation}\label{eq:basic:matrices}
    \tilde c_1 f(k_1) + \tilde c_2 f(k_2) + \tilde c_4 f(k_4) + \tilde c_5 f(k_5) + \tilde c_6 f(k_6) = \RR \, .
\end{equation}
Unfortunately the functions $\tilde c_i$ are not strictly positive.  Let us define the auxiliary parameter
$$  c_0 = \max_{\substack{\left\| \RR \right\|_0\leq 1, i=4,5,6}} |\tilde c_i(\RR)| \, . $$
Then from \eqref{eq:basic:matrices} and \eqref{eq:mixed:up:matrices}, we have that 
\begin{align}
    \left(1 -\delta_{ij} \right) &\bigg{(} \tilde c_1 f(k_1) + \tilde c_2 f(k_2) + \tilde c_4 f(k_4) + \tilde c_5 f(k_5) + \tilde c_6 f(k_6) \notag\\
    &\qquad + 2c_0 \left( f(k_4) + f(k_5) + f(k_6) + f(k_7) + f(k_8) + f(k_9) \right) \bigg{)}^{ij} = \left(1 -\delta_{ij} \right) \RR^{ij}\, ; \notag
\end{align}
that is, off the diagonal, the matrix on the left hand side of the equation is equal to $\RR$. The advantage now is that the coefficients on $f(k_i)$ for $4\leq i \leq 9$ are all positive. In order to ensure equality \emph{on} the diagonal, we may replace the coefficients $\tilde c_1$ and $\tilde c_2$ on $f(k_1)$ and $f(k_2)$ with new coefficients $\dot c_1$ and $\dot c_2$ such that
\begin{align}
      &\bigg{(}\dot c_1 f(k_1) + \dot c_2 f(k_2) + \tilde c_4 f(k_4) + \tilde c_5 f(k_5) + \tilde c_6 f(k_6) \notag\\
    &\qquad\qquad + 2c_0 \left( f(k_4) + f(k_5) + f(k_6) + f(k_7) + f(k_8) + f(k_9) \right) \bigg{)}^{ij}  =  \RR^{ij}\,  \label{eq:almost}
\end{align}
for all $1\leq i,j\leq 3$. Since the coefficients $\dot c_1$ and $\dot c_2$ are not necessarily positive, we set 
$$  \dot c_0 = \max_{\norm{\RR}_0\leq 1,i=1,2} \left| \dot c_i (\RR) \right| \, .  $$
Then from \eqref{eq:mixed:matrices} and \eqref{eq:almost}, we have that
\begin{align}
      &\bigg{(}\dot c_1  f(k_1) + \dot c_2 f(k_2) + 2\dot c_0 \left( f(k_1)+f(k_2)+f(k_3) \right) \notag\\
      &\qquad + \tilde c_4 f(k_4) + \tilde c_5 f(k_5) + \tilde c_6 f(k_6) + 2c_0 \left( f(k_4) + f(k_5) + f(k_6) + f(k_7) + f(k_8) + f(k_9) \right) \bigg{)}^{ij}  =  \RR^{ij}\,  \notag
\end{align}
Aggregating coefficients for each of the nine tensors on the left-hand side, we see that each is strictly positive.  We still have to ensure the second condition in \eqref{eq:to:check}; however, this is easily achieved by replacing the constant $\dot c_0$ with a non-negative, bounded function $c_0(\RR):B_1(0)\rightarrow [0,\infty)$ which depends on $\RR$ and imposes that the sum of the coefficients is equal for all $\RR\in B_1(0)$. Thus we have achieved everything in \eqref{eq:to:check} for $n=0$. Note that the value of $\calc_{\rm sum}$ could in principle be computed explicitly and depends on the choice of the function $c_0(\RR)$.

To construct $\mathcal{K}_n$ for $n=1$, choose a rational rotation matrix $O_1$ with $0 < \left\| O_1 - \textnormal{Id} \right\|_0 \ll 1$, and set $\mathcal{K}_1= O_1 \mathcal{K}_0$.  Plugging the rotated vectors $O_1 k_i^0$ into \eqref{eq:to:check}, we find that the effect of $O_1$ is that
\begin{equation}
    \sum_{i=1}^9 \left(c_i^0\right)^2\left(\RR\right)\left( 3 k_i^0 \otimes k_i^0 - \textnormal{Id} \right) = \RR \quad \rightarrow \quad \sum_{i=1}^9 \left(c_i^0\right)^2\left(\RR\right)O_1\left( 3 k_i^0 \otimes k_i^0 - \textnormal{Id} \right)O_1^T = O_1 \RR O_1^T \, . \notag
\end{equation}
Using the equalities $\textnormal{tr}(AB)=\textnormal{tr}(BA)$ and $O^T=O^{-1}$, we have that $O_1 \RR O_1^T$ is still traceless and symmetric. Define 
$$  c^1_i\left(\RR\right) = c^0_i\left(O_1^T \RR O_1 \right) \,   $$
for $\left\|\RR\right\|$ sufficiently small so that $c^0_i\left(O_1 \RR O_1^T\right)$ is well-defined and strictly positive for each $i$. Then this specifies a value $\varepsilon_1$ such that for all $\left\| \RR \right\|_0 \leq \varepsilon_1$,
\begin{align}
   \sum_{i=1}^9 \left(c_i^1\right)^2\left(\RR\right)\left( 3 k_i^1 \otimes k_i^1 - \textnormal{Id} \right) &= \sum_{i=1}^9 \left(c^0_i\right)^2 \left( O_1^T \RR O_1\right) O_1 \left( 3 k_i^0 \otimes k_i^0 - \textnormal{Id} \right) O_1^T \notag\\
   &= \RR \, , \notag
\end{align}
showing that the first equality in \eqref{eq:to:check} is satisfied for $n=1$.  The second equality in \eqref{eq:to:check} comes immediately from the construction of the $c_i^0$'s and $c^1_i$'s.  In addition, $\mathcal{K}_0$ and $\mathcal{K}_1$ will be disjoint if $0 < \left\| O_1 - \textnormal{Id} \right\|_0\ll 1$.  Iterating this procedure and taking the minimum value of $\varepsilon_n$ concludes the proof.
\end{proof}


\subsection*{Acknowledgements}
The authors thank Vlad Vicol and Steve Shkoller for stimulating conversations. RB was partially supported by the NSF under Grant No. DMS-1911413 and No. DMS-195435.  MN was partially supported by the NSF under Grant No. DMS-1928930 while participating in a program hosted by the Mathematical Sciences Research Institute during the spring 2021 semester, and by the NSF under Grant No. DMS-1926686 while a member at the Institute for Advanced Study. 

\bibliographystyle{abbrv}
\bibliography{cite}

\begin{bibdiv}
\begin{biblist}

\bib{Arnold66}{article}{
      author={Arnold, Vladimir},
       title={On the differential geometry of infinite dimensional lie groups
  and its applications to the hydrodynamics of perfect fluids},
    language={fr},
        date={1966},
     journal={Annales of the Fourier Institute},
      volume={16},
      number={1},
       pages={319\ndash 361},
         url={http://www.numdam.org/articles/10.5802/aif.233/},
}

\bib{bcd}{book}{
      author={Bahouri, Hajer},
      author={Chemin, Jean-Yves},
      author={Danchin, Raphaël},
       title={Fourier analysis and nonlinear partial differential equations},
   publisher={Springer Berlin Heidelberg},
        date={2011},
         url={https://doi.org/10.1007/978-3-642-16830-7},
}

\bib{BBV19}{article}{
      author={Beekie, R.},
      author={Buckmaster, T.},
      author={Vicol, V.},
       title={Weak solutions of ideal mhd which do not conserve magnetic
  helicity},
        date={2020},
     journal={Annals of PDE},
      volume={6},
      number={1},
       pages={Paper No. 1, 40 pp.},
}

\bib{brue2020positive}{article}{
      author={Bru\'{e}, Elia},
      author={Colombo, Maria},
      author={De~Lellis, Camillo},
       title={Positive solutions of transport equations and classical
  nonuniqueness of characteristic curves},
        date={2021},
        ISSN={0003-9527},
     journal={Arch. Ration. Mech. Anal.},
      volume={240},
      number={2},
       pages={1055\ndash 1090},
         url={https://doi.org/10.1007/s00205-021-01628-5},
}

\bib{BCV18}{article}{
      author={Buckmaster, T.},
      author={Colombo, M.},
      author={Vicol, V.},
       title={Wild solutions of the {N}avier-{S}tokes equations whose singular
  sets in time have {H}ausdorff dimension strictly less than $1$},
        date={2018},
     journal={arXiv:1809.00600},
}

\bib{BDLISZ15}{article}{
      author={Buckmaster, T.},
      author={De~Lellis, C.},
      author={Isett, P.},
      author={Sz{{\'e}}kelyhidi, L., Jr.},
       title={Anomalous dissipation for $1/5$-{H}\"older {E}uler flows},
        date={2015},
     journal={Ann. of Math.},
      volume={182},
      number={1},
       pages={127\ndash 172},
}

\bib{BDLSV17}{article}{
      author={Buckmaster, T.},
      author={De~Lellis, C.},
      author={Sz{\'{e}}kelyhidi~Jr., L.},
      author={Vicol, V.},
       title={Onsager{\textquotesingle}s conjecture for admissible weak
  solutions},
        date={2018-07},
     journal={Comm. Pure Appl. Math.},
      volume={72},
      number={2},
       pages={229\ndash 274},
         url={https://doi.org/10.1002/cpa.21781},
}

\bib{BSV16}{article}{
      author={Buckmaster, T.},
      author={Shkoller, S.},
      author={Vicol, V.},
       title={Nonuniqueness of weak solutions to the {SQG} equation},
        date={2019},
     journal={Comm. Pure Appl. Math.},
      volume={72},
      number={9},
       pages={1809\ndash 1874},
}

\bib{BV_EMS19}{article}{
      author={Buckmaster, T.},
      author={Vicol, V.},
       title={Convex integration and phenomenologies in turbulence},
        date={2019},
     journal={EMS Surv. Math. Sci.},
      volume={6},
      number={1},
       pages={173\ndash 263},
      eprint={1901.09023},
}

\bib{BV19}{article}{
      author={Buckmaster, T.},
      author={Vicol, V.},
       title={Nonuniqueness of weak solutions to the {N}avier-{S}tokes
  equation},
        date={2019},
     journal={Ann. of Math.},
      volume={189},
      number={1},
       pages={101\ndash 144},
}

\bib{bmnv21}{article}{
      author={Buckmaster, Tristan},
      author={Masmoudi, Nader},
      author={Novack, Matthew},
      author={Vicol, Vlad},
       title={{Non-conservative $H^{\frac 12-}$ weak solutions of the
  incompressible 3D Euler equations}},
        date={2021},
     journal={arxiv:2101.09278},
}

\bib{BusuiocIftimieLFNL16}{article}{
      author={Busuioc, A.~V.},
      author={Iftimie, D.},
      author={Lopes~Filho, M.~C.},
      author={Nussenzveig~Lopes, H.~J.},
       title={Uniform time of existence for the alpha {E}uler equations},
        date={2016},
        ISSN={0022-1236},
     journal={J. Funct. Anal.},
      volume={271},
      number={5},
       pages={1341\ndash 1375},
         url={https://doi.org/10.1016/j.jfa.2016.06.006},
      review={\MR{3522011}},
}

\bib{BusuiocIftimie17}{article}{
      author={Busuioc, Adriana~Valentina},
      author={Iftimie, Drago\c{s}},
       title={Weak solutions for the {$\alpha$}-{E}uler equations and
  convergence to {E}uler},
        date={2017},
        ISSN={0951-7715},
     journal={Nonlinearity},
      volume={30},
      number={12},
       pages={4534\ndash 4557},
         url={https://doi.org/10.1088/1361-6544/aa8982},
      review={\MR{3734146}},
}

\bib{BusuiocIftimieLFNL12}{article}{
      author={Busuioc, A.V.},
      author={Iftimie, D.},
      author={{Lopes Filho}, M.C.},
      author={{Nussenzveig Lopes}, H.J.},
       title={Incompressible euler as a limit of complex fluid models with
  navier boundary conditions},
        date={2012},
        ISSN={0022-0396},
     journal={Journal of Differential Equations},
      volume={252},
      number={1},
       pages={624\ndash 640},
  url={https://www.sciencedirect.com/science/article/pii/S0022039611002324},
}

\bib{busuioc99}{article}{
      author={Busuioc, Valentina},
       title={On second grade fluids with vanishing viscosity},
        date={1999},
        ISSN={0764-4442},
     journal={C. R. Acad. Sci. Paris S\'{e}r. I Math.},
      volume={328},
      number={12},
       pages={1241\ndash 1246},
         url={https://doi.org/10.1016/S0764-4442(99)80447-9},
}

\bib{ChenFoiasHolmOlsonTiti98}{article}{
      author={Chen, Shiyi},
      author={Foias, Ciprian},
      author={Holm, Darryl},
      author={Olson, Eric},
      author={Titi, Edriss},
      author={Wynne, Shannon},
       title={Camassa-holm equations as a closure model for turbulent channel
  and pipe flow},
        date={199812},
     journal={Physical Review Letters},
      volume={81},
}

\bib{ChenHolmMargolinZhang99}{incollection}{
      author={Chen, Shiyi},
      author={Holm, Darryl~D.},
      author={Margolin, Len~G.},
      author={Zhang, Raoyang},
       title={Direct numerical simulations of the {N}avier-{S}tokes alpha
  model},
        date={1999},
      volume={133},
       pages={66\ndash 83},
         url={https://doi.org/10.1016/S0167-2789(99)00099-8},
        note={Predictability: quantifying uncertainty in models of complex
  phenomena (Los Alamos, NM, 1998)},
      review={\MR{1721140}},
}

\bib{cheskidov2020sharp}{article}{
      author={Cheskidov, A.},
      author={Luo, X.},
       title={Sharp nonuniqueness for the {N}avier-{S}tokes equations},
        date={2020},
     journal={arXiv:2009.06596},
}

\bib{cheskidov2020nonuniqueness}{article}{
      author={Cheskidov, Alexey},
      author={Luo, Xiaoyutao},
       title={Nonuniqueness of weak solutions for the transport equation at
  critical space regularity},
        date={2021},
        ISSN={2524-5317},
     journal={Ann. PDE},
      volume={7},
      number={1},
       pages={Paper No. 2, 45},
         url={https://doi.org/10.1007/s40818-020-00091-x},
      review={\MR{4199851}},
}

\bib{CET94}{article}{
      author={Constantin, Peter},
      author={E, Weinan},
      author={Titi, Edriss~S.},
       title={Onsager's conjecture on the energy conservation for solutions of
  {E}uler's equation},
        date={1994},
        ISSN={0010-3616},
     journal={Comm. Math. Phys.},
      volume={165},
      number={1},
       pages={207\ndash 209},
         url={http://projecteuclid.org/euclid.cmp/1104271041},
      review={\MR{1298949}},
}

\bib{Dai18nonunique}{article}{
      author={Dai, Mimi},
       title={{Nonunique Weak Solutions in Leray--Hopf Class for the
  Three-Dimensional Hall-MHD System}},
        date={2021},
     journal={SIAM Journal on Mathematical Analysis},
}

\bib{DaneriSzekelyhidi17}{article}{
      author={Daneri, S},
      author={Sz{{\'e}}kelyhidi, L., Jr.},
       title={Non-uniqueness and h-principle for {H}\"older-continuous weak
  solutions of the {E}uler equations},
        date={2017},
     journal={Arch. Rational Mech. Anal.},
      volume={224},
      number={2},
       pages={471\ndash 514},
}

\bib{DLSZ19}{article}{
      author={De~Lellis, Camillo},
      author={Sz\'{e}kelyhidi, L\'{a}szl\'{o}, Jr.},
       title={On turbulence and geometry: from {N}ash to {O}nsager},
        date={2019},
        ISSN={0002-9920},
     journal={Notices Amer. Math. Soc.},
      volume={66},
      number={5},
       pages={677\ndash 685},
      review={\MR{3929468}},
}

\bib{DunnFosdick74}{article}{
      author={Dunn, J.~Ernest},
      author={Fosdick, Roger~L.},
       title={Thermodynamics, stability, and boundedness of fluids of
  complexity {$2$} and fluids of second grade},
        date={1974},
        ISSN={0003-9527},
     journal={Arch. Rational Mech. Anal.},
      volume={56},
       pages={191\ndash 252},
         url={https://doi.org/10.1007/BF00280970},
      review={\MR{351249}},
}

\bib{FoiasHolmTiti01}{article}{
      author={Foias, Ciprian},
      author={Holm, Darryl~D.},
      author={Titi, Edriss~S.},
       title={The {N}avier-{S}tokes-alpha model of fluid turbulence},
        date={2001},
        ISSN={0167-2789},
     journal={Phys. D},
      volume={152/153},
       pages={505\ndash 519},
         url={https://doi.org/10.1016/S0167-2789(01)00191-9},
        note={Advances in nonlinear mathematics and science},
      review={\MR{1837927}},
}

\bib{HolmMarsdenJerroldRatiu98a}{article}{
      author={Holm, Darryl},
      author={Marsden, Jerrold},
      author={Ratiu, Tudor},
       title={Euler-poincaré models of ideal fluids with nonlinear
  dispersion},
        date={199805},
     journal={Physical Review Letters - PHYS REV LETT},
      volume={80},
       pages={4173\ndash 4176},
}

\bib{HolmMarsdenJerroldRatiu98b}{article}{
      author={Holm, Darryl~D},
      author={Marsden, Jerrold~E},
      author={Ratiu, Tudor~S},
       title={The euler–poincaré equations and semidirect products with
  applications to continuum theories},
        date={1998},
        ISSN={0001-8708},
     journal={Advances in Mathematics},
      volume={137},
      number={1},
       pages={1\ndash 81},
  url={https://www.sciencedirect.com/science/article/pii/S0001870898917212},
}

\bib{HouLi06}{article}{
      author={Hou, Thomas~Y.},
      author={Li, Congming},
       title={On global well-posedness of the {L}agrangian averaged {E}uler
  equations},
        date={2006},
        ISSN={0036-1410},
     journal={SIAM J. Math. Anal.},
      volume={38},
      number={3},
       pages={782\ndash 794},
         url={https://doi.org/10.1137/050625783},
      review={\MR{2262942}},
}

\bib{IssetVicol15}{article}{
      author={Isett, P.},
      author={Vicol, V.},
       title={H\"older continuous solutions of active scalar equations},
        date={2015},
        ISSN={2199-2576},
     journal={Annals of PDE},
      volume={1},
      number={1},
       pages={1\ndash 77},
         url={http://dx.doi.org/10.1007/s40818-015-0002-0},
}

\bib{Isett12}{book}{
      author={Isett, Philip},
       title={H\"{o}lder continuous {E}uler flows in three dimensions with
  compact support in time},
      series={Annals of Mathematics Studies},
   publisher={Princeton University Press, Princeton, NJ},
        date={2017},
      volume={196},
        ISBN={978-0-691-17483-9},
         url={https://doi.org/10.1515/9781400885428},
      review={\MR{3616258}},
}

\bib{Isett2018}{article}{
      author={Isett, Philip},
       title={{A proof of Onsager{\textquotesingle}s conjecture}},
        date={2018},
     journal={Annals of Mathematics},
      volume={188},
      number={3},
       pages={871},
         url={https://doi.org/10.4007/annals.2018.188.3.4},
}

\bib{dlk20}{misc}{
      author={Lellis, Camillo~De},
      author={Kwon, Hyunju},
       title={On non-uniqueness of h\"{o}lder continuous globally dissipative
  euler flows},
        date={2020},
}

\bib{LinshizTiti10}{article}{
      author={Linshiz, Jasmine~S.},
      author={Titi, Edriss~S.},
       title={On the convergence rate of the {E}uler-{$\alpha$}, an inviscid
  second-grade complex fluid, model to the {E}uler equations},
        date={2010},
        ISSN={0022-4715},
     journal={J. Stat. Phys.},
      volume={138},
      number={1-3},
       pages={305\ndash 332},
         url={https://doi.org/10.1007/s10955-009-9916-9},
      review={\MR{2594899}},
}

\bib{LFNLTitiZang15}{article}{
      author={Lopes~Filho, Milton~C.},
      author={Nussenzveig~Lopes, Helena~J.},
      author={Titi, Edriss~S.},
      author={Zang, Aibin},
       title={Convergence of the 2{D} {E}uler-{$\alpha$} to {E}uler equations
  in the {D}irichlet case: indifference to boundary layers},
        date={2015},
        ISSN={0167-2789},
     journal={Phys. D},
      volume={292/293},
       pages={51\ndash 61},
         url={https://doi.org/10.1016/j.physd.2014.11.001},
      review={\MR{3286611}},
}

\bib{MarsdenRatiuShkoller00}{article}{
      author={Marsden, J.~E.},
      author={Ratiu, T.~S.},
      author={Shkoller, S.},
       title={The geometry and analysis of the averaged {E}uler equations and a
  new diffeomorphism group},
        date={2000},
        ISSN={1016-443X},
     journal={Geom. Funct. Anal.},
      volume={10},
      number={3},
       pages={582\ndash 599},
         url={https://doi.org/10.1007/PL00001631},
      review={\MR{1779614}},
}

\bib{ModenaSZ17}{article}{
      author={Modena, Stefano},
      author={Sz\'{e}kelyhidi, L\'{a}szl\'{o}, Jr.},
       title={Non-uniqueness for the transport equation with {S}obolev vector
  fields},
        date={2018},
        ISSN={2524-5317},
     journal={Ann. PDE},
      volume={4},
      number={2},
       pages={Paper No. 18, 38},
         url={https://doi.org/10.1007/s40818-018-0056-x},
      review={\MR{3884855}},
}

\bib{MohseniKosovicShkollerMarsden03}{article}{
      author={Mohseni, Kamran},
      author={Kosović, Branko},
      author={Shkoller, Steve},
      author={Marsden, Jerrold~E.},
       title={Numerical simulations of the lagrangian averaged navier–stokes
  equations for homogeneous isotropic turbulence},
        date={2003},
     journal={Physics of Fluids},
      volume={15},
      number={2},
       pages={524\ndash 544},
      eprint={https://doi.org/10.1063/1.1533069},
         url={https://doi.org/10.1063/1.1533069},
}

\bib{novack2020}{article}{
      author={Novack, Matthew},
       title={Nonuniqueness of weak solutions to the 3 dimensional
  quasi-geostrophic equations},
        date={2020},
        ISSN={0036-1410},
     journal={SIAM J. Math. Anal.},
      volume={52},
      number={4},
       pages={3301\ndash 3349},
         url={https://doi.org/10.1137/19M1281009},
      review={\MR{4126319}},
}

\bib{Onsager49}{article}{
      author={Onsager, L.},
       title={Statistical hydrodynamics},
        date={1949},
     journal={Nuovo Cimento (9)},
      volume={6},
      number={Supplemento, 2(Convegno Internazionale di Meccanica Statistica)},
       pages={279\ndash 287},
      review={\MR{0036116 (12,60f)}},
}

\bib{Shkoller98}{article}{
      author={Shkoller, Steve},
       title={Geometry and curvature of diffeomorphism groups with {$H^1$}
  metric and mean hydrodynamics},
        date={1998},
        ISSN={0022-1236},
     journal={J. Funct. Anal.},
      volume={160},
      number={1},
       pages={337\ndash 365},
         url={https://doi.org/10.1006/jfan.1998.3335},
      review={\MR{1658676}},
}

\bib{Shkoller00}{article}{
      author={Shkoller, Steve},
       title={Analysis on groups of diffeomorphisms of manifolds with boundary
  and the averaged motion of a fluid},
        date={2000},
        ISSN={0022-040X},
     journal={J. Differential Geom.},
      volume={55},
      number={1},
       pages={145\ndash 191},
         url={http://projecteuclid.org/euclid.jdg/1090340568},
      review={\MR{1849348}},
}

\end{biblist}
\end{bibdiv}

\end{document}